\documentclass[a4paper,11pt]{article}

\usepackage{color}
\usepackage[latin1]{inputenc}
\usepackage[T1]{fontenc}
\usepackage[normalem]{ulem}
\usepackage[english]{babel}
\usepackage{verbatim}
\usepackage{graphicx}
\usepackage{enumerate,enumitem}
\usepackage{amsmath,amssymb,amsfonts,amsthm,mathrsfs}
\usepackage{rotating,relsize}
\usepackage{float}
\usepackage{array}
\usepackage{MnSymbol}
\setlist[enumerate]{leftmargin=.7cm,label=\roman*)}

\newtheorem{theorem}{Theorem}[section]
\newtheorem*{theorem*}{Theorem}
\newtheorem*{cor*}{Corollary}

\newtheorem{lemma}[theorem]{Lemma}
\newtheorem{prop}[theorem]{Proposition}
\newtheorem{cor}[theorem]{Corollary}

\theoremstyle{definition}
\newtheorem{defn}[theorem]{Definition}
\newtheorem{rem}[theorem]{Remark}
\newtheorem{ex}[theorem]{Example}

\input xy
\xyoption{all}

\DeclareMathOperator{\id}{id}

\DeclareMathOperator{\map}{Map}

\DeclareMathOperator{\Top}{Top}

\DeclareMathOperator*{\hocolim}{hocolim}
\DeclareMathOperator*{\holim}{holim}

\DeclareMathOperator{\THH}{THH}
\DeclareMathOperator{\THR}{THR}

\DeclareMathOperator{\hofib}{hof}

\DeclareMathOperator{\conn}{Conn}

\DeclareMathOperator{\KR}{KR}

\DeclareMathOperator{\Sp}{Sp}

\DeclareMathOperator{\Hom}{Hom}
\DeclareMathOperator*{\iter}{\bigcirc}

\begin{document}
\begin{center}\LARGE{Higher Equivariant Excision}
\end{center}

\begin{center}\Large{Emanuele Dotto}

\end{center}
\begin{center}
\textit{\begin{tabular}{c}
dotto@mit.edu\\
Massachusetts Institute of Technology\\
77 Massachusetts Avenue\\
Cambridge MA, USA
\end{tabular}}
\end{center}

\vspace{.1cm}

\abstract{We develop a theory of Goodwillie calculus for functors between $G$-equivariant homotopy theories, where $G$ is a finite group. We construct $J$-excisive approximations for any finite $G$-set $J$. These combine into a poset, the Goodwillie tree, that extends the classical Goodwillie tower. We prove convergence results for the tree of a functor on pointed $G$-spaces that commutes with fixed-points, and we reinterpret the Tom Dieck-splitting as an instance of a more general splitting phenomenon that occurs for the fixed-points of the equivariant derivatives of these functors. 
As our main example we describe the layers of the tree of the identity functor in terms of the equivariant Spanier-Whitehead duals of the partition complexes.
}

\vspace{.7cm}


\tableofcontents

\section*{Introduction}

Goodwillie calculus of functors was developed in the seminal papers \cite{calcI},\cite{calcII} and \cite{calcIII} for functors between the categories of pointed spaces and spectra, and it was later extended to a model categorical framework in \cite{BR}. The key idea is to approximate a homotopy invariant functor between model categories $F\colon\mathscr{C}\to\mathscr{D}$ by a tower of weaker and weaker homology theories
\[F\to \big(\dots\longrightarrow P_nF\longrightarrow P_{n-1}F\longrightarrow\dots\longrightarrow P_1F\big).\]
This theory has found important applications in homotopy theory and in geometric topology (see for example \cite{McCarthy}, \cite{calcI}, \cite{WJR} and \cite{AM}). If $G$ is a finite group and $F\colon\mathscr{C}^G\to\mathscr{D}^G$ is a homotopy functor between model categories of $G$-objects (for example pointed $G$-spaces or genuine $G$-spectra) one can certainly apply Goodwillie's calculus machinery, but with a fundamental flaw. The Goodwillie tower above approximates $F$ by ``na\"{i}ve'' equivariant homology theories, as opposed to ``genuine'' ones. For example the first stage of the tower of the identity functor on pointed $G$-spaces is equivalent to
\[P_1I(X)\simeq \hocolim_{n}\Omega^n\Sigma^n X\]
and not to the stabilization of $X$ by the representations of $G$. The aim of this paper is to adapt Goodwillie calculus to a genuine equivariant context.

The central idea of equivariant calculus is to replace the cubes used in \cite{calcIII} to construct the Goodwillie tower with equivariant cubes that are indexed on finite $G$-sets. This is analogous to the way one replaces the natural numbers with the representations of $G$ in the construction of genuine $G$-spectra.
The paper is based on the foundations of homotopy theory of equivariant diagrams of \cite{Gdiags}, and particularly on the notion of cartesian and cocartesian equivariant cubes. A part of this material is recollected in \S\ref{sec:preliminaries}. 
% Given a finite $G$-set $J$, a $J$-cube in a category $\mathscr{C}$ is a functor $X\colon \mathcal{P}(J)\to\mathscr{C}$ defined on the poset category of subsets of $J$ ordered by inclusion, together with morphisms
%\[X_U\longrightarrow X_{gU}\]
%for every group element $g$ in $G$, which are natural in $U\subset J$, and which satisfies axioms similar to those of a $G$-action. Then $X$ is cartesian or cocartesian if the respective canonical map
%\[X_{\emptyset}\longrightarrow\holim_{\mathcal{P}_0(J)}X\ \ \ \ \ \ \ \ \ \ \ \ \ \ \ \ \ \ \ \ \ \hocolim_{\mathcal{P}_1(J)}X\longrightarrow X_J\]
%is a $G$-equivalence, with respect to certain $G$-actions induced on the homotopy limit and colimit.

In \S\ref{stronglycocart} we introduce the concept of strongly cocartesian $J_+$-cube, which constitutes the building block of the definition of $J$-excision. Here $J_+$ is a finite $G$-set $J$ with an added disjoint basepoint. A $J_+$-cube is strongly cocartesian if it is extended by iterated equivariant homotopy pushouts from the subcategory
\[\bigcup_{o\in J/G}\mathcal{P}(o_+)\backslash o_+\]
of the poset category $\mathcal{P}(J_+)$ of subsets of $J_+$ (see Definition \ref{strongly}). A homotopy functor $\Phi\colon\mathscr{C}^G\to\mathscr{D}^G$ is defined to be $J$-excisive if it sends strongly cocartesian $J_+$-cubes in $\mathscr{C}$ to cartesian $J_+$-cubes in $\mathscr{D}$ (\ref{J-exc}). When $J=\{1,\dots,n\}$ has the trivial $G$-action we recover Goodwillie's notion of $n$-excision (see \ref{rem:excision}), and when $J=G$ is free and transitive we recover the notion of $G$-excision of \cite{Blumberg}, \cite{Gcalc} and \cite{Gdiags}. The following examples of equivariantly excisive functors are in \ref{ex:excision}.
\begin{enumerate}
\item The $G$-equivariant Dold-Thom construction and the identity functor on $G$-spectra are $G$-excisive, and the Real topological Hochschild homology functor associated to a split square-zero extension of Wall-antistructures is $\mathbb{Z}/2$-excisive,
\item For every $G$-spectrum $E$ and every finite $G$-set $K$, the functor $E\wedge (-)^{\wedge K}$ from pointed $G$-spaces to $G$-spectra is $(K\times G)$-excisive, and for every $G$-spectrum $E$ with $\Sigma_n$-action $(E\wedge (-)^{\wedge n})_{h\Sigma_n}$ is $(n\times G)$-excisive,
\item The functor $(E\wedge (-)^{\wedge n})\wedge_{\Sigma_n}E\mathcal{F}_n$ is $(n\times G)$-excisive, where $E\mathcal{F}_n$ is a universal space for $(G,\Sigma_n)$-vector bundles, in the sense of \cite{BH}.
\end{enumerate}

The first main result of the paper is the existence of the universal $J$-excisive approximations, which is proved in \ref{exisiveapprox}.
\begin{theorem*}(Equivariant excisive approximations)
For every homotopy functor $\Phi\colon\mathscr{C}^G\to\mathscr{D}^G$ and every finite $G$-set $J$ there is a $J$-excisive functor $P_J\Phi\colon\mathscr{C}^G\to\mathscr{D}^G$ and a natural transformation $\Phi\to P_J\Phi$ which is essentially initial among the maps from $\Phi$ to a $J$-excisive functor.
\end{theorem*}
In \S\ref{sec:changegroup} we prove a formula that expresses the restriction of $P_J\Phi$ to a subgroup $H$ of $G$ in terms of the excisive approximations of $\Phi$ by certain $H$-subsets of $J$. It follows that the underlying non-equivariant homotopy type of $P_J\Phi$ is that of $P_{J/G}\Phi$. That is to say, for every object $c\in\mathscr{C}^G$ on which $G$-acts trivially there is a non-equivariant equivalence
\[P_{J/G}\Phi(c)\stackrel{\simeq}{\longrightarrow}P_{J}\Phi(c)\]
where $P_{J/G}$ is Goodwillie's $n$-excisive approximation, for the integer $n=|J/G|$. Thus we think of $P_{J}\Phi$ as an equivariant enhancement of $P_n\Phi$ that builds in the orbit types of $J$. This is even more apparent in the case of a transitive $G$-set $T$ ($n=1$) where, if $\Phi$ is reduced, there is an equivalence
\[P_T\Phi\simeq \hocolim_k\Omega^{k T}\Phi \Sigma^{k T}\]
where $\Omega^{k T}$ and $\Sigma^{k T}$ are respectively the loops and the suspension by the $k$-fold direct sum of the permutation representation of $T$ (see \ref{ex:transapprox}).

In \S\ref{sec:tree} we assemble the various $J$-excisive approximations $P_J\Phi$ into a diagram which restricts to the standard Goodwillie tower of $\Phi$ on the $G$-sets with the trivial action. The functoriality of $P_J\Phi$ in $J$ is somewhat richer than it is non-equivariantly. Not only we have maps $P_J\Phi\to P_K\Phi$ when $K$ is a subset of $J$, as one would expect, but we do any time that there is a $G$-map $K\to J$ which induces an injective map on orbits $K/G\hookrightarrow J/G$. This is because in the presence of such a map $K$-excision implies $J$-excision (see \ref{injonorb}). The various $J$-excisive approximations of $\Phi\colon\mathscr{C}^G\to\mathscr{D}^G$ combine into a contravariant diagram shaped over the poset of isomorphism classes of finite $G$-sets ordered by the relation $K\leq J$ if there is a $G$-map $K\to J$ injective on orbits:
\vspace{-.7cm}

\[\xymatrix@R=3pt@C=8pt{
\vdots\ar[d]&&&&&&\vdots\ar[d]&&&&&&&&\vdots\ar[d]
\\
P_{n}\Phi\ar[ddd]|{\stackrel{\vdots}{\phantom{a} }}\ar[rrrrrr]\ar@/_.6pc/[drrrrrrrr]&&&&&&\overset{\phantom{A} }{\underset{\ \ }\dots}\ar[dddr]|<<<<<<{\phantom{ffff} }\ar[rrrrrrrr]|<<<<<<<<<<<<<<<<<<{\ \ }&&\vdots\ar[d]&&&&&&P_{nG}\Phi\ar[ddd]|{\stackrel{\vdots}{\phantom{a} }}
\\
&&&&&&&&\overset{\phantom{A} }{\underset{\ \ }\dots}\ar[ddl]\ar@/_.6pc/[urrrrrr] &&&&&&&
\\
&&&&&&\overset{\phantom{A} }{\underset{\ \ }\dots}\ar@/^.6pc/[drrrrrrrr]&&\ar@{}[dr]|{\ddots}&&&&&&
\\
P_{2}\Phi\ar@/^.6pc/[urrrrrr]\ar[rrr]\ar[dddd]&&&\dots\ar[rrr]&&&&\hspace{-1.6cm}P_{G/H\amalg G/K}\Phi\hspace{-1.3cm}\ar[ddddl]\ar[dddddr] &\ar[rrr]&&&\dots\ar[rrr]&&& P_{2G}\Phi\ar[dddd]
\\
\\
\\
\\
P_{1}\Phi\ar@/_.6pc/[drrrrrrrr]\ar[rrr]&&&\dots\ar[rrr]&&&P_{G/H}\Phi \ar[rrrrr]|<<<<<<<<<<<<{\ \  }&&&&&\dots\ar[rrr]&&& P_{G}\Phi
\\
&&&&&&&&P_{G/K}\Phi\ar@/_.6pc/[urrrrrr]
&&\ar@{}[ul]|<<<<<{\ddots}
}\]
We think of the number of orbits of $J$ as the order of the excision, and of the orbit types of $J$ as ``how genuine'' the excision is. A subtower is the restriction of this diagram to a filtration of proper inclusions of finite $G$-sets
\[\underline{J}=\big(J_1\subset \dots\subset J_n\subset J_{n+1}\subset\dots\big).\]
On the left-hand side lies the Goodwillie tower, corresponding to the ``na\"{i}ve'' excisive approximations, and all the way to the right we find the approximations by the free $G$-sets $nG=n\times G$, that we call the ``genuine'' subtower. The middle part of the diagram consists of the ``subtowers on incomplete universes''. In particular we observe that it is $P_G\Phi$, and no longer $P_1\Phi$, that receives a map from all the excisive approximations. We calculate the genuine tower of the symmetric indexed power functors associated to families of subgroups in \S\ref{indexedpowers}.

In \S\ref{sec:convergence} we discuss the convergence properties of this diagram for the homotopy functors $\Phi\colon \Top^{G}_\ast\to \Top^{G}_\ast$ which commute with fixed-points (see \ref{def:split}). We use the maps which induce isomorphisms on orbits (drawn horizontally in the diagram) to compare the homotopy limits of the different subtowers.
The following is proved in \ref{convergence} and \ref{convanalytic}.
\begin{theorem*}(Convergence)
Let $\Phi\colon \Top^{G}_\ast\to \Top^{G}_\ast$ be a homotopy functor which commutes with fixed-points. A map of filtrations $\underline{f}\colon \underline{K}\to \underline{J}$ such that $f_n\colon K_n\to J_n$ is an isomorphism on $G$-orbits induces a $G$-equivalence on homotopy limits
\[\underline{f}^\ast\colon\holim_{n}P_{J_n}\Phi(X)\stackrel{\simeq}{\longrightarrow}\holim_{n}P_{K_n}\Phi(X)\]
for every pointed $G$-space $X$.
It follows that if the $H$-fixed-points of $\Phi$ are $\rho_H$-analytic for every subgroup $H$ of $G$, the canonical map to the limit of any subtower
\[\Phi(X)\stackrel{\simeq}{\longrightarrow}\holim_{n}P_{J_n}\Phi(X)\]
is a $G$-equivalence for every $X$ with $\rho_H$-connected $H$-fixed-points.
\end{theorem*}
For the identity functor $I\colon \Top^{G}_\ast\to \Top^{G}_\ast$ on pointed $G$-spaces this theorem is closely related to Carlsson's Theorem \cite[II.7]{Carlsson}: the totalizations of the cosimplicial $G$-spaces associated respectively to the na\"{i}ve and to the genuine stable homotopy monads
\[Q(X)=\Omega^{\infty}\Sigma^{\infty}X \ \ \ \ \ \ \ \ \ \ \ \ \ \mbox{and}\ \ \ \ \ \ \ \ \ \ \ \ \ Q_G(X)=\Omega^{\infty G}\Sigma^{\infty G}X\]
are $G$-equivalent. In fact these totalizations, $\holim_{n}P_{n}I(X)$, and $\holim_{n}P_{nG}I(X)$, are all $G$-equivalent (see \ref{Carlssonthing}). The convergence theorem provides a tower which starts at the stable equivariant homotopy of a pointed $G$-space
\[\dots \longrightarrow P_{nG}I(X)\longrightarrow\dots\to P_{2G}I(X)\longrightarrow P_{G}I(X)\simeq Q_{G}(X)\]
and which converges to $X$ if all the fixed-points of $X$ are simply-connected (or in fact nilpotent). In \S\ref{sec:Id} we describe the layers of this tower
\[D_{nG}I=\hofib\big(P_{nG}I\longrightarrow P_{(n-1)G}I\big)\]
that we call the $nG$-differential.
The layers of the Goodwillie tower of the identity on pointed spaces were first calculated in \cite{BJ}, and they were later described in terms of the partition complexes in \cite{AM} and \cite{Arone}.
The partition complex $T_{k}$ is a pointed $\Sigma_k$-space defined from the poset of unordered partitions of the set $\{1,\dots,k\}$. We consider $T_{k}$ as a $G\times\Sigma_k$-space on which $G$ acts trivially. The following is proved in \S\ref{sec:Id} by analyzing the iterated equivariant Snaith-splitting in a fashion which is analogous to \cite{Arone}.
\begin{theorem*}(Layers of the identity)
For every pointed $G$-space $X$ and every integer $n\geq 0$, there is a natural $G$-equivalence
\[D_{nG}I(X)\simeq \Omega^{\infty G}\big(\bigvee_{k=n}^{n|G|}Map_\ast (T_k,\mathbb{S}_G\wedge X^{\wedge k})\wedge_{\Sigma_k}\overline{E}\mathcal{F}_k(n)\big)\]
where $\mathbb{S}_G$ is the $G$-equivariant sphere spectrum, and $\overline{E}\mathcal{F}_k(n)$ is the universal pointed $G\times \Sigma_k$-space of the set of $H$-actions on a $k$-elements set with $n-1$ orbits, and the trivial $H$-sets when $k=n$, for subgroups $H$ of $G$ (see \ref{DnGsym}).
\end{theorem*}

In particular $D_{nG}I(X)$ is the infinite loop space of a $G$-spectrum whose geometric $H$-fixed-points are dual to the partition complexes of $H$-sets with $n-1$ orbits (see \ref{equivariantpart}). As for the genuine fixed-points, the Tom Dieck-splitting gives a decomposition of the $G$-fixed points of $D_GI(X)\simeq Q_G(X)$, which we generalize to higher layers in \S\ref{tomdieckspl}. We say that a homotopy functor $\Phi$ is $J$-homogeneous if it is $J$-excisive and if $P_K\Phi$ is equivariantly contractible for every proper subset $K$ of $J$. In \S\ref{layers} we describe a procedure for extracting a $J$-homogeneous functor $D_J\Phi$ from the excisive approximations of $\Phi$. We think of this functor as the  ``$J$-th layer'' of the Goodwillie tree of $\Phi$, and in particular for trivial and free $G$-sets these are equivalent to the actual layers
\[D_{n}\Phi\simeq\hofib\big(P_n\Phi\to P_{n-1}\Phi\big)\ \ \ \ \ \  \ \ \ \ \ \ \ \ D_{nG}\Phi\simeq\hofib\big(P_{nG}\Phi\to P_{(n-1)G}\Phi\big).\]
In \ref{higertomdieck} we prove the following decomposition result.
\begin{theorem*}(Generalized Tom Dieck-splitting)
Let $\Phi\colon \Top_{\ast}^G\to \Top_{\ast}^{G}$ be a homotopy functor which commutes with fixed-points. For every finite $G$-set $J$ and every pointed $G$-space $X$ there is an equivalence
\[(D_{J}\Phi(X))^G\simeq \prod_{H\lhd G}\hofib\big((D_{J/H}\Phi(X))^G\to \holim_{H<L\lhd G}(D_{J/L}\Phi(X))^G\big)\]
where the product runs over the normal subgroups of $G$. In particular $(D_{J}\Phi(X))^G$ contains the classical differential $(D_{n}\Phi(X))^G$ as a factor, where $n=|J/G|$.
\end{theorem*}
We show in \ref{transTD} and \ref{nonAb} that when $n=1$ and $\Phi$ is the identity this is indeed the Tom Dieck-splitting. In \ref{fixderI} we calculate the terms of this decomposition for the identity functor, for larger $n$. Given a normal subgroup $H\lhd G$ we let $\mathcal{Q}_{k,H}$ be the set of actions on a $k$-elements set of subgroups $L$ of $G$ which contain $H$, and where $H$ is the largest subgroup of $L$ normal in $G$ acting trivially. We let $\mathcal{Q}_{k,H}(n)$ be the subset of the actions with $n-1$ orbits and the trivial actions when $k=n$, and $\overline{E}\mathcal{Q}_{k,H}(n)$ its universal pointed $G\times\Sigma_k$-space. 

\begin{cor*}(Higher Tom Dieck-splitting) For every pointed $G$-space $X$ there is a weak equivalence
\[(D_{nG}I(X))^G\simeq \prod_{H\lhd G}\prod_{k=n}^{n|G|}\Big(\Omega^{\infty G/H}Map_\ast\big(T_k,\mathbb{S}_{G/H}\wedge (X^H)^{\wedge k}\big)\wedge_{\Sigma_k}\overline{E}\mathcal{Q}_{k,H}(n)\Big)^{G/H}\]
where $\mathbb{S}_{G/H}$ is the $G/H$-sphere spectrum, and $G/H$ acts freely on $\overline{E}\mathcal{Q}_{k,H}(n)$. 
\end{cor*}
The key ingredient needed to obtain these splitting results is to deloop $J$-homogeneous functors by the permutation representation of $J$. This is done in \S\ref{delooping} using equivariant enhancements of the techniques of \cite[\S 2]{calcIII}. These delooping results also constitute the first step towards a classification of $J$-homogeneous functors. Section \ref{multilinear} studies the relationship between the cross-effect functor and equivariant excision. As a result of this analysis we obtain a classification of the finitary homotopy functors $\Top^{G}_\ast\to\Top^{G}_\ast$ which are $J$-excisive and $n$-reduced, where $n=|J/G|$. We call these functors strongly $J$-homogeneous, since this condition implies $J$-homogeneity. In \ref{classification} we prove that strongly $J$-homogeneous functors are classified by $G$-spectra on the universe $\bigoplus_k\mathbb{R}[J]$ with a na\"{i}ve $\Sigma_n$-action. The statement of \ref{classification} can be simplified by the following.
\begin{theorem*}(Classification of strongly homogeneous functor) For any strongly $J$-homogeneous homotopy functor $\Phi\colon \Top^{G}_\ast\to \Top_{\ast}^G$ there is a natural equivalence
\[\Phi(X)\simeq \Omega^{\infty J}(E_\Phi\wedge X^{\wedge n})_{h\Sigma_{n}}\]
for every finite pointed ``$J$-CW-complex'' $X$, where $n=|J/G|$ and $E_\Phi$ is a certain $G$-spectrum with na\"{i}ve $\Sigma_n$-action which is constructed from the $n$-th cross effect of $\Phi$.
\end{theorem*}
We suggest that general $J$-homogeneous functors should be classified by a theory of equivariant spectra parametrized by a category of finite $H$-sets with $n-1$ orbits, where the orbit type of these sets depend on $J$ (see \ref{rem:classification}(iv)). A convenient model for the homotopy theory of such spectra is provided in \cite{Mackey} in the setting of spectral Mackey functors. This thesis is supported by the calculation of the layers of the identity functor, which seems to suggest that $nG$-homogeneous functors are classified by the spectral Mackey functors on the category of all  $H$-sets with $n-1$ orbits, and the trivial $H$-set $n$. This classification will be investigated in future work.

\section*{Acknowledgments}
I am thankful to Greg Arone and Mike Hill for very valuable insights and suggestions that ultimately improved the quality of the paper. I am indebted to Tomer Schlank for discussing some of the very technical aspects of this work.
I thank Rosona Eldred for several useful correspondences.
Finally I want to thank the Bourbon Seminar: Clark Barwick, Saul Glasman, Marc Hoyois, Denis Nardin, Sune Precht Reeh, Tomer Schlank and Jay Shah for the many inspiring conversations.

\section{Preliminaries on equivariant diagrams}\label{sec:preliminaries}

We review some of the foundations of equivariant diagrams, recollected from \cite{Gdiags} and \cite{GdiagTop}.

\subsection{Equivariant diagrams}
Let $G$ be a finite group. A $G$-model category is a category $\mathscr{C}$ together with a model structure on the category of $H$-objects $\mathscr{C}^H=Fun(H,\mathscr{C})$ for every subgroup $H$ of $G$. These model structures are required to be compatible under induction and coinduction, and to be suitably enriched over simplicial $H$-sets (see \cite[2.1]{Gdiags}). We also assume that every $\mathscr{C}^H$ is cofibrantly generated and locally finitely presentable. Examples of $G$-model categories include the $G$-category of simplicial sets where $sSet^H$ is equipped with the fixed-points model structure, and the $G$-category of orthogonal spectra where $\Sp_{O}^H$ has the stable model structure associated to the universe of $H$-representations which are restrictions of representations in a complete $G$-universe (see \cite[2.5]{Gdiags}).

Let $I$ be a category with $G$-action, by which we mean a functor $a\colon G\to Cat$. We denote by $I$ the value of $a$ at the unique object of the category $G$, and by $g\colon I\to I$ the action automorphisms.

\begin{defn}(\cite[2.2]{js}, \cite[3.1]{vilf})
An $I$-shaped $G$-diagram in $\mathscr{C}$ is a functor $X\colon I\to \mathscr{C}$ together with a $G$-structure: natural transformations $\phi_g\colon X\to X\circ g$ for every $g$ in $G$, which satisfy $\phi_1=\id_X$ and $\phi_g|_{h}\circ\phi_h=\phi_{gh}$, where $\phi_g|_{h}$ is the restriction of $\phi_g$ along $h\colon I\to I$.
A morphism of $G$-diagrams is a natural transformation $f\colon X\to Y$ which commutes with the $G$-structure. This defines a category of $G$-diagrams, which we denote $\mathscr{C}_{a}^I$.
\end{defn}

We observe that the category $\mathscr{C}_{a}^I$ is isomorphic to the category of diagrams $\mathscr{C}^{G\ltimes_a I}$, where $G\ltimes_a I$ is the Grothendieck construction of the functor $a\colon G\to Cat$. Although this description is often useful (for example we see that $\mathscr{C}_{a}^I$ is complete and co-complete), we think of the morphisms of $G\ltimes_a I$ coming from $I$ and of the morphisms coming from $G$ as two separate structures: an object of $\mathscr{C}_{a}^I$ is a diagram in $\mathscr{C}$ of shape $I$ together with a $G$-structure. Throughout the paper we will mostly be concerned with $G$-diagrams of cubical shape.

\begin{ex}\label{Gcubes}
Let $J$ be a finite $G$-set, and let $\mathcal{P}(J)$ be the poset category of all the subsets of $J$ ordered by inclusion. The category $\mathcal{P}(J)$ has a $G$-action, defined by sending a subset $U\subset J$ to the image $g\cdot U=g(U)$ by the map $g\colon J\to J$. A $G$-diagram $X\in \mathscr{C}_{a}^{\mathcal{P}(J)}$ is called a $J$-cube.
\end{ex}

The presence of a $G$-model structure on $\mathscr{C}$ allows us to define equivariant homotopy limits and colimits of $G$-diagrams with suitable homotopy invariant properties. Indeed, if $X\in \mathscr{C}_{a}^I$ is a $G$-diagram, the Bousfield-Kan formulas for the underlying diagram $X\colon I\to \mathscr{C}$
\[\Hom_I(NI/-,X)\ \ \ \ \ \ \ \ \ \ \ \ \ \ \ \ \ \ \mbox{and}\ \ \ \ \ \ \ \ \ \  \ \ \ \ \ \ \ \ N(-/I)^{op}\otimes_I X\]
have natural $G$-actions constructed from the $G$-structure of $X$, respectively ``by conjugation'' and ``diagonally'' (see \cite[1.16]{Gdiags}). 

\begin{ex}\label{loopsusp}
Let $J$ be a finite $G$-set, and let $J_+$ be the $G$-set $J$ with an added fixed base-point. Let $\mathcal{P}_0(J_+)$ be the subposet of $\mathcal{P}(J_+)$ of non-empty subsets. Given a pointed $G$-space $X$, we define a $G$-diagram $\omega^{J}X\in (\Top_\ast)_{a}^{\mathcal{P}_0(J_+)}$ with values
\[(\omega^{J}X)_U=\left\{\begin{array}{ll}\ast&\mbox{if $U\neq J_+$}\\
X&\mbox{if $U= J_+$}
\end{array}
\right.\]
A proper inclusion $U\to V$ is sent to the base-point inclusion if $V=J_+$, and to the identity otherwise. The $G$-structure maps $(\omega^{J}X)_U\to (\omega^{J}X)_{gU}$ are defined by the identity on the point if $U$ is a proper subset of $J_+$, and by the $G$-action of $X$ for $U=J_+$. There is a natural $G$-equivariant homeomorphism
\[\Hom_I(N\mathcal{P}_0(J_+)/-,\omega^JX)\cong \Omega^JX\]
where $\Omega^J=Map_\ast(S^J,-)$ is the loop space by the permutation representation sphere $S^J=\mathbb{R}[J]^+$, with the conjugation action.
Similarly, there is a dual $G$-diagram $\sigma^{J}X\in (\Top_\ast)_{a}^{\mathcal{P}_1(J_+)}$ shaped on the category $\mathcal{P}_1(J_+)$ of proper subsets of $J_+$, and there is a natural $G$-homeomorphism
\[N(-/\mathcal{P}_1(J_+))^{op}\otimes_I \sigma^JX\cong \Sigma^JX\]
where $\Sigma^J$ is the suspension by the permutation representation sphere $S^J$. The details can be found in \cite[\S 1.3]{Gdiags}.

A similar construction makes sense if we drop the extra basepoint from $J_+$. In this case we recover the loop and suspension spaces by the reduced regular representation of $J$.
\end{ex}

The value of a $G$-diagram $X\in \mathscr{C}_{a}^I$ at an object $i$ of $I$ has an action of the stabilizer group $G_i=\{g\in G\ |\ gi=i\}$ which is induced by the $G$-structure. Moreover the value of a morphism of $G$-diagrams $f\colon X\to Y$ at $i$ is a $G_i$-equivariant map $f_i\colon X_i\to Y_i$.

\begin{defn}\label{def:eqGdiags} Let $\mathscr{C}$ be a $G$-model category, and let $I$ be a category with $G$-action.
A morphism of $G$-diagrams $f\colon X\to Y$ in $\mathscr{C}_{a}^I$ is an equivalence if for every object $i$ of $I$ the map $f_i\colon X_i\to Y_i$ is an equivalence in $\mathscr{C}^{G_i}$.
\end{defn}

It turns out that the equivalences of $G$-diagrams are the weak equivalences of a ``$G$-projective'' model structure on $\mathscr{C}_{a}^I$ (see \cite[2.6]{Gdiags}). We say that $X\in \mathscr{C}_{a}^I$ is pointwise cofibrant (resp. fibrant) if for every $i$ in $I$ the value $X_i$ is a cofibrant (resp. fibrant) object of $\mathscr{C}^{G_i}$. The Bousfield-Kan formulas define functors
\[\Hom_I(NI/-,-)\ ,\ N(-/I)^{op}\otimes_I (-)\colon \mathscr{C}_{a}^I\longrightarrow \mathscr{C}^G\]
which preserve the equivalences between respectively the pointwise fibrant and the pointwise cofibrant $G$-diagrams (see \cite[2.22]{Gdiags}). The existence of the $G$-projective model structure on $\mathscr{C}_{a}^I$ insures that every diagram can be replaced by an equivalent pointwise fibrant or pointwise cofibrant diagram (see \cite[2.10]{Gdiags}). Given a $G$-diagram $X\in\mathscr{C}_{a}^I$, we define $G$-objects in $\mathscr{C}$
\[\holim_I X=\Hom_I(NI/-,FX)\ \ \ \ \ \ \ \ \ \ \ \ \ \ \ \ \hocolim_I X=N(-/I)^{op}\otimes_I QX\]
where $X\stackrel{\simeq}{\to} FX$ and $QX\stackrel{\simeq}{\to} X$ are respectively pointwise fibrant and pointwise cofibrant replacements. Since $\mathscr{C}_{a}^I$ is a simplicial cofibrantly generated model category these replacements can be chosen to be functorial, by \cite[6.3]{RSS}. Thus homotopy limits and homotopy colimits define homotopy invariant functors 
\[\holim_I,\hocolim_I\colon\mathscr{C}_{a}^I\longrightarrow \mathscr{C}^G.\]

\subsection{Homotopy (co)limits and fixed points}\label{sec:fixedpts}

Let $G$ be a finite group, and let us consider the $G$-model category of pointed compactly generated Hausdorff spaces $\mathscr{C}=\Top_\ast$, where the category of $H$-objects $\Top_{\ast}^H$ has the fixed-points model structure for every subgroup $H$ of $G$.

Let $I$ be a category with $G$-action. The category of $G$-diagrams $(\Top_\ast)_{a}^I$ has a fixed-points diagram functor
\[(-)^H\colon (\Top_\ast)_{a}^I\longrightarrow \Top_{\ast}^{I^H}\]
where $I^H$ is the subcategory of $I$ which consists of the objects and the morphisms of $I$ that are fixed (strictly) by the $H$-action. For every $X\in (\Top_\ast)_{a}^I$, the diagram $X^H\colon I^H\to \Top_\ast$ sends an object $i$ to $X^{H}_i$. Observe that since $i$ is $H$-fixed the $G$-structure on $X$ provides $X_i$ with an $H$-action. Moreover a morphism of $G$-diagrams $f\colon X\to Y$ is an equivalence if and only if it restricts to pointwise equivalences $f^H\colon X^H\to Y^H$ for every subgroup $H$ of $G$. 
\begin{ex}\label{ex:fixedcube}
If $J$ is a finite $G$-set and $I=\mathcal{P}(J)$ is of cubical shape, there is an isomorphism of posets $\mathcal{P}(J)^H\cong \mathcal{P}(J/H)$. Thus the $H$-fixed points diagram of a $J$-cube is a $J/H$-cube.
\end{ex}
The pointwise inclusions $X^{G}_i\to X_i$, for $i\in I^G$, induce natural maps
\begin{equation}\label{fixedptsdiag}
\hocolim_{I^G}X^G\stackrel{\cong}{\longrightarrow} (\hocolim_{I}X)^G\ \ \ \ \ \ \ \ \ \ \ \ \ \ \ \ (\holim_{I}X)^G\stackrel{res}{\longrightarrow}\holim_{I^G}X^G
\end{equation}
The first map is a homeomorphism and the second map is a fibration (see \cite[1.2.1,1.3.3]{GdiagTop}). These are analogous respectively to the homeomorphism $(K\wedge Y)^G\cong K^G\wedge Y^G$ and to the restriction map $\map_\ast(K,Y)^G\to \map_\ast(K^G,Y^G)$ for pointed $G$-spaces $K$ and $Y$, where $K$ is cofibrant.

Now suppose that $\mathscr{C}=\Sp_O$ is the $G$-model category of orthogonal spectra, and let $X\in (\Sp_O)^{I}_a$ be a $G$-diagram. We construct similar geometric fixed-points diagrams, and maps analogous to (\ref{fixedptsdiag}). We define the geometric $G$-fixed-points of a spectrum $E\in \Sp^{G}_O$ to be the spectrum $\Phi^G(E)$ with $k$-th level space
\[\Phi^G(E)_k=E_{k\rho}^G\]
where $E_{k\rho}$ is the value of $E$ at the $k$-fold direct sum of the regular representation $\rho$ of $G$ (see \cite{Schwede}). If $i$ is an object of $I^G$, the $i$-vertex of a $G$-diagram $X\in (\Sp_O)^{I}_a$ is an orthogonal $G$-spectrum $X_i\in \Sp_{O}^{G}$, and the geometric fixed-points spectra $\Phi^G(X_i)$ define a diagram
\[\Phi^G(X)\colon I^G\longrightarrow \Sp_O.\]

\begin{prop}\label{prop:geomholim}
For every $G$-diagram of orthogonal spectra $X\in (\Sp_O)^{I}_a$ there is an isomorphism of orthogonal spectra
\[\Phi^G(\hocolim_{I}X)\cong \hocolim_{I^G}\Phi^G(X)\]
and a restriction map $\Phi^G(\holim_{I}X)\to \holim_{I^G}\Phi^G(X)$. If $J$ is a finite $G$-set and $X\in(\Sp_O)^{\mathcal{P}_0(J)}_a$ is a punctured $J$-cube, the restriction map
\[\Phi^G(\holim_{\mathcal{P}_0(J)}X)\stackrel{\simeq}{\longrightarrow} \holim_{\mathcal{P}_0(J)^G}\Phi^G(X)\cong \holim_{\mathcal{P}_0(J/G)}\Phi^G(X)\]
is a stable equivalence of orthogonal spectra.
\end{prop}

\begin{proof}
The claim about homotopy colimits follows immediately from (\ref{fixedptsdiag}) by observing that there is an equivariant homeomorphism
\[(\hocolim_{i\in I} X_i)_{V}\cong \hocolim_{i\in I} (X_i)_{V}\]
for every $G$-representation $V$. The restriction map on homotopy limits is defined as follows. Let $\iota\colon I^G\to I$ be the inclusion of the fixed-points category, and let $j\colon (\iota/-)\to (I/-)$ be the induced map of $I$-diagrams. Precomposition along $j$ defines a map of $G$-spectra 
\[\Hom_I(NI/-,FX)\stackrel{j^\ast}{\longrightarrow} \Hom_I(N\iota/-,FX)\cong\Hom_{I^G}(NI^G/-,(FX)|_{I^G})\]
where $FX$ is a point-wise fibrant replacement of $X$.
The restriction map is defined by taking the geometric fixed-points of this map
\[\Phi^G(\holim_{I}X)=\Phi^G(\Hom_I(NI/-,FX))\stackrel{}{\longrightarrow} \Phi^G(\Hom_{I^G}(NI^G/-,(FX)|_{I^G}))\cong  \holim_{I^G}\Phi^G(X).\]
The last isomorphism is immediate from the definition of $\Phi^G$, because the category $I^G$ indexing the diagram $NI^G/-\colon I^G\to \Sp_{O}^G$ has the trivial action. Also notice that $\Phi^GFX_i$ is a fibrant spectrum for every $i\in I^G$.

Let us show that the restriction map is an equivalence in the cubical case. Suppose that $X\in(\Sp_O)^{\mathcal{P}_0(J)}_a$ is a punctured $J$-cube, and let $\overline{X}\in(\Sp_O)^{\mathcal{P}(J)}_a$ be the extension of $X$ to a $J$-cube defined by setting 
\[\overline{X}_\emptyset=\holim_{\mathcal{P}_0(J)}X.\]
To be precise this is only a coherently commutative diagram, but one can replace $X$ with a Reedy fibrant diagram (with an argument dual to \cite[A.2]{Gdiags}) and take the categorical indexed limit, obtaining a strictly commutative diagram equivalent to $\overline{X}$. The restriction map is an equivalence precisely when the $J/G$-cube of orthogonal spectra $\Phi^G\overline{X}$ is cartesian. Since $\Sp_O$ is stable, it is sufficient to show that $\Phi^G\overline{X}$ is cocartesian. Because of the isomorphism
\[\hocolim_{\mathcal{P}_1(J/G)}\Phi^G\overline{X}\cong \Phi^G(\hocolim_{\mathcal{P}_1(J)}\overline{X})\]
the $J/G$-cube $\Phi^G\overline{X}$ is cocartesian if $\overline{X}$ is a cocartesian $J$-cube (in the sense of \cite[3.3]{Gdiags}). This is the case because cartesian and cocartesian $J$-cubes of spectra agree (\cite[3.35]{Gdiags}), and because $\overline{X}$ is cartesian by construction.
\end{proof}

\subsection{Equivariant homotopy functors}\label{sec:htpyfctr}

Let $G$ be a finite group and let $\mathscr{C}$ and $\mathscr{D}$ be $G$-model categories.
In the paper we will focus on functors $\Phi\colon \mathscr{C}\to\mathscr{D}^G$. These have the advantage, as opposed to functors $\mathscr{C}^G\to\mathscr{D}^G$, of inducing functors \[\Phi_\ast\colon \mathscr{C}^{I}_a\longrightarrow\mathscr{D}^{I}_a\]
for every category with $H$-action $I$, and every subgroup $H$ of $G$. An $H$-diagram $X\in \mathscr{C}^{I}_a$ is sent to the diagram $\Phi(X)$ in $\mathscr{D}$ with the $H$-structure defined by the maps
\[\Phi(X_i)\stackrel{h}{\longrightarrow}\Phi(X_i)\stackrel{\Phi(h)}{\longrightarrow}\Phi(X_{hi})\]
where the first map is the action of $h$ on the $G$-object $\Phi(X_i)$ and the second map is the image by $\Phi$ of the $H$-structure map of $X$. By setting $I$ to be the trivial category with $H$-action we obtain a functor $\Phi\colon \mathscr{C}^H\to\mathscr{D}^H$ for every subgroup $H$ of $G$. 

\begin{defn}(\cite[3.6]{Gdiags})\label{def:htpyfctr} A functor $\Phi\colon \mathscr{C}\to\mathscr{D}^G$ is called a homotopy functor if its extension $\Phi\colon \mathscr{C}^H\to\mathscr{D}^H$ preserves equivalences of cofibrant objects for every subgroup $H$ of $G$.
\end{defn}

Although our input is a functor $\Phi\colon \mathscr{C}\to\mathscr{D}^G$ we always think of $\Phi$ as the collection of extensions $\mathscr{C}^H\to\mathscr{D}^H$ over the subgroups $H$ of $G$. The same goes for equivalences of functors. A natural transformation $\Gamma\colon \Phi\to\Psi$ of functors $\Phi,\Psi\colon \mathscr{C}\to\mathscr{D}^G$ extends to a natural transformation $\Gamma\colon \Phi_\ast\to\Psi_\ast$ between the extensions $\Phi_\ast,\Psi_\ast\colon \mathscr{C}^{I}_a\to\mathscr{D}^{I}_a$ for every category with $H$-action $I$. In particular there is a map $\Gamma_c\colon\Phi(c)\to\Psi(c)$ in $\mathscr{D}^H$ for every $H$-object $c\in\mathscr{C}^H$.

\begin{defn}(\cite[3.8]{Gdiags})
A natural transformation $\Gamma\colon \Phi\to\Psi$ is a weak equivalence if for every subgroup $H$ of $G$ and every object $c\in\mathscr{C}^H$ the map $\Gamma_c\colon\Phi(c)\to\Psi(c)$ is an equivalence in $\mathscr{D}^H$.
\end{defn}

\begin{ex}\label{ex:htpyfctrs}
The following are all examples of homotopy functors in the sense of Definition \ref{def:htpyfctr}.
\begin{enumerate}
\item The inclusion $\mathscr{C}\to \mathscr{C}^G$ of objects with the trivial $G$-action is always a homotopy functor. Its extension $\mathscr{C}^H\to \mathscr{C}^H$ is the identity functor.
\item Let $\Psi\colon \Top_\ast\to\Top_\ast$ be an endofunctor on the category of pointed compactly generated Hausdorff spaces. Suppose that $\Psi$ preserve weak equivalences of pointed spaces and that it commutes with fixed-points (limits over the category $H$, for every subgroup $H$ of $G$). The composition of $\Psi$ with the inclusion of spaces with the trivial $G$-action 
\[\Phi\colon \Top_\ast\stackrel{\Psi}{\longrightarrow}\Top_{\ast}\longrightarrow\Top_{\ast}^G\]
is automatically a homotopy functor. Its extension $\Phi\colon \Top_{\ast}^H\to\Top_{\ast}^H$ satisfies $\Phi(X)^H=\Phi(X^H)$ for every subgroup $H$ of $G$.
\item Waldhausen's $A$-theory functor $A\colon \Top\to\Top$ preserves equivalences and it commutes with fixed-points. Composing it with the inclusion of spaces with the trivial $G$-action yields to a homotopy functor $A\colon \Top\to\Top^G$. The fixed-points space of the extension to $H$-spaces $A(X)^H$ is the $K$-theory of the category of finite retractive spaces over $X^H$. We refer to this functor as na\"{i}ve equivariant $A$-theory.
\item There is a genuine equivariant $A$-theory functor $A_G\colon \Top\to\Top^G$ defined by Merling in \cite[8.8]{Mona} as follows. We let $R_X$ be the Waldhausen category of finite retractive spaces over $X$, and let $EG$ be the translation category of the $G$-set $G$ (its classifying space is $G$-free and contractible). The category of functors $Cat(EG,R_X)$ inherits a Waldhausen structure (see \cite[\S 8.1]{Mona}), and its $K$-theory space is denoted 
\[A_G(X)=K(Cat(EG,R_X)).\]
This space inherits a $G$ action from the action on $EG$, defining a homotopy functor $A_G\colon \Top\to\Top^G$. For every $H$-space $X$ the fixed-points space $A(X)^H$ is the $K$-theory of finite retractive $H$-spaces over $X$ with the na\"{i}ve $H$-equivalences (see \cite[8.9]{Mona}).
\item For any $G$-CW-complex $K$ the functors $K\wedge (-),\map_{\ast}(K,-)\colon  \Top_\ast\to\Top^{G}_\ast$ are homotopy functors. Their extensions to $\Top^{H}_\ast$ are the standard smash and mapping spaces respectively with the diagonal and the conjugation actions.
\item Similarly for any orthogonal $G$-spectrum $E$ the functor $E\wedge (-)\colon  \Top_\ast\to\Sp_{O}^{G}$ is a homotopy functor, and if $E$ is cofibrant so is $E\wedge (-)\colon  \Sp_{O}\to\Sp_{O}^{G}$.
\item For any Abelian group with $G$-action $M$ the Dold-Thom construction $M\colon \Top_\ast\to \Top^{G}_\ast$ that sends a pointed space $X$ to the space of reduced configurations of points in $X$ with labels in $M$ is a homotopy functor. For a $G$-space $X$ the $G$-action on $M(X)$ is both on the points of $X$ and on the labels.
\end{enumerate}
\end{ex}

\section{The equivariant Goodwillie tree}

Let $G$ be a finite group, let $\mathscr{C}$ and $\mathscr{D}$ be $G$-model categories, and let $\Phi\colon\mathscr{C}\to\mathscr{D}^G$ be a homotopy functor. Given a finite $G$-set $J$, we approximate $\Phi$ by a ``$J$-excisive'' functor (to be made precise in \ref{J-exc}). When $J$ is the set $\underline{n}=\{1,\dots,n\}$ with the trivial $G$-action this is Goodwillie's $n$-excisive approximation. When $J=G$ is transitive and free this is the best approximation of $\Phi$ by a genuine $G$-homology theory.

The section is organized as follows. In \S\ref{stronglycocart} we study suitable strongly cocartesian equivariant cubes, which will be the building blocks of the notion of $J$-excision. Section \ref{sec:higherexc} focuses on $J$-excision. In \S\ref{sec:ecxapprox} we construct the $J$-excisive approximations. Section \ref{sec:tree} compares the various $J$-excisive approximations, assembling them into a tree which extends the standard Goodwillie tower. In \S\ref{sec:changegroup} we give formulas for the restrictions of the excisive approximations to subgroups of $G$. In \S\ref{sec:convergence} we discuss the convergence of the subtowers of the tree.

\subsection{Strongly cocartesian equivariant cubes}\label{stronglycocart}

Let $J$ be a finite $G$-set and let $\mathscr{C}$ be a $G$-model category. We recall that the poset category $\mathcal{P}(J)$ of subsets of $J$ ordered by inclusion inherits a $G$-action $a$, and that a $J$-cube in $\mathscr{C}$ is a $G$-diagram $X\in \mathscr{C}^{\mathcal{P}(J)}_a$. Following \cite[3.3]{Gdiags}, a $J$-cube $X$ is called cartesian if the canonical map
\[X_\emptyset\longrightarrow \holim_{\mathcal{P}_0(J)}X\]
is an equivalence in the model category of $G$-objects $\mathscr{C}^G$. Notice that since the empty set is fixed by the $G$-action $X_\emptyset$ is indeed a $G$-object.
Dually, $X$ is cocartesian if
\[\hocolim_{\mathcal{P}_1(J)}X\longrightarrow X_J\]
is an equivalence in $\mathscr{C}^G$.

We want to extend Goodwillie's notion of strongly cocartesian cubes to equivariant cubes. This amounts to choosing a subcategory of $\mathcal{P}(J)$ which determines the cube up to iterated homotopy pushouts. In order to define such a subcategory we need to assume that our indexing $G$-set has a fixed-point, in other words that it is of the form $J_+$. This fixed base-point is analogous to the standard requirement in stable equivariant homotopy theory that the $G$-representations considered contain a copy of the trivial representation (cf. \ref{loopsusp}), and it is consistent with the classical fact that $n$-excision is about $(n+1)$-cubes.

Let $J$ be a finite $G$-set, and define $St(J_+)$ to be the subposet of $\mathcal{P}(J_+)$ defined by the union
\[St(J_+)=\bigcup_{o\in J/G}\mathcal{P}_1(o_+)\]
taken over the set of orbits of $J$. In order to formulate $J$-excision we will have to consider cubes of size $J_+$ that only carry the action of a subgroup of $G$. To this end we let $J_+|_{H}$ be the restriction of the $G$-action on $J_+$ to a subgroup $H$ of $G$.

\begin{defn}\label{strongly} Let $\mathscr{C}$ be a $G$-model category and let $J$ be a finite $G$-set. We say that a $J_+|_H$-cube $X\in \mathscr{C}^{\mathcal{P}(J_+|_{H})}_a$ is $G$-strongly cocartesian if, for every subset $S\subset J_+$ which does not belong to $St(J_+)$, the restriction $X|_{\mathcal{P}(S)}$ is a cocartesian $S$-cube.
\end{defn}
The restriction $X|_{\mathcal{P}(S)}$ is an $S$-cube, where $S$ has an action of the stabilizer group $H_S$ of $S$ in $\mathcal{P}(J_+|_{H})$. Definition \ref{strongly} requires that the map $\hocolim_{\mathcal{P}_1(S)}X\to X_S$
is an equivalence in $\mathscr{C}^{H_S}$ for every $S$ outside $St(J_+)$.
\begin{rem}\label{rem:stronglycocart}\
\begin{enumerate}
\item If $J$ is the set $\underline{n}=\{1,\dots,n\}$ with $n$-elements and the trivial $G$-action, the category $St(\underline{n}_+)$ consists of the initial maps $\emptyset\to \{j\}$, for $j\in \underline{n}_+$. The complement of $St(\underline{n}_+)$ in $\mathcal{P}(\underline{n}_+)$ consists of the subsets which contain at least two elements. Thus the $G$-strongly cocartesian $\underline{n}_+|_H$-cubes in the sense of Definition \ref{strongly} are the same as the strongly cocartesian $(n+1)$-cubes in Goodwillie's sense (\cite[2.1]{calcII}), in the category $\mathscr{C}^H$ of $H$-objects in $\mathscr{C}$,
\item If $J=T$ is a transitive $G$-set, the $G$-strongly cocartesian $T_+|_H$-cubes are just the cocartesian $T_+|_H$-cubes,
\item If $G$ acts non-trivially on $J$, the property of being strongly cocartesian is not preserved by the restriction of the group action to a subgroup. For example, a $J_+|_H$-cocartesian cube which is $G$-strongly cocartesian does not need to be $H$-strongly cocartesian.
This is because the $G$-orbit decomposition of $J$ and the $H$-orbit decomposition of $J|_H$ are usually different. In particular if $J$ has $n$-elements, a $G$-strongly cocartesian  $J_+$-cube does not need to be strongly cocartesian as an $(n+1)$-cube.
\item There is an unbased version of strong cocartesianity, for cubes indexed on an unbased $G$-set $J$. These are the $J$-cubes for which the restrictions $X|_{\mathcal{P}(S)}$ are cocartesian, for every subset $S$ of $J$ outside the subcategory $\bigcup_{o\in J/G}\mathcal{P}_1(o)$ of $\mathcal{P}(J)$. This definition does not extend Goodwillie's definition of strongly cocartesian cube, since for $J=\underline{n}$ these are the $n$-cubes in which all the maps are equivalences. These unbased strongly cocartesian cubes will however lead to a certain notion of ``reduced $J$-excision''. We will compare it to the actual notion of $J$-excision in \ref{unbased}.
\end{enumerate}
\end{rem}

We construct a family of ``generating'' $G$-strongly cocartesian $J_+$-cube. These will be used in section \ref{sec:ecxapprox} to define $J$-excisive approximations to homotopy functors. The standard example of a strongly cocartesian $(n+1)$-cube in $\mathscr{C}$ from \cite{calcIII} is the functor that sends a subset $U$ of $\underline{n+1}$ to the join $c\star U$, for a fixed object $c$ in $\mathscr{C}$. We extend this construction to finite $G$-sets, by gluing together iterated cones, whose dimensions depend on the size of the $G$-orbits of $J$.

\begin{defn}
The star category of a subset $U\subset J_+$ is the subposet of $\mathcal{P}(J_+)$ defined by any of the following equivalent expressions
\[St(U):=St(J_+)/U=St(J_+)\cap\mathcal{P}(U)=\bigcup_{o\in J/G}\mathcal{P}(o_+\cap U)\backslash o_+.\]
\end{defn}

Let $H$ be a subgroup of $G$, and let $c$ be an object of $\mathscr{C}^H$. For every subset $U$ of $J_+$ we define the $U$-star diagram of $c$ to be the functor $St^U(c)\colon St(U)\rightarrow \mathscr{C}$ that sends the empty set to $c$, and every other subset to the terminal object $\ast\in \mathscr{C}$. Define a cube $\Lambda^J(c)\colon \mathcal{P}(J_+|_H)\rightarrow \mathscr{C}$ with vertices
\[\Lambda^J(c)_{U}=\hocolim_{St(U)}St^U(c).\]
The map in the cube $\Lambda^J(c)$ corresponding to an inclusion $U\subset V$ of subsets of $J_+$ is the canonical map
\[\hocolim_{St(U)}St^U(c)=\hocolim_{St(U)}(St^V(c))|_{St(U)}\longrightarrow \hocolim_{St(V)}St^V(c)\]
induced by the associated inclusion of posets $St(U)\subset St(V)$.
The cube $\Lambda^J(c)$ has a canonical $H$-structure, defined as follows.
Every element $g$ of $G$ defines a functor $g\colon St(U)\rightarrow St(gU)$ by restricting the automorphism $g\colon\mathcal{P}(J_+)\to\mathcal{P}(J_+)$. The $H$-action on $c$ induces natural transformations
\[\xymatrix@R=15pt{St(U)\ar[dr]_{h}\ar[rr]^-{St^U(c)}& \ar@{}@<-2ex>[d]^-*+={\Downarrow \! \gamma_h}& \mathscr{C}\\
&St(hU)\ar[ur]_{St^{hU}(c)}
}\]
The $H$-structure is defined by the maps
\[h\colon \Lambda^J(c)_{U}=\hocolim_{St(U)}St^{U}(c)\stackrel{(\gamma_h)_\ast}{\longrightarrow}\hocolim_{St(U)}(St^{hU}(c)\circ h)\longrightarrow \hocolim_{St(hU)}St^{hU}(c)=\Lambda^J(c)_{hU}\]
where the second map is the canonical map induced by the functor $h\colon St(U)\rightarrow St(hU)$.

\begin{ex}\label{spelloutlambda}\
\begin{enumerate}
\item If $J=\underline{n}$ has the trivial $G$-action and $U$ is a subset of $\underline{n}_+$, the $U$-star diagram of $c$ is
\[St^U(c)=\vcenter{\xymatrix@=10pt{
\ast&\dots &\ast\\
\ast&c\ar[l]\ar[ul]\ar[ur]\ar[r]\ar[dr]\ar[d]\ar[dl]\ar[l]
&\ast\\
\ast&\ast&\ast\\
}}\]
with one ray for each element of $U$. The homotopy colimit of this diagram is the join $c\star U$, and $\Lambda^{\underline{n}}(c)$ is the diagram $\Lambda^{\underline{n}}_U(c)=c\star U$ used in \cite{calcIII} to define the $n$-excisive approximations,
\item If $J=T$ is a transitive $G$-set, the poset $St(U)$ is the cube $\mathcal{P}(U)$ if $U$ is a proper subset of $T_+$, and it is the punctured cube $\mathcal{P}_1(T_+)$ if $U=T_+$. The diagram $\Lambda^{T}(c)$ has vertices
\[\Lambda^{T}_U(c)=\left\{\begin{array}{ll}\displaystyle
\hocolim_{\mathcal{P}(U)}St^U(c)=C^U(c)&,\mbox{ if $U$ is a proper subset of $T_+$}\\
\displaystyle
\hocolim_{\mathcal{P}_1(T_+)}St^T(c)=\Sigma^{T}c &,\mbox{ if $U=T_+$}
\end{array}\right.
\]
Here $C^U(c)$ denotes the $U$-fold cone on $c$ (this is in fact our definition of $U$-fold cone), and $\Sigma^{T}c$ is the suspension by the permutation representation of $T$ (see \ref{loopsusp}). Since the $U$-fold cone is $G_U$-contractible when $U$ is not empty ($\mathcal{P}(U)$ has a $G_U$-invariant final object, see \cite[2.25]{Gdiags}), the restriction of $\Lambda^{T}_U(c)$ to $\mathcal{P}_0(T_+)$ is equivalent to the diagram $\omega^T(\Sigma^Tc)$ of \ref{loopsusp}, whose homotopy limit is $\Omega^T\Sigma^Tc$.

\item Suppose that $G=\mathbb{Z}/2=\{0,1\}$ and take $J=\underline{3}\times\mathbb{Z}/2$ the disjoint union of three copies of $\mathbb{Z}/2$ with the diagonal action.
The star categories of the subsets $\underline{3}\times\mathbb{Z}/2$ and $(\mathbb{Z}/2\amalg 0\amalg 0)_+$ are respectively
\vspace{-.5cm}
\[St(\underline{3}\times\mathbb{Z}/2)=\vcenter{\hbox{\xymatrix@R=10pt@C=10pt{
&&\mathbb{Z}/2\\
&0\ar[ur]&&1\ar[ul]\\
&0\ar[dl]&\emptyset\ar[dl]\ar[dr]\ar[l]\ar[r]\ar[ul]\ar[ur]&0
\ar[dr]\\
\mathbb{Z}/2&1\ar[l]&&1\ar[r]&\mathbb{Z}/2
}}}\ \ \ \ 
St((\mathbb{Z}/2\amalg 0\amalg 0)_+)=\vcenter{\hbox{\xymatrix@=4pt{
\mathbb{Z}/2\\
& 1\ar[ul]\ar[rr]&&+1\\
0\ar[uu]\ar[rr]&&+0\\
&\emptyset\ar[ul]\ar[uu]\ar[rr]\ar[dl]\ar[ddr]&&+\ar[uu]\ar[dl]
\ar[ul]
\ar[ddr]\\
0\ar[rr]&&+0\\
&&0\ar[rr]&&+0
}}}
\]
\vspace{-.5cm}
\item In general, the $U$-vertex of $\Lambda^J(c)$ is a gluing of iterated cones or suspensions of $c$. Let us enumerate the set of orbits $J/G=\{o_1,\dots,o_n\}$, and let us define $U_{o}=U\cap o_+$. The homotopy type of $\Lambda^J(c)$ is
\[\Lambda^{J}_U(c)\cong\left\{\!\!\begin{array}{ll}
C^{U_{o_1}}(c)\amalg_c\dots \amalg_cC^{U_{o_n}}(c)\simeq c\star\{o\ |\ U_o\neq\emptyset\} & \mbox{if $+\notin U$}\\
\\
C^{U_{o_1}}(c)\amalg_{C^{+}(c)}\dots \amalg_{C^{+}(c)}C^{U_{o_n}}(c)\simeq \ast
& \mbox{if $+\in U$ and $o_+\not\subset U$}\\
&\mbox{for all $o\in J/G$}\\
\\
(\coprod\limits_{C^+(c)}^{\{o|o_+\subset U\}}\Sigma^{o}c)\amalg_{C^+(c)}(\coprod\limits_{C^+(c)}^{\{o|o_+\not\subset U\}}C^{U_o}(c))\simeq\!\! \coprod\limits_{C^+(c)}^{\{o|o_+\subset U\}}\Sigma^{o}c&\mbox{otherwise}
\end{array}\right.\]
where the coproduct superscript denotes the number of times that we are taking the union over $C^+(c)$.
\end{enumerate}
\end{ex}

\begin{prop}\label{Lambdastronglycocart}
For every subgroup $H$ of $G$, every finite $G$-set $J$ and every $H$-object $c\in \mathscr{C}^H$, the $J_+|_H$-cube $\Lambda^J(c)$ is $G$-strongly cocartesian.
\end{prop}

\begin{proof}
Let $S$ be a subset of $J_+$ that does not belong to $St(J_+)$. We need to show that the canonical map
\[\hocolim_{U\in \mathcal{P}_1(S)}\hocolim_{St(U)}St^U(c)\longrightarrow \hocolim_{St(S)}St^{S}(c)\]
is an equivalence in $\mathscr{C}^{H_S}$. Let $\mathcal{P}_1(S)\wr St$ be the Grothendieck construction of the $G_S$-diagram of categories $St(-)\colon \mathcal{P}_1(S)\rightarrow Cat$, with $G_S$-structure defined by the functors $g\colon St(U)\rightarrow St(gU)$. The category $\mathcal{P}_1(S)\wr St$ inherits an evident $G_S$-action, and the diagrams $St^U(c)$ for $U\in \mathcal{P}_1(S)$ assemble into an $H_S$-diagram
\[St(c)\colon \mathcal{P}_1(S)\wr St\longrightarrow \mathscr{C}\]
that sends an object $(U\in\mathcal{P}_1(S),V\in St(U))$ to $c$ if $V$ is empty, and to the terminal object otherwise. The map above factors as
\[\xymatrix{\displaystyle\hocolim_{U\in \mathcal{P}_1(S)}\hocolim_{St(U)}St^U(c)\ar[dr]_-{\simeq}\ar[r]&\displaystyle
\hocolim_{St(S)}St^{S}(c)\\
&\displaystyle \hocolim_{\mathcal{P}_1(S)\wr St}St(c)\ar[u]_-{p_\ast}
}\]
where the diagonal map is an equivalence in $\mathscr{C}^{H_S}$ by the equivariant Fubini Theorem \cite[2.26]{Gdiags}. The vertical map is induced by the $G_S$-equivariant projection $p\colon \mathcal{P}_1(S)\wr St\rightarrow St(S)$ which sends a pair $(U\in \mathcal{P}_1(S),V\in St(U))$ to $V$ seen as an object of $St(S)$ through the inclusion $St(U)\subset St(S)$. The map $p_\ast$ is well defined, because $St(c)\circ p= St^{S}(c)$ as functors $St(S)\rightarrow \mathscr{C}$. By the equivariant cofinality Theorem of \cite[2.25]{Gdiags}, the map $p_\ast$ is an $H_S$-equivalence if we can show that the under categories $V/ p$ are $H_V$-contractible for every object $V\in St(S)$. The category $V/p$ is the full subcategory of $\mathcal{P}_1(S)\wr St$ of objects 
\[V/ p=\{(U\in \mathcal{P}_1(S),W\in St(U)) \ |\ V\subset W \}.\]
We claim that the pair $(V,V\in St(V))$ is initial in $V/ p$. The only thing that could go wrong is that $V$ is not a proper subset of $S$. But since $V$ belongs to $St(S)$, it is a subset of $o_+\cap S$ properly included in $o_+$, for some $G$-orbit $o$ of $J$. Since $S$ does not belong to $St(J_+)$ we have two possibilities. The first possibility is that $S=o'_+$ for some orbit $o'\in J/G$. In this case either $o=o'$ and $V$ is clearly proper in $S$, or $o\neq o'$ and $V$ is a subset of $\{+\}$, which is a proper subset of $S$. The second possibility is that $S$ intersects two distinct $o_+$ and $o'_+$. Since $V$ is included in $o_+$ it must be a proper subset of $S$.
This shows that the pair $(V,V\in St(V))$ is a well defined initial object of $V/p$. Moreover it is $H_V$-invariant, and it defines an initial object in every fixed categories $(V/ p)^K$ for subgroups $K\leq H_V$. This shows that $V/p$ is $H_V$-contractible.
\end{proof}

We end this section by reformulating the definition of $G$-strongly cocartesian cubes in terms of faces. The next results are analogous to \cite[1.7]{calcII} and \cite[1.20-1.24]{calcII}. The techniques we use for proving these results are different than those of \cite{calcII}. This is because the $G$-actions prevent us from using induction on the dimension of the cubes.

Let $I$ and $J$ be finite $G$-sets, and let $X\in \mathscr{C}^{\mathcal{P}(I\amalg J)}_a$ be an $(I\amalg J)$-cube in $\mathscr{C}$. Through the equivariant isomorphisms of categories $\mathcal{P}(I\amalg J)\cong \mathcal{P}(I)\times \mathcal{P}(J)$, the cube $X$ defines a cube of cubes $\widetilde{X}\colon \mathcal{P}(J)\rightarrow \mathscr{C}^{\mathcal{P}(I)}$ with vertices $\widetilde{X}_U(T)=X(U\cup T)$. We need a bit of care with the group actions, since $\widetilde{X}$ is not a $J$-cube of $I$-cubes. For this one would need natural transformations $\widetilde{X}_U\rightarrow \widetilde{X}_{gU}$, that is maps $X(U\cup T)\rightarrow X((gU)\cup T)$, but the $G$-structure provides us with maps $X(U\cup T)\rightarrow X(gU\cup gT)$. What the $G$-structure on $X$ gives $\widetilde{X}$ is, for every $U\subset J$, a $G_U$-structure on $\widetilde{X}_U$.

\begin{prop}\label{cubeofcartcubes} Let $I$ and $J$ be finite $G$-sets, and let $X\in \mathscr{C}^{\mathcal{P}(I\amalg J)}_a$ be an $(I\amalg J)$-cube in a $G$-model category $\mathscr{C}$.
\begin{enumerate}
\item Suppose that for every non-empty subset $U$ of $J$ the cube $\widetilde{X}_U\in \mathscr{C}^{\mathcal{P}(I|_{G_U})}_a$ is cartesian. Then $X\in \mathscr{C}^{\mathcal{P}(I\amalg J)}_a$ is cartesian if and only if $\widetilde{X}_\emptyset=X|_{\mathcal{P}(I)}\in \mathscr{C}^{\mathcal{P}(I)}_a$ is cartesian.
\item Dually, if $\widetilde{X}_U\in \mathscr{C}^{\mathcal{P}(I|_{G_U})}_a$ is cocartesian for every proper subset $U$ of $J$, then $X\in \mathscr{C}^{\mathcal{P}(I\amalg J)}_a$ is cocartesian if and only if $\widetilde{X}_J\in \mathscr{C}^{\mathcal{P}(I)}_a$ is cocartesian.
\end{enumerate}
\end{prop}

\begin{proof}
We prove the statement about cocartesian cubes. Let us first prove it in the case where $I$ has a basepoint fixed by $G$. That is, let us replace $I$ in the statement by a $G$-set of the form $I_+$.
The functor $F\colon \mathcal{P}_{1}(I_+)\times\mathcal{P}_{1}(J_+)\rightarrow \mathcal{P}_{1}(I_+\amalg J)$ that sends a pair of proper subsets to their union inside $I_+\amalg J$ induces an equivalence in $\mathscr{C}^G$
\[F_\ast\colon \hocolim_{\mathcal{P}_{1}(I_+)\times\mathcal{P}_{1}(J_+)}F^\ast X\stackrel{\simeq}{\longrightarrow}\hocolim_{\mathcal{P}_{1}(I_+\amalg J)}X.\]
This is proved in \cite{Gdiags}, and it strongly relies on the presence of a base-point. Now consider the commutative diagram
\[\xymatrix{\displaystyle\hocolim_{\mathcal{P}_1(I_+)}\widetilde{X}_J\ar[d]_{\alpha}\ar[rr]&&X(I_+\amalg J)\\
\displaystyle\hocolim_{U\in\mathcal{P}_1(J_+)}\hocolim_{\mathcal{P}_1(I_+)}\widetilde{X}_U
\ar[r]_{\simeq}&\displaystyle\hocolim_{\mathcal{P}_1(I_+)\times\mathcal{P}_1(J_+)}F^\ast X\ar[r]^-{F_\ast}_-\simeq&\displaystyle\hocolim_{\mathcal{P}_1(I_+\amalg J)}X\ar[u]
}.\]
The lower left horizontal map is an equivalence by the Fubini Theorem \cite[2.26]{Gdiags}. The right vertical map is an equivalence in $\mathscr{C}^G$ precisely when $X$ is cocartesian, and the top horizontal map when $\widetilde{X}_J$ is cocartesian. Therefore the result follows if we can show that the left vertical map $\alpha$ is an equivalence in $\mathscr{C}^G$. We claim that for every proper subset $U\subset J_+$ different from $J$ there is a natural equivalence in $\mathscr{C}^{G_U}$
\[\hocolim_{\mathcal{P}_{1}(I_+)}\widetilde{X}_U\stackrel{\simeq}{\longrightarrow} X(U\cup I_+).\]
When $U$ is a proper subset of $J$ this is because $\widetilde{X}_U$ is assumed to be cocartesian, and if $U$ contains the basepoint $+$ because the maps $\widetilde{X}_U(V)\rightarrow \widetilde{X}_U(V_+)$ are all identities (see \cite[3.30]{Gdiags}). Therefore $\alpha$ factors as
\[\xymatrix{\displaystyle\hocolim_{\mathcal{P}_{1}(I_+)}\widetilde{X}_J\ar[r]^-{\alpha}\ar[dr]_-{\beta}& \displaystyle\hocolim_{U\in\mathcal{P}_{1}(J_+)}\hocolim_{\mathcal{P}_{1}(I_+)}\widetilde{X}_U\ar[d]^{\simeq}\\
&\displaystyle\hocolim_{U\in\mathcal{P}_{1}(J_+)}Z_U
}\]
where $Z\colon \mathcal{P}_1(J_+)\rightarrow \mathscr{C}$ has vertices $Z_U=X(I_+\cup U)$ for $U\neq J$ and $Z_J=\hocolim_{\mathcal{P}_1(I_+)}\widetilde{X}_J$. We show that $\beta$ is an equivalence. For every proper subset $U\subset J$ the map $Z_U\rightarrow Z_{U_+}$ is the identity map, and this gives a $G_U$-equivariant isomorphism $\hocolim_{U\in\mathcal{P}(J)}Z_U\cong\hocolim_{U\in\mathcal{P}_{1}(J_+)}Z_U$. Therefore the map
\[\beta\colon \hocolim_{\mathcal{P}_{1}(I_+)}\widetilde{X}_J=Z_J\longrightarrow \hocolim_{U\in\mathcal{P}(J)}Z_U\cong\hocolim_{U\in\mathcal{P}_{1}(J_+)}Z_U\]
is an equivalence by cofinality, because $J$ is a $G$-invariant final object of $\mathcal{P}(J)$.

Let us finally consider the case where $I$ does not necessarily have a base-point. For a general finite $G$-set $K$, a $K$-cube $Y\in \mathscr{C}^{\mathcal{P}(K)}_a$ can be extended to a $K_+$-cube $Y^+$, by sending the subsets of $K_+$ that do not contain $+$ to the initial object in $\mathscr{C}$, and the subsets of the form $W_+$ for a $W\subset K$ to $Y_W$. This procedure preserves homotopy colimits, in the sense that there is a natural $G$-equivalence
\[\hocolim_{\mathcal{P}_{1}(K_+)}Y^+\stackrel{\simeq}{\longrightarrow}\hocolim_{\mathcal{P}_{1}(K)}Y.\]
Thus $Y$ is cocartesian if and only if $Y^+$ is cocartesian. Now let $X\in \mathscr{C}^{\mathcal{P}(I\amalg J)}_a$ be an $(I\amalg J)$-cube with the property that for every proper subset $U$ of $J$, the $I|_{G_U}$-cube $\widetilde{X}_U$ is cocartesian. Notice that by definition $(\widetilde{X}^{+})_U=(\widetilde{X}_U)^{+}$ as $I_+|_{G_U}$-cubes. Therefore the cubes $(\widetilde{X}^{+})_U$ are also cocartesian for proper subsets $U$ of $J$. By the pointed argument above, $(\widetilde{X}^{+})_J$ is cocartesian if an only if $X^{+}$ is. Since our extension preserves colimits, this is equivalent to the statement that $\widetilde{X}_J$ is cocartesian if an only if $X$ is.
\end{proof}

\begin{prop}\label{facesstronglycocart}
Let $I$ and $J$ be finite $G$-sets, let $H$ be a subgroup of $G$ and let $X\in \mathscr{C}^{\mathcal{P}(I_+\amalg J|_H)}_a$ be a $G$-strongly cocartesian $(I_+\amalg J)|_H$-cube. Then for every subset $U$ of $J$, the $(I_+)|_{H_U}$-cube $\widetilde{X}_U=X_{U\cup(-)}$ is $G$-strongly cocartesian.
\end{prop}

\begin{proof}
Let $S$ be a subset of $I_+$ which is not in $St(I)$. We need to show that $(\widetilde{X}_U)|_{\mathcal{P}(S)}$ is cocartesian. We prove it by induction on the number of elements in $U$.

If $U$ is empty, $(\widetilde{X}_\emptyset)|_{\mathcal{P}(S)}=X|_{\mathcal{P}(S)}$ is cocartesian because $X$ is assumed to be $G$-strongly cocartesian.
Now suppose that $(\widetilde{X}_W)|_{\mathcal{P}(S)}$ is cocartesian for every subset $W$ of $J$ with $k$ elements, and let $U$ be a subset of $J$ of cardinality $k+1$. This means that $(\widetilde{X}_W)|_{\mathcal{P}(S)}$ is cocartesian for every proper subset $W$ of $U$. By \ref{cubeofcartcubes} the cube $(\widetilde{X}_U)|_{\mathcal{P}(S)}$ is cocartesian if and only if $X|_{\mathcal{P}(S\amalg U)}$ is cocartesian, and this is the case because $X$ is $G$-strongly cocartesian.
\end{proof}

The following is a reformulation of Definition \ref{strongly} in terms of the ``equivariant faces'' of a $J_+$-cube.

\begin{cor}\label{faces}
Let $J$ be a finite $G$-set and let $H$ be a subgroup of $G$. A $J_+|_H$-cube $X\in \mathscr{C}^{\mathcal{P}(J_+|_H)}_a$ is $G$-strongly cocartesian if and only if for every $G$-subset $K$ of $J$ and every subset $U$ of the complement $J\backslash K$, the $K_+|_{H_U}$-cube $\widetilde{X}_U$ is $G$-strongly cocartesian.
\end{cor}
\begin{proof}
One implication is obvious, by choosing $J=K$ and $U=\emptyset$. The other implication is immediate from \ref{facesstronglycocart}, by the $G$-equivariant decomposition $J=K\amalg J\backslash K$.
\end{proof}

The following property of $G$-strongly cocartesian cubes will be helpful later on. It is analogous to the property that the value $X_U$ of a strongly cocartesian $(n+1)$-cube $X$ decomposes as the iterated homotopy pushout of the initial vertices $X_{\{u\}}$'s over $X_{\emptyset}$, for $u\in U$.

\begin{lemma}\label{unionintersectionthing}
Let $X\in \mathscr{C}^{\mathcal{P}(J_+|_H)}_a$ be $G$-strongly cocartesian. Then for every subset $U\in \mathcal{P}(J_+)$ the canonical map from the iterated homotopy pushout
\[X_{U\cap o^{1}_+}\coprod_{X_{U\cap +}}X_{U\cap o^{2}_+}\coprod_{X_{U\cap +}}\dots\coprod_{X_{U\cap +}}X_{U\cap o^{n}_+} \stackrel{\simeq}{\longrightarrow}X_U\]
is an $H_U$-equivalence, where $J/G=\{o^1,\dots, o^n\}$ is the orbits decomposition of $J$.
\end{lemma}

\begin{proof}
If $U$ is a subset of $o_+$ for some $G$-orbit $o$ the result is obvious. We prove the claim on the remaining subsets of $J_+$ by induction on the cardinality of $U$. Suppose that $U$ is not included in any $o_+$. In particular $U$ is not in $St(J_+)$, and thus $X_U$ is equivalent to
\[X_U\stackrel{\simeq}{\longleftarrow}\hocolim_{V\in\mathcal{P}_1(U)}X_V
\stackrel{\simeq}{\longleftarrow}
\hocolim_{V\in\mathcal{P}_1(U)}\big(X_{V\cap o^{1}_+}\coprod_{X_{V\cap +}}X_{V\cap o^{2}_+}\coprod_{X_{V\cap +}}\dots\coprod_{X_{V\cap +}}X_{V\cap o^{n}_+}\big)
\]
where the second map is an equivalence by the inductive hypothesis. This is equivalent to
\[(\hocolim_{V\in\mathcal{P}_1(U)}X_{V\cap o^{1}_+})\coprod_{\displaystyle\hocolim_{V\in\mathcal{P}_1(U)}X_{V\cap +}}\dots\coprod_{\displaystyle\hocolim_{V\in\mathcal{P}_1(U)}X_{V\cap +}}(\hocolim_{V\in\mathcal{P}_1(U)}X_{V\cap o^{n}_+}).\]
In every $U$-cube $X_{(-)\cap o_+}$ all the maps $X_{V\cap o_+}\to X_{(V\cup \{u\})\cap o_+}$ are identities, for a choice of $u\in U\backslash o_+$, and similarly for the cube $X_{(-)\cap +}$. Thus the $U$-cubes $X_{(-)\cap o_+}$ and $X_{(-)\cap +}$ are homotopy cocartesian (by \cite[3.30]{Gdiags}), and the expression above is equivalent to
\[X_{U\cap o^{1}_+}\coprod_{X_{U\cap +}}X_{U\cap o^{2}_+}\coprod_{X_{U\cap +}}\dots\coprod_{X_{U\cap +}}X_{U\cap o^{n}_+}.\]
\end{proof}

\subsection{Higher equivariant excision}\label{sec:higherexc}

The aim of this section is to define excision with respect to a finite $G$-set $J$. Let $\mathscr{C}$ and $\mathscr{D}$ be $G$-model categories, and let $\Phi\colon \mathscr{C}\to\mathscr{D}^G$ be a homotopy functor (see \ref{def:htpyfctr}). We recall that $\Phi$ induces a functor
\[\Phi_\ast\colon \mathscr{C}^{\mathcal{P}(J_+)}_a\longrightarrow\mathscr{D}^{\mathcal{P}(J_+)}_a\]
from the category of $J_+$-cubes in $\mathscr{C}$ to the category of $J_+$-cubes in $\mathscr{D}$.

\begin{defn}\label{J-exc}
Let $\mathscr{C}$ and $\mathscr{D}$ be $G$-model categories and let $J$ be a finite $G$-set. A homotopy functor $\Phi\colon \mathscr{C}\to\mathscr{D}^G$ is called $J$-excisive if the functor $\Phi_\ast\colon \mathscr{C}^{\mathcal{P}(J_+|_H)}_a\to\mathscr{D}^{\mathcal{P}(J_+|H)}_a$ sends $G$-strongly cocartesian $J_+|_H$-cubes in $\mathscr{C}$ to cartesian cubes in $\mathscr{D}$, for every subgroup $H$ of $G$.
\end{defn}

\begin{rem}\label{rem:excision}\
\begin{enumerate}
\item Suppose that $J$ is the set $\underline{n}=\{1,\dots, n\}$ with the trivial $G$-action. A homotopy functor $\Phi\colon \mathscr{C}\to\mathscr{D}^G$ is $\underline{n}$-excisive if and only if the associated functor $\Phi\colon \mathscr{C}^H\to\mathscr{D}^H$ is $n$-excisive in the sense of \cite[3.1]{calcII}, for every subgroup $H$ of $G$. This is immediate from \ref{rem:stronglycocart}(i).
\item Suppose that $J=T$ is a transitive $G$-set, that the categories $\mathscr{C}$ and $\mathscr{D}$ are pointed, and that $\Phi(\ast)$ is equivalent to $\ast$ in $\mathscr{D}^G$. Then $\Phi$ is $T$-excisive if and only if for every $c\in\mathscr{C}^H$ a certain natural zig-zag of maps defines an equivalence
\[\Phi(c)\simeq \Omega^{T|_H}\Phi(\Sigma^{T|_H}c)\]
in $\mathscr{D}^H$, for every $H\leq G$. This was first proved in \cite[3.25]{Gdiags} for enriched functors, and it is related to the loops and suspension cubes of \ref{loopsusp}. It is a direct consequence of Theorem \ref{exisiveapprox} below.
\item Suppose that $J$ is the free transitive $G$-set $J=G$ and that $\mathscr{C}$ and $\mathscr{D}$ are pointed. It turns out that under a mild extra assumption on the $G$-model category $\mathscr{D}$, it is sufficient to verify the $G$-excision condition on $G_+$-cubes, and not on all $G_+|_H$-cubes for subgroups $H$ of $G$ (see \cite[3.23]{Gdiags}). The author is not aware of a similar result for general $G$-sets.
\end{enumerate}
\end{rem}

\begin{ex}\label{ex:excision}\
\begin{enumerate}
\item Let $M$ be a simplicial Abelian group with additive $G$-action. The equivariant Dold-Thom construction $M(-)\colon sSet_\ast\to sSet_{\ast}^G$ is $G$-excisive (see \cite[3.3]{Gdiags}). Observe that for a fixed pointed $G$-space $X$ and a discrete $M$, the homotopy groups of $M(X)^G$ are the Bredon homology of $X$ with coefficients in the Mackey functor $G/H\mapsto M^H$.
\item Let $E$ be an orthogonal $G$-spectrum. The functor $E\wedge (-)\colon \Top_\ast\to \Sp_{O}^G$ is $G$-excisive. The same is true for $E\wedge (-)\colon \Sp_{O}\to \Sp_{O}^G$ when $E$ is cofibrant. This is because $E\wedge(-)$ commutes with equivariant colimits, and therefore it preserves cocartesian equivariant cubes. The claim follows from the fact that cocartesian $J$-cubes of spectra are cartesian for every finite $G$-set $J$ (see \cite[3.35]{Gdiags}). In particular the identity functor on $G$-spectra is $G$-excisive.
\item Let $A$ be a ring with an anti-involution $w\colon A^{op}\to A$, and let $M$ be an $A$-bimodule with an anti-involution $h\colon M^{op}\to M$ (an additive map which satisfies $h(am)=h(m)w(a)$ and $h(ma)=w(a)h(m)$ for all $a\in A$ and $m\in M$). The Real topological Hochschild homology of $A$ with coefficients in $M$ is a symmetric $\mathbb{Z}/2$-spectrum $\THR(A;M(-))$, defined from suitably derived smash products
\[HM\wedge \bigwedge_{[k]} HA\]
where $[k]=\{0,\dots,k\}$ has the $\mathbb{Z}/2$-action $i\mapsto k-i$, and the smash over $[k]$ is the norm construction of \cite{HHR}. This construction has been studied in details in the author's thesis \cite{thesis}. By replacing $M$ with the Dold-Thom construction $M(X)$ of a pointed $\mathbb{Z}/2$-space $X$, we obtain a functor $\THR(A;M(-))\colon sSet_\ast\to \Sp_{\Sigma}^{\mathbb{Z}/2}$ to the category of symmetric $\mathbb{Z}/2$-spectra, analogous to the functor $\THH(A;M(-))$ of \cite{DM}. At least when $2\in A$ is invertible, the functor $\THR(A;M(-))$ is $\mathbb{Z}/2$-excisive (see \cite[5.2.5]{thesis} and \cite[3.2.7,3.3.2]{Gcalc}).
\item Let $K$ be a finite $G$-set and let $E$ be an orthogonal $G$-spectrum. The functor $\Top_\ast\to \Sp_{O}^ G$ that sends a pointed space $X$ to the suspension of the $K$-indexed smash product
\[E\wedge X^{\wedge K}\]
is $(K\times G)$-excisive (notice that $K\times G$ is equivariantly isomorphic to $|K|\times G$, where $|K|$ is the underlying set of $K$ with the trivial $G$-action). Our proof uses the equivariant excisive approximations, and it is postponed to Proposition \ref{smashfctr}. In particular $E\wedge X^{\wedge n}$ is $n\times G$-excisive.
\item Let $E$ be an orthogonal spectrum with $G\times \Sigma_n$-action. The homotopy functor $\Top_\ast\to \Sp_{O}^ G$ that sends a pointed space $X$ to
\[(E\wedge X^{\wedge n})_{h\Sigma_n}\]
is also $n\times G$-excisive. This follows from \ref{ex:excision}(iv) and the fact that cocartesian equivariant cubes and cartesian equivariant cubes of spectra agree. Since we are taking homotopy orbits for the $\Sigma_n$-action, we are thinking of $E$ as a genuine $G$-spectrum with na\"{i}ve $\Sigma_n$-action. In our point-set models, this model category and the model category of genuine $G\times\Sigma_n$-spectra have the same underlying category.
\item Let $\mathcal{F}_n$ be the family of subgroups $\Gamma$ of $G\times \Sigma_n$ that intersect $1\times \Sigma_n$ trivially. Alternatively, this is the family of subgroups $\Gamma$ such that the $\Sigma_n$-action on $(G\times \Sigma_n)/\Gamma$ is free, or equivalently the family of graphs of the group homomorphisms $H\to\Sigma_n$ for subgroups $H\leq G$. Let $E\mathcal{F}_n$ be a classifying space for this family, that is a $G\times\Sigma_n$-space whose fixed points $E\mathcal{F}^{\Gamma}_n$ are contractible if $\Gamma\in\mathcal{F}_n$ and empty otherwise. The space $E\mathcal{F}_n$ is a universal space for $(G,\Sigma_n)$-vector bundles, and it occurs in the theory of $N_\infty$-operads of \cite{BH}, in the equivariant Snaith-splitting of \cite[VII-5.7]{LMS}, and in the proof of Sullivan's conjecture of \cite{Carlsson}. 

Given an orthogonal spectrum with $G\times \Sigma_n$-action, let $\mathcal{S}^{\mathcal{F}_n}_E\colon \Top_\ast\to \Sp_{O}^ G$ be the homotopy functor that sends a pointed space $X$ to
\[\mathcal{S}^{\mathcal{F}_n}_E(X)=(E\wedge X^{\wedge n})\wedge_{\Sigma_n}(E\mathcal{F}_n)_+\]
where $\Sigma_n$ acts on the $n$-fold smash product by permuting the factors. We will show in \ref{smashfctr} that $\mathcal{S}^{\mathcal{F}_n}_E$ is $n\times G$-excisive. We think of $E$ as a spectrum with a genuine action of the subgroups of the family $\mathcal{F}_n$. These are $G\times\Sigma_n$-spectra whose homotopy type is determined by the fixed-point spectra of the subgroups $\Gamma\in \mathcal{F}_n$ and which have all the transfers between these groups. An efficient definition of this category is Barwick's category of spectral Mackey functors on the Burnside category of the full subcategory of the orbit category of $G\times \Sigma_n$ generated by the transitive $G\times \Sigma_n$-sets which are $\Sigma_n$-free (see also \ref{rem:classification}(iv)).
\end{enumerate}
\end{ex}

The $J$-excision property is invariant under the notion of equivalence of \ref{def:eqGdiags}. We recall that a natural transformation $\Gamma\colon\Phi\to \Psi$ is an equivalence if $\Gamma_c\colon\Phi(c)\to\Psi(c)$ is an equivalence in $\mathscr{D}^H$ for every $H$-object $c\in\mathscr{C}^H$.

\begin{lemma}
Let $\Phi,\Psi\colon \mathscr{C}\to\mathscr{D}^G$ be equivalent homotopy functors, and let $J$ be a finite $G$-set. Then $\Phi$ is $J$-excisive if and only if $\Psi$ is $J$-excisive.
\end{lemma}

\begin{proof}
An equivalence $\Gamma\colon\Phi\to \Psi$ extends to an equivalence $\Phi_\ast\simeq\Psi_\ast\colon \mathscr{C}^{I}_a\to\mathscr{D}^{I}_a$ for every category with $H$-action $I$. By this we mean that for every $H$-diagram $X$ in $\mathscr{C}^{I}_a$ the map $\Gamma_X\colon \Phi(X)\to\Psi(X)$ is an equivalence in $\mathscr{D}^{I}_a$, in the sense of \ref{def:eqGdiags}. In particular if $X\in\mathscr{C}^{\mathcal{P}(J_+|_H)}_a$ is $G$-strongly cocartesian, the $J_+|_H$-cubes $\Phi(X)$ and $\Psi(X)$ in $\mathscr{D}$ are equivalent. Hence $\Phi(X)$ is cartesian if and only if $\Psi(X)$ is cartesian.
\end{proof}

We analyze how $J$-excision and $K$-excision compare for different $G$-sets $J$ and $K$. Classically, any $n$-excisive functor is $m$-excisive for every $m\geq n$. The relationship for equivariant calculus turns out to be more involved, as we show in the next three results.

\begin{prop}\label{inclusions}
Let $J$ be a finite $G$-set and let $K\subset J$ be a $G$-invariant subset of $J$. Every $K$-excisive homotopy functor $\Phi\colon \mathscr{C}\to\mathscr{D}^G$ is also $J$-excisive.
\end{prop}

\begin{proof}
Let $X\in \mathscr{C}^{\mathcal{P}(J_+|_H)}_a$ be $G$-strongly cocartesian. Decompose $J$ as the disjoint union of $G$-sets $J=K\amalg I$ where $I$ the complement of $K$ in $J$. By \ref{facesstronglycocart}, for every subset $U$ of $I$, the cube $\widetilde{X}_U=X_{U\cup(-)}\colon \mathcal{P}(K_+)\to \mathscr{C}$ is a $G$-strongly cocartesian $(K_+|_{H_U})$-cube. Since $\Phi$ is $K$-excisive $\Phi_\ast(\widetilde{X}_U)=\widetilde{\Phi_\ast(X)}_U$
is a cartesian $K_+|_{H_U}$-cube in $\mathscr{D}$, for every $U\subset I$. By \ref{cubeofcartcubes} the $J$-cube $\Phi(X)$ is also cartesian.
\end{proof}

\begin{prop}\label{isoonorb}
Let $p\colon K\to J$ be an equivariant map of finite $G$-set that induces an isomorphism on $G$-orbits. Any $K$-excisive homotopy functor $\Phi\colon  \mathscr{C}\to\mathscr{D}^G$ is also $J$-excisive. In particular if $\Phi\colon  \mathscr{C}\to\mathscr{D}^G$ is $K$-excisive, the functor $\Phi\colon\mathscr{C}^H\to\mathscr{D}^H$ is $|K/G|$-excisive in non-equivariant calculus, for every subgroup $H$ of $G$.
\end{prop}

\begin{proof}
Let $X\in \mathscr{C}^{\mathcal{P}(J_+|_H)}_a$ be a $G$-strongly cocartesian $J_+|_H$-cube. We need to show that $\Phi(X)$ is cartesian. Precomposition with the image functor $p\colon \mathcal{P}(K_+)\to \mathcal{P}(J_+)$ induces a pullback map $p^\ast\colon \mathscr{C}^{\mathcal{P}(J_+|_H)}_a\to \mathscr{C}^{\mathcal{P}(K_+|_H)}_a$. We claim that $p^\ast X$ is $G$-strongly cocartesian. To see this, let $S\subset K_+$ be a subset which is not in the star category $St(K_+)$. We need to show that $(p^\ast X)|_{\mathcal{P}(S)}$ is cocartesian. The square of categories with $G$-action
\[
\xymatrix{
\mathcal{P}(S)\ar[r]\ar[d]_p&\mathcal{P}(K_+)\ar[d]^{p}\\
\mathcal{P}(p(S))\ar[r]&\mathcal{P}(J_+)
}
\]
clearly commutes, where the horizontal maps are inclusions. Hence the restriction $(p^\ast X)|_{\mathcal{P}(S)}$ is equal to $p^{\ast}(X|_{\mathcal{P}(p(S))})$. Since $p$ is an isomorphism on orbits the subset $p(S)\subset J_+$ lies outside $St(J_+)$. Thus $X|_{\mathcal{P}(p(S))}$ is cocartesian. The functor $p^{\ast}\colon \mathscr{C}^{\mathcal{P}(p(S))}_a\to \mathscr{C}^{\mathcal{P}(S)}_a$ preserves cocartesian cubes by \cite[A3]{Gdiags}, because the map $p\colon S\to p(S)$ is surjective. This shows that $p^{\ast}(X|_{\mathcal{P}(p(S))})=(p^\ast X)|_{\mathcal{P}(S)}$ is cocartesian, and therefore that $p^\ast X$ is $G$-strongly cocartesian. As $\Phi$ is $K$-excisive, the $K_+|_H$-cube $\Phi(p^\ast X)=p^{\ast}\Phi(X)$ is cartesian. Observe that since $p$ is an isomorphism on orbits, it is a surjective map. By \cite[A1]{Gdiags} the functor $p^\ast\colon \mathscr{C}^{\mathcal{P}(J_+|_H)}_a\to \mathscr{C}^{\mathcal{P}(K_+|_H)}_a$ detects cartesian cubes, and $\Phi(X)$ must be cartesian.

Now consider the projection map $p\colon K\to K/G$. This is an equivariant map, where $G$ acts trivially on $K/G$, and it induces an isomorphism on $G$-orbits. Thus a $K$-excisive functor $\Phi$ is $K/G$-excisive. Since $K/G$ has the trivial $G$-action, the functors $\Phi\colon\mathscr{C}^H\to\mathscr{D}^H$ are $|K/G|$-excisive for every $H\leq G$, by \ref{rem:excision}(i). 
\end{proof}

\begin{cor}\label{injonorb}
Let $\alpha\colon K\to J$ be an equivariant map of finite $G$-sets, which descends to an injective map $\overline{\alpha}\colon K/G\to J/G$ on orbits. Then every $K$-excisive homotopy functor $\Phi\colon \mathscr{C}\to\mathscr{D}^G$ is also $J$-excisive.
%
%Conversely, if $K$ and $J$ are two finite $G$-sets with the property that every $K$-excisive functor $\Phi\colon \mathscr{C}\to\mathscr{D}^G$ is also $J$-excisive, there exist an equivariant map $K\to J$ which is injective on orbits.
\end{cor}

\begin{proof}
Every $G$-map $\alpha\colon K\to J$ factors canonically as $K\to \alpha(K)\to J$. The second map is the inclusion of the image of $\alpha$. The first map is the corestriction $\alpha\colon K\to \alpha(K)$, which is an isomorphism on orbits since $\overline{\alpha}\colon K/G\to J/G$ is assumed to be injective. Hence a $K$-excisive functor $\Phi\colon \mathscr{C}\to\mathscr{D}^G$ is $\alpha(K)$-excisive by \ref{isoonorb}, and further $J$-excisive by \ref{inclusions}.
%
%Conversely, suppose that there is no $G$-map $K\to J$ injective on orbits. We need to define a functor $\Phi$ which is $K$-excisive, but not $J$-excisive. Observe that there cannot be a $G$-map injective on orbits $K\to nG$, where $n=|J/G|$ (otherwise we can compose it with $nG\to J$).
%
%
%
%
% Let us denote $n=|K/G|$ and $k=|J/G|$. Suppose first that $n>k$. In this case the functor $\Phi(X)=\mathbb{S}\wedge X^n$ from \ref is $nG$-excisive, hence $K$-excisive by \re}. However $P_k\Phi\simeq \ast$ by \cite{calcIII} and therefore $\Phi$ is not $k$-excisive, and in particular not $J$-excisive by \ref. Now suppose that $n\leq k$.
%................................................
%
%
\end{proof}

\begin{rem}
Let us denote $nG=n\times G$ the disjoint union of $n$-copies of $G$ with the diagonal action, and let $J$ be a $G$-set with $n$-orbits. A choice of point $j_o$ in every orbit $o\in J/G$ defines a projection map $G\to G/G_{j_o}\cong o$. These maps assemble into a map $nG\to J$ which is bijective on orbits. By \ref{isoonorb}, any $nG$-excisive functor is also $J$-excisive.

Similarly given any $G$-set $J$, a choice of an orbit $o\in J/G$ and of a point $j\in o$ determines a map $G\to G/G_j\cong o\to J$ which is injective on $G$-orbits. Hence a $G$-excisive functor is $J$-excisive, for every finite $G$-set $J$.
\end{rem}

We end the section with a discussion about the role played by the basepoint of $J_+$ in the definition of $J$-excision. Define the reduced star category of a finite $G$-set $J$ to be the subposet
\[\overline{St}(J)=\bigcup_{o\in J/G}\mathcal{P}_1(o)\ \ \subset \mathcal{P}(J).\]
By replacing $St(J_+)$ with $\overline{St}(J)$ in the definition of $G$-strongly cocartesian $J_+$-cubes, we obtain a notion of $G$-strongly cocartesian $J$-cubes which leads to a different notion of $J$-excision. We call this ``reduced $J$-excision'', and we call a homotopy functor $\Phi\colon \mathscr{C}\to\mathscr{D}^G$ that satisfies this condition $\overline{J}$-excisive.

\begin{rem}\
\begin{enumerate}
\item If $J=n$ has the trivial $G$-action, every homotopy functor is $\overline{n}$-excisive. This is because $\overline{St}(J)$ is the trivial poset $\{\emptyset\}$, and a strongly cocartesian $n|_H$-cubes in the unbased sense is an $n$-cube in $\mathscr{C}^H$ in which all the maps are equivalences. These cubes are preserved by homotopy functors, and they are cartesian.
\item
Let $J=T$ be a transitive $G$-set and suppose that $\mathscr{C}$ and $\mathscr{D}$ are pointed. Let $\Phi\colon \mathscr{C}\to\mathscr{D}^G$ be a homotopy functor which preserves the zero-object, and let us assume for simplicity that $\Phi$ is $sSet$-enriched. Then $\Phi$ is $T$-excisive if and only if the adjoint assembly map $\Phi(c)\to \Omega^{T|_H}\Phi(\Sigma^{T|_H} c)$ is an equivalence for every $H$-object $c\in\mathscr{C}^H$ (see \cite[3.26]{Gdiags}). Similarly, $\Phi$ is $\overline{T}$-excisive if and only if the adjoint assembly map $\Phi(c)\to\Omega^{\overline{T}|_H}\Phi(\Sigma^{\overline{T}|_H} c)$ is an equivalence for every $H$-object $c\in\mathscr{C}^H$, where $\Omega^{\overline{T}}$ and $\Sigma^{\overline{T}}$ denote respectively the loops and the suspension with respect to the reduced permutation representation $\overline{\mathbb{R}}[T]=\ker(\mathbb{R}[T]\to \mathbb{R})$. Observe that $T$-excision implies $\overline{T}$-excision, since $\overline{\mathbb{R}}[T]$ is a sub-representation of $\mathbb{R}[T]$. Since $\mathbb{R}[T]$ splits canonically as $\mathbb{R}[T]\cong \overline{\mathbb{R}}[T]\oplus\mathbb{R}$, we see that $\overline{T}$-excision and $T$-excision are equivalent if we require additionally that $\Phi$ is classically $1$-excisive. Proposition \ref{unbased} generalizes this relationship.
\end{enumerate}
\end{rem}

\begin{prop}\label{unbased} Let $J$ be a finite $G$-set.
Every $J$-excisive homotopy functor $\Phi\colon \mathscr{C}\to\mathscr{D}^G$ is $\overline{J}$-excisive and $|J/G|$-excisive. The converse is true if the $G$-action on $J$ is either trivial or transitive.
\end{prop}
\begin{proof}
Suppose first that $\Phi$ is $J$-excisive. It is classically $|J/G|$-excisive by \ref{isoonorb}. Let $X$ be a $G$-strongly cocartesian $J|_H$-cube, and let us prove that $\Phi(X)$ is cartesian. Define a $J_+|_H$-cube $Y$ with vertices $Y_U=X_U$ for subsets $U$ of $J$, and with the terminal object $Y_{U_+}=\ast$ at the subsets that contain the basepoint. We claim that $Y$ is a strongly cocartesian $J_+|_H$-cube. Let $S$ be a subset of $J_+$ which is not in $St(J_+)$. If $S$ does not contain the basepoint it must intersect at least two orbits, and therefore it belong to $\overline{St}(J)$. Since $X$ is a $G$-strongly cocartesian $J|_H$-cube the map
\[\hocolim_{\mathcal{P}_1(S)}Y=\hocolim_{\mathcal{P}_1(S)}X\stackrel{\simeq}{\longrightarrow}X_S=Y_S\]
is an equivalence. If $S$ contains the basepoint, let us write $S=R_+$ for a subset $R$ of $J$. Notice that $R$ does not belong to $\overline{St}(J)$.
The homotopy colimit over $\mathcal{P}_1(R_+)$ can be computed in two steps, as
\[\hocolim_{\mathcal{P}_1(R_+)}Y\simeq
\hocolim\left(\vcenter{\hbox{\xymatrix@=13pt{\displaystyle
\hocolim_{\mathcal{P}_1(R)}Y\ar[r]\ar[d]&\displaystyle\hocolim_{\mathcal{P}_1(R)}Y_{(-)\cup +}\\
Y_{R}
}}}\right)\simeq \hocolim\left(\vcenter{\hbox{\xymatrix@=13pt{
X_R\ar[r]\ar[d]&\ast\\
X_R
}}}\right)\simeq \ast =Y_{R_+}.\]
This shows that $Y$ is strongly cocartesian. By assumption $\Phi(Y)$ is a cartesian $J_+|_H$-cube. Moreover the $J|_H$-cube $\Phi(Y_{(-)\cup +})$ is the constant cube $\Phi(\ast)$, and it is therefore cartesian. By \ref{cubeofcartcubes} the restriction $\Phi(Y)|_{\mathcal{P}(J)}=\Phi(X)$ is also cartesian.

Now let us prove the converse. If $J$ has the trivial $G$-action there is nothing to prove. Suppose that $J=T$ is transitive and that $\Phi$ is both $\overline{T}$-excisive and $(|T/G|=1)$-excisive. Let $X$ be a cocartesian $T_+$-cube. For every $U\subset T$ consider the homotopy cocartesian square
\[\xymatrix@=13pt{X_U\ar[r]\ar[d]&X_{U_+}\ar[d]\\
X_T\ar[r]&P_U
}\]
where $P_U$ is the homotopy pushout. This defines a $T$-cube of cocartesian squares $Z\colon \mathcal{P}(T)\to \mathscr{C}^{\mathcal{P}(1_+)}$. Since $\Phi$ is $|T/G|=1$-excisive the $T$-cube of squares $\Phi(Z)\colon \mathcal{P}(T)\to \mathscr{D}^{\mathcal{P}(1_+)}$ is pointwise cartesian. By \ref{facesstronglycocart} the adjoint $T\amalg 1_+$-cube $\Phi(\widetilde{Z})\in \mathscr{D}^{\mathcal{P}(T\amalg 1_+)}_a$ is also cartesian, that is the canonical map
\[\Phi(X_\emptyset)=\Phi(Z_\emptyset)\stackrel{\simeq}{\longrightarrow}\holim_{\mathcal{P}_0(T\amalg 1_+)}\Phi(\widetilde{Z})\]
is an equivalence in $\mathscr{D}^G$. It is therefore sufficient to show that this homotopy limit is equivalent to $\holim_{\mathcal{P}_0(T_+)}\Phi(X)$.
This can be calculated in two steps, as
\[\holim_{\mathcal{P}_0(T\amalg 1_+)}\Phi(\widetilde{Z})\simeq
\holim\left(\!
\vcenter{\hbox{\xymatrix@=13pt{
&\displaystyle\holim_{\mathcal{P}_0(T_+)}\Phi(\widetilde{Z})\ar[d]\\
\Phi(\widetilde{Z}_{1})\ar[r]&\displaystyle
\holim_{\mathcal{P}_0(T_+)}\Phi(\widetilde{Z}_{(-)\amalg 1})
}}}\!
\right)\!
=\holim
\left(\!
\vcenter{\hbox{\xymatrix@=13pt{
&\displaystyle\holim_{\mathcal{P}_0(T_+)}\Phi(X)\ar[d]\\
\Phi(\widetilde{Z}_1)\ar[r]&\displaystyle
\holim_{\mathcal{P}_0(T_+)}\Phi(\widetilde{Z}_{(-)\amalg 1})
}}}\!
\right).\]
It is therefore sufficient to show that the bottom horizontal map is an equivalence, or in other words that $\Phi(\widetilde{Z}_{(-)\amalg 1})\in\mathscr{D}^{\mathcal{P}(T_+)}_a$ is cartesian. As a map of $T$-cubes, it is given at a subset $U\subset T$ by
\[\Phi(\widetilde{Z}_{U\amalg 1})=\Phi(X_T)\longrightarrow \Phi(P_U)=\Phi(\widetilde{Z}_{U_+\amalg 1})
.\]
The $T$-cube $\Phi(X_T)$ is constant, and in particular it is cartesian. By \ref{facesstronglycocart} it suffices to show that $\Phi(P_U)$ is a cartesian $T$-cube. Since $\Phi$ is assumed to be $\overline{T}$-excisive this holds if we show that $P_U$ is a cocartesian $T$-cube in $\mathscr{C}$. It's homotopy colimit over $\mathcal{P}_1(T)$ is by definition
\[\hocolim_{U\in \mathcal{P}_1(T)}P_U\simeq \hocolim\left(\vcenter{\hbox{\xymatrix@=13pt{\displaystyle
\hocolim_{\mathcal{P}_1(T)}X\ar[r]\ar[d]&\displaystyle\hocolim_{\mathcal{P}_1(T)}X_{(-)\cup +}\\
X_T
}}}\right)\simeq \hocolim_{\mathcal{P}_1(T)}X.\]
Since $X$ is assumed to be cocartesian, this last homotopy colimit is equivalent to $X_{T_+}$, which is also equivalent to
\[X_T\simeq \hocolim\left(\vcenter{\hbox{\xymatrix@=13pt{X_T\ar[r]\ar[d]&X_{T_+}\\
X_T
}}}\right)=P_{T}.\]

%Define a $J/G$-cube of $J$-cubes $Z\colon \mathcal{P}(J/G_+)\times \mathcal{P}(J)\to C$, by $Z_{(W,U)}=X_{p^{-1}(W)\cup U}$ where $p\colon J_+\to J/G_+$ is the projection map. When $J$ has two orbits $z$ and $w$, this is the cube of $J$-cubes
%\[\vcenter{\hbox{\xymatrix@=8pt{X\ar[rr]\ar[dd]\ar[dr]
%&&X_{+\cup (-)}\ar[dr]\ar@{|}[d]\\
%&X_{z\cup (-)}\ar[dd]\ar[rr]&\ar[d]&X_{z_+\cup (-)}\ar[dd]\\
%X_{w\cup (-)}\ar@{-}[r]\ar[dr]&\ar[r]&X_{w_+\cup (-)}\ar[dr]\\
%& X_{z\cup w}\ar[rr]&&X_{z\cup w_+\cup (-)}
%}}}\]
%We claim that for every subset $U$ of $J$, the $J/G_+$-cube $Z(-,U)$ is strongly cocartesian, and that for every non-empty subset $W$ of $J/G_+$ different than $\{+\}$, the $J$-cube $Z(W,-)$ is strongly cocartesian. The first claim is clear: the faces of $Z(-,U)$ are of the form
%\[\xymatrix@=13pt{X_{A}\ar[r]\ar[d]&X_{A\amalg z}\\
%X_{A\amalg w}&X_{A\amalg z\amalg w}
%}\]
%for subsets $A$ of $J_+$.
\end{proof}

\subsection{Equivariant excisive approximations}\label{sec:ecxapprox}

Let $G$ be a finite group, let $\mathscr{C}$ and $\mathscr{D}$ be $G$-model categories, and let $J$ be a finite $G$-set. The aim of this section is to construct the universal $J$-excisive approximation of a homotopy functor $\Phi\colon \mathscr{C}\to\mathscr{D}^G$.

Recall from Section \ref{stronglycocart} that for every $H$-object $c$ in $\mathscr{C}^H$ there is a $G$-strongly cocartesian $J_+|_H$-cube $\Lambda^J(c)$, with vertices
\[\Lambda_{U}^J(c)=\hocolim_{St(U)}St^U(c)\]
where $St(U)$ is the subposet $St(U)=\bigcup_{o\in J/G}\mathcal{P}_1(o_+)\cap\mathcal{P}(U)$ of $\mathcal{P}(J_+)$, and $St^U(c)$ is the diagram with value $c$ at the empty-set, and with the terminal object everywhere else. The $J$-excisive approximation of $\Phi$ is defined in a manner similar to Goodwillie's $n$-excisive approximation $P_n$, by replacing the $(n+1)$-cube $X\star U$ from \cite{calcIII} with $\Lambda_{U}^J(c)$. We start with defining a functor $T_J\Phi\colon \mathscr{C}\to\mathscr{D}^G$ by
\[T_J\Phi(c)=\holim_{\mathcal{P}_0(J_+)} \Phi(\Lambda^J(c)).\]
The functor $T_J$ comes equipped with a natural transformation $t_J\Phi\colon \Phi(c)=\Phi(\Lambda_{\emptyset}^J(c))\rightarrow T_J\Phi(c)$. We let the $J$-excisive approximation of $\Phi$ be the functor $P_J\Phi\colon \mathscr{C}\to\mathscr{D}^G$ defined as the sequential homotopy colimit
\[P_J\Phi(c)=\hocolim\big(\Phi(c)\stackrel{t_J\Phi}{\longrightarrow} T_J\Phi(c)\stackrel{t_JT_J\Phi}{\longrightarrow} T_JT_J\Phi(c)\longrightarrow\dots\big).\]
The canonical map from the initial object of the sequence to the homotopy colimit defines a natural transformation $p_J\colon \Phi\rightarrow P_J\Phi$. Since for the trivial $G$-set $J=\underline{n}$ there is a natural isomorphism $\Lambda^{\underline{n}}_U(c)\cong c\star U$ this construction extends Goodwillie's $n$-excisive approximations.

\begin{theorem}\label{exisiveapprox}
The functor $P_J\Phi\colon \mathscr{C}\to\mathscr{D}^G$ is $J$-excisive, and the map $p_J\colon \Phi\rightarrow P_J\Phi$ is essentially initial among the maps from $\Phi$ to a $J$-excisive functor.
\end{theorem}

\begin{proof}
Let $H$ be a subgroup of $G$ and let $X\in \mathscr{C}^{\mathcal{P}(J_+|_H)}_a$ be a $J_+|_H$-cube in $\mathscr{C}$. We use the technique of \cite{calcIII} and \cite{rezk} of factorizing the map $t_J\Phi\colon \Phi(X)\rightarrow T_J\Phi(X)$ through a certain $J_+|_H$-cube $Y\in \mathscr{D}^{\mathcal{P}(J_+|_H)}_a$. Then we show that when $X$ is $G$-strongly cocartesian $Y$ is cartesian. By cofinality, the sequential homotopy colimit $P_J\Phi(X)$ will be equivalent to a sequential homotopy colimit of cartesian cubes, which is itself homotopy cartesian (this is proved in \cite[A.8]{Gdiags} and it uses that each $\mathscr{D}^H$ is locally finitely presentable). This will show that $P_J\Phi$ is $J$-excisive.

We define a functor $K\colon \mathcal{P}(J_+)\times \mathcal{P}(J_+)\rightarrow \mathscr{C}$ by
\[K(U,T)=\hocolim_{S\in St(J_+)}X_{(S\cap U)\cup T}.\]
This has an $H$-structure defined by the natural transformations
\[h\colon K(U,T)\stackrel{h}{\longrightarrow} \hocolim_{S\in St(J_+)}X_{h\big((S\cap U)\cup T\big)}=\hocolim_{St(J_+)}X_{((-)\cap hU)\cup hT}\circ h\longrightarrow K(hU,hT)\]
where the first map is from the $H$-structure of $X$ and the second map is the restriction along the functor $h\colon St(J_+)\rightarrow St(J_+)$. Define the $J_+|_H$-cube $Y\in \mathscr{D}_{a}^{\mathcal{P}(J_+|_H)}$ by
\[Y_T=\holim_{\mathcal{P}_0(J_+)}\Phi(K(-,T))\]
with the $H$-structure induced from the one of $K$. We define, for every $T\in\mathcal{P}(J_+)$, a factorization
\[\xymatrix@R=13pt{\Phi(X_T)\ar[dr]_-\phi\ar[rr]^-{t_J\Phi(X_T)}
&&\displaystyle\holim_{\mathcal{P}_0(J_+)} \Phi(\Lambda^J(X_T))=T_J\Phi(X_T)\\
&\displaystyle\holim_{\mathcal{P}_0(J_+)}\Phi(K(-,T))\ar[ur]_-{\psi}
}\]
which assemble into a factorization of the map $t_J\Phi(X)$ through $Y$. The map $\psi$ is the homotopy limit of $\Phi$ applied to the map of $H_T$-diagrams $K(-,T)\rightarrow \Lambda^J(X_T)$ defined at a vertex $U\in \mathcal{P}_1(J_+)$ by
\[K(U,T)=\hocolim_{S\in St(J_+)}X_{(S\cap U)\cup T}=
\hocolim_{St(J_+)}X_{(-)\cup T}\circ ((-)\cap U){\longrightarrow}\hocolim_{S\in St(U)}X_{S\cup T}\longrightarrow \Lambda^J(X_T).\]
Here the first map is induced by the functor $(-\cap U)\colon St(J_+)\rightarrow St(U)$. The second one is induced by the map of $St(U)$-shaped diagrams $X_{-\cup T}\rightarrow St^U(X_T)$ which is the identity on the initial vertex $\emptyset\in St(U)$, and which is the map to the terminal object everywhere else. The map $\phi$ is induced by the natural transformation of $\mathcal{P}_0(J_+)$-diagrams $\Phi(X_T)\to \Phi(K(-,T))$ (where $\Phi(X_T)$ is seen as a constant diagram) defined at a vertex $U\in \mathcal{P}_0(J_+)$ by applying $\Phi$ to the composite
\[X_T\stackrel{\simeq}{\longrightarrow}X_T\otimes NSt(J_+)\cong\hocolim_{St(J_+)}X_T\longrightarrow \hocolim_{St(J_+)}X_{((-)\cap U)\cup T}.\]
Here the first map is induced by the inclusion $\{\emptyset\}\to St(J_+)$, and the second one by the inclusions $T\to (S\cap U)\cup T$.

%For a proper subset $U$ of $J_+$ we denote $\Delta^{\overline{U}}:=N\mathcal{P}_0(U)$. The map $\phi$ has $U$-component $\Phi(X_T)\rightarrow map_{\mathscr{D}}(\Delta^{\overline{U}},\Phi(K(U,T)))$ which is the adjoint of the map
%\[\Phi(X_T)\otimes\Delta^{\overline{U}}\longrightarrow \Phi(X_T\otimes\Delta^{\overline{U}})\longrightarrow \Phi(K(U,T))\]
%Here the second map is induced by $X_T\otimes\Delta^{\overline{U}}\rightarrow X_{T\cup U}\otimes\Delta^{\overline{U}}\stackrel{S=U}{\longrightarrow} K(U,T)$.

We are left with proving that $Y$ is cartesian whenever $X$ is $G$-strongly cocartesian. We claim that when $X$ is $G$-strongly cocartesian, there is an equivalence of $H$-diagrams
\begin{equation}\label{eqKX}
K(U,T)\simeq X_{U\cup T}
\end{equation}
analogous to the situation of \cite{rezk}.
This would show that for any non-empty $U\subset J_+$ the cube $\Phi(K(U,-))$ is cartesian, since for every $u\in U$ the map $\Phi(X_{U\cup T})\rightarrow \Phi(X_{U\cup (T\cup u)})$ is the identity (see \cite[3.30]{Gdiags}). Hence $Y$ will be a homotopy limit of cartesian cubes, and therefore itself cartesian. In order to prove the equivalence (\ref{eqKX}) we calculate the homotopy colimit defining $K(U,T)$ by covering the category $St(J_+)$ by the punctured poset cubes $\mathcal{P}_{1}(o_+)$, for $o\in J/G$. By \cite[1.10]{calcII} the homotopy colimit of a diagram over $St(J_+)$ decomposes as 
\begin{equation}\label{eqKiter}
\displaystyle K(U,T)\stackrel{\simeq}{\longleftarrow}\hocolim_{W\in\mathcal{P}_1(J/G)}\hocolim_{\bigcap\limits_{o\in (J/G)\backslash W}\mathcal{P}_{1}(o_+)}X_{((-)\cap U)\cup T}
\end{equation}
(it also follows by \ref{coverings}, which is an equivariant enhancement of \cite[1.10]{calcII}).
Since the various posets $\mathcal{P}_{1}(o_+)$ overlap only on the edge $\emptyset\to +$, the inner homotopy colimit in (\ref{eqKiter}) is
\[\hocolim(X_{T}\to X_{(+\cap U)\cup T})\stackrel{\simeq}{\longrightarrow} X_{(+\cap U)\cup T}\]
whenever $|W|<|J/G|-1$, and in particular it is independent of $W$. For the subsets $W\subset J/G$ of the form $W=(J/G)\backslash o$ for some $G$-orbit $o$, the inner homotopy colimit of (\ref{eqKiter}) is equivalent to
\[\hocolim_{\mathcal{P}_{1}(o_+)}X_{((-)\cap U)\cup T}\stackrel{\simeq}{\longrightarrow} X_{(o_+\cap U)\cup T}.\]
This map is an equivalence for the following reasons. When $U$ contains $o_+$ we have that $\big(X_{((-)\cap U)\cup T}\big)|_{\mathcal{P}(o_+)}=\big(X_{(-)\cup T}|_{\mathcal{P}(o_+)}\big)$. This is the translated cube $\widetilde{X}_T$ from \ref{faces} if $T$ does not intersect $o_+$, and it has therefore homotopy colimit equivalent to $X_{(o_+\cap U)\cup T}$ because $X$ is $G$-strongly cocartesian. In the case where $T$ does intersects $o_+$, the maps $X_{S\cup T}\to X_{(S\cup t)\cup T}$ are identities for $t\in T\cap o_+$, and the colimit is also equivalent to $X_{(o_+\cap U)\cup T}$, by \cite[3.30]{Gdiags}. Finally, if $U$ does not contain $o_+$, the maps $X_{(S\cap U)\cup T}\to X_{((S\cup j)\cap U)\cup T}$ are identities for some $j$ in $o_+$ which does not belong to $U$, and the homotopy colimit is again equivalent to $X_{(o_+\cap U)\cup T}$. Thus (\ref{eqKiter}) becomes an equivalence
\[K(U,T)\simeq\hocolim_{W\in\mathcal{P}_{1}(J/G)}\left\{\begin{array}{ll}
X_{(o_+\cap U)\cup T}&\mbox{for $W=(J/G)\backslash o$ for some $o\in J/G$}\\
X_{(+\cap U)\cup T} &\mbox{for}\ |W|<|J/G|-1
\end{array}\right.\]
This is the iterated homotopy pushout 
\[K(U,T)\simeq X_{(o^{1}_+\cap U)\cup T}\coprod_{X_{(+\cap U)\cup T}}\dots \coprod_{X_{(+\cap U)\cup T}} X_{(o^{n}_+\cap U)\cup T}\]
which is equivalent to $X_{U\cup T}$ by the decomposition \ref{unionintersectionthing} applied to the $G$-strongly cocartesian cube $\widetilde{X}_T$ from \ref{faces}.
This concludes the proof that $P_J\Phi$ is $J$-excisive.

The argument of \cite[1.8]{calcIII} applies verbatim to our equivariant situation, showing that $\Phi\to P_J\Phi$ is homotopy initial among maps to $J$-excisive functors.
\end{proof}

\begin{rem}\label{PJcommute} The excisive approximations are defined as sequential homotopy colimits of finite equivariant homotopy limits. Because we are assuming that the model categories $\mathscr{D}^H$ are locally finitely presentable sequential homotopy colimits in $\mathscr{D}^G$ commute with finite homotopy limits of $G$-diagrams in $\mathscr{D}$ (see \cite[A.8]{Gdiags}). It follows that the equivariant excisive approximations satisfy properties similar to \cite[1.7]{calcIII}. More precisely for every pair of finite $G$-sets $J$ and $K$:
\begin{enumerate}
\item There are natural equivalences $T_KT_J\Phi\simeq T_JT_K\Phi$ and $P_KP_J\Phi\simeq P_JP_K\Phi$,
\item The construction $T_J$ commutes with pointwise equivariant homotopy limits and with pointwise sequential homotopy colimits of functors,
\item The construction $P_J$ commutes with pointwise finite equivariant homotopy limits and with pointwise sequential homotopy colimits of functors,
\item If the target $G$-model category $\mathscr{D}$ is $J$-stable (the cocartesian $J_+$-cubes are the same as the cartesian $J_+$-cubes, for example the $G$-model category of $G$-spectra) both $T_J$ and $P_J$ commute with all the pointwise homotopy colimits of functors.
\end{enumerate}
\end{rem}

We end the section with some examples of equivariant excisive approximations.

\begin{ex}\label{ex:transapprox}
Suppose that $\mathscr{C}$ and $\mathscr{D}$ are pointed, and that $\Phi\colon\mathscr{C}\to \mathscr{D}^G$ sends the zero object of $\mathscr{C}$ to an object $G$-equivalent to the zero object of $\mathscr{D}^G$. Suppose that $J=T$ is a transitive $G$-set. Then $\Lambda^T(c)$ is equivalent to the $T$-suspension cube $\sigma^Tc$ of Example \ref{loopsusp}. Since $\Phi$ preserves the zero object the natural transformation $\omega^T\Phi(\Sigma^Tc)\to\Phi(\sigma^Tc)$ induced by the map from the zero-object is an equivalence, where $\omega^T$ is the loop cube of \ref{loopsusp}. This induces a weak equivalence on homotopy limits
\[T_T\Phi(c)=\holim_{\mathcal{P}_0(T_+)}\Phi(\sigma^Tc)\stackrel{\simeq}{\longleftarrow}
\holim_{\mathcal{P}_0(T_+)}\omega^T\Phi(\Sigma^Tc)\cong \Omega^T\Phi(\Sigma^Tc).\]
It follows that the $n$-fold iteration $T_{T}^{(n)}\Phi(c)$ is equivalent to $\Omega^{nT}\Phi(\Sigma^{nT}c)$. If $\Phi\colon\mathscr{C}\to \mathscr{D}^G$ is enriched over simplicial sets there are actual maps $\Phi(c)\to \Omega^{nT}\Phi(\Sigma^{nT}c)$, and a natural equivalence
\[P_T\Phi(c)\simeq\hocolim\big(
\Phi(c)\longrightarrow\Omega^{T}\Phi(\Sigma^{T}c)\longrightarrow \Omega^{2T}\Phi(\Sigma^{2T}c)\longrightarrow\dots
\big).\]
In this case the proof of \ref{exisiveapprox} specializes to \cite[3.26]{Gdiags}.
\end{ex}

\begin{ex}[Genuine vs na\"{i}ve homology theories]\label{naivevsgenuine}
Let $E\in \Sp^{G}_O$ be an orthogonal $G$-spectrum. For every transitive $G$-set $T$, we consider the infinite loop space functor $\Omega^{T}_E\colon \Top_\ast\to \Top_{\ast}^G$ defined by
\[\Omega^{T}_E(X)=\Omega^{\infty T}(E\wedge X)=\hocolim_k\Omega^{kT}(E_{kT}\wedge X)\]
where $E_{kT}$ is the value of $E$ at the permutation representation $\mathbb{R}[kT]$.
This is the homology theory on the universe $\bigoplus_{k\geq 0}\mathbb{R}[T]$ associated to $E$, or said in other words $\Omega^{T}_E$ is $T$-excisive (the approximation map $\Omega^{T}_E\to P_T\Omega^{T}_E$ is an equivalence by \ref{ex:transapprox}). In particular $\Omega^{G}_E$ is a genuine homology theory ($G$-excisive) and $\Omega^{1}_E$ is a na\"{i}ve one ($\underline{1}$-excisive). If $f\colon T\to S$ is a map of transitive $G$-sets, there is an induced map $f^\ast\colon \Omega^{S}_E(X)\to \Omega^{T}_E(X)$ defined by suspending by the sphere of the canonical complement of the inclusion $f^\ast\colon \mathbb{R}[S]\to \mathbb{R}[T]$. We claim that there is a natural zig-zag of equivalences $\Omega^{T}_E(X)\simeq P_T\Omega^{S}_E(X)$, under which the canonical map $f^\ast$ corresponds to the universal approximation map $p_T\colon \Omega^{S}_E(X)\to P_T \Omega^{S}_E(X)$. In particular $\Omega^{G}_E$ is the $G$-excisive approximation of $\Omega^{1}_E$: the closest genuine homology theory approximating $\Omega^{1}_E$. The zig-zag is
\[P_T\Omega^{S}_E(X)\stackrel{\simeq}{\longleftarrow}\hocolim_k\Omega^{kT}\hocolim_n\Omega^{nS}(E_{nS}\wedge \Sigma^{kT}X)\stackrel{\simeq}{\longrightarrow}\hocolim_{k,n}\Omega^{nS+kT}(E_{nS+kT}\wedge X)\]
and the latter space is equivalent to $\Omega^{T}_E(X)$ because $\mathbb{R}[T]$ is cofinal among the sums of permutation representations of $T$ and $S$ (in fact $f^\ast$ includes $\mathbb{R}[S]$ into $\mathbb{R}[T]$).
\end{ex}

\begin{ex}
Let $A$ be a ring with a Wall antistructure and let $M$ be an $A$-bimodule. One of the main results of \cite{Gcalc} shows that the $\mathbb{Z}/2$-excisive approximation of the Real $K$-theory functor $\widetilde{\KR}(A\ltimes M(-))\colon sSet_\ast\to\Sp^{\mathbb{Z}/2}_{\Sigma}$ is equivalent to $H(A;M(S^{1,1}))$, the Real MacLane homology of $A$ with the coefficients in the Dold-Thom construction of the sign-representation sphere $S^{1,1}$. The spectrum $H(A;M(S^{1,1}))$ is equivalent to the Real topological Hochschild homology $\THR(A;M(S^{1,1}))$ when $2\in A$ is invertible. This is an equivariant analogue of \cite{DM}.
\end{ex}

We conclude the section by determining the excision properties of the indexed power functor.

\begin{prop}\label{smashfctr} Let $E$ be an orthogonal $G$-spectrum and $K$ a finite $G$-set.
The functor $\mathcal{T}_{E}^ K\colon \Top_\ast\to \Sp_{O}^ G$ defined by
\[\mathcal{T}_{E}^ K(X)=E\wedge X^{\wedge K}\]
 is $K\times G$-excisive, or equivalently $nG$-excisive where $n=|K|$.
\end{prop}

\begin{proof}
It suffices to show that the map $\mathcal{T}_{E}^ K\to T_{K\times G}\mathcal{T}_{E}^ K$ is an equivalence, that is to say that for every pointed $H$-space $X$ the $(K\times G)_+|_H$-cube cube $\mathcal{T}_{E}^ K(\Lambda^{K\times G}(X))$ is cartesian. Since the category of $G$-spectra is $G$-stable, this is equivalent to showing that $\mathcal{T}_{E}^ K(\Lambda^{K\times G}(X))$ is cocartesian (see \cite[3.35]{Gdiags}). Moreover smashing with $E$ preserves colimits, and it is sufficient to show that the cube of pointed spaces $\Lambda^{K\times G}(X)^{\wedge K}$ is cocartesian. We verify that the map
\[\hocolim_{U\in\mathcal{P}_1(K\times G_+)}(\big(\hocolim_{V\in St(U)}St^{U}(X)_V\big)^{\wedge K})\longrightarrow\big(\hocolim_{W\in St(K\times G_+)}St^{K\times G_+}(X)_{W}
\big)^{\wedge K}\]
is indeed an $H$-equivalence.
Classically the smash product commutes with homotopy colimits in each variable. An easy fixed points argument using the isomorphism (\ref{fixedptsdiag}) from \S\ref{sec:fixedpts} shows that smashing over the $G$-set $K$ also commutes with colimits in each variable, in the sense that the map above is an $H$-equivalence if and only if
\[\hocolim_{U\in\mathcal{P}_1(K\times G_+)}\hocolim_{\underline{V}\in St(U)^K}\big(St^{U}(X)_{V_k}^{\wedge K}\big)\longrightarrow\hocolim_{\underline{W}\in St(K\times G_+)^K}\big(St^{K\times G_+}(X)_{W_k}^{\wedge K}
\big)\]
is an $H$-equivalence. Here $\underline{V}=\{V_k\}_{k\in K}$ is a collection of objects of $St(U)$, whose components are permuted by the $G$-action, and similarly for $\underline{W}$. This $G$-action induces a $G$-structure on the categories valued functor 
\[St(-)^K\colon \mathcal{P}_1(K\times G_+)\longrightarrow Cat\]
defining a further $G$-action on its Grothendieck construction $\wr St(-)^K$. The diagrams $St^U(X)^{\wedge K}$ assemble into a $G$-diagram $St(X)^{\wedge K}$ on the category with $G$-action $\wr St(-)^K$. There is an equivariant projection $p\colon \wr St(-)^K\to St(K\times G_+)^K$ that sends a pair $(U\subsetneq K\times G_+,\underline{V}\in St(U)^K)$ to $\underline{V}$ viewed as a collection of subsets of $K\times G_+$. The map above fits into a commutative diagram
\[\xymatrix{\displaystyle\hocolim_{U\in \mathcal{P}_1(K\times G_+)}\hocolim_{\underline{V}\in St(U)^K}\big(St^{U}(X)_{V_k}^{\wedge K}\big))\ar[r]\ar[d]^-{\simeq}&\displaystyle\hocolim_{\underline{W}\in St(K\times G_+)^K}\big(St^{(K\times G)_+}(X)_{W_k}^{\wedge K}\big)\\
\displaystyle\hocolim_{\wr St(-)^K}St(X)^{\wedge K}\ar[ur]_-{p_\ast}
}.\]
The vertical map is an equivalence by the twisted Fubini theorem of \cite[2.26]{Gdiags}. The diagonal map is an equivalence by equivariant cofinality, if we can show that the categories ${\underline{W}}/p$ are $G_{\underline{W}}$-contractible for every object $\underline{W}$ of $St(K\times G_+)^K$. We claim that the pair $(\bigcup_{k\in K}W_k,\underline{W})$ defines a $G_{\underline{W}}$-invariant initial object in ${\underline{W}}/p$. For this to be a well-defined object we need the union $\bigcup_{k\in K}W_k$ to be a proper subset of $K\times G_+$. Each $W_k$ is a proper subset of $o_+$, for some orbit $o\in (K\times G)/G$. Since $K\times G$ has only $|K|$-orbits we cannot cover $K\times G$ with the $W_k$'s.
\end{proof}

\subsection{The equivariant Goodwillie tree}\label{sec:tree}

We combine the $J$-excisive approximations of a homotopy functor $\Phi\colon \mathscr{C}\to \mathscr{D}^G$ constructed in \S\ref{sec:ecxapprox} into a diagram analogous to the Taylor tower of classical functor calculus.

Observe that the universal property of the $J$-excisive approximation of \ref{exisiveapprox} guarantees that, if there exist a map of finite $G$-sets $\alpha\colon K\to J$ which is injective on orbits, there is an essentially unique weak map $P_J\Phi\to P_K\Phi$. This is because the functor $P_K\Phi$ is $J$-excisive by \ref{injonorb}. It is going to be convenient to construct an actual natural transformation $\alpha^\ast\colon P_J\Phi\to P_K\Phi$ instead of working in the homotopy category. The map $\alpha^\ast$ is induced on sequential homotopy colimits by the following composite
\[T_J\Phi(c)=\holim_{\mathcal{P}_0(J_+)}\Phi\big(\Lambda^{J}(c)\big)\stackrel{\alpha^\ast}{\longrightarrow} \holim_{\mathcal{P}_0(K_+)}\alpha^{\ast}\Phi\big(\Lambda^{J}(c)\big)\stackrel{\eta_\alpha}{\longrightarrow} \holim_{\mathcal{P}_0(K_+)}\Phi\big(\Lambda^{K}(c)\big)=T_K\Phi(c).\]
The first map is induced on homotopy limits by the image functor $\alpha\colon \mathcal{P}_0(K_+)\to \mathcal{P}_0(J_+)$. The second map is induced by the map of punctured $K_+$-cubes $\eta_\alpha\colon\alpha^{\ast}\Lambda^{J}(c)\to\Lambda^{K}(c)$ defined at a non-empty subset $U$ of $K_+$ by the canonical map on homotopy colimits
\[\Lambda^{J}(c)_{\alpha(U)}=\hocolim_{St(\alpha(U))}St^ {\alpha(U)}(c)\longrightarrow \hocolim_{St(U)}St^ {U}(c)=\Lambda^{K}(c)\]
induced by the functor $\alpha^{-1}(-)\cap U\colon St(\alpha(U))\to St(U)$. This functor is well defined because we are assuming that $\alpha$ is injective on orbits. It sends a subset $V\subset \alpha(U)\cap o_+$, where $o$ is an orbit in $J/G$, to the subset $\alpha^{-1}(V)\cap U$ of
\[\alpha^{-1}(\alpha(U)\cap o_+)\cap U=\alpha^{-1}(\alpha(U))\cap U\cap \alpha^{-1}(o_+)=U\cap w_+\]
where $w$ is the unique orbit which is sent to $o$ by $\alpha$.

\begin{rem}
By the universal property of the excisive approximation \ref{exisiveapprox} any two $G$-maps $\alpha,\beta\colon K\to J$ injective on orbits induce maps $\alpha^\ast,\beta^\ast\colon P_J\Phi\to P_K\Phi$ which are equivalent in the homotopy category. An explicit zig-zag of equivalence is provided by the commutative diagram
\[\xymatrix@C=40pt{P_J\Phi\ar[d]_-{\alpha^\ast}\ar[r]_-{\simeq}^-{p_JP_J\Phi}&
P_JP_J\Phi\ar[d]_-{P_J\alpha^\ast}&P_J\Phi\ar[d]^{P_Jp_K\Phi}\ar[l]_-{P_Jp_J\Phi}^-{\simeq}
\ar[r]^-{P_Jp_J\Phi}_-{\simeq}&P_JP_J\Phi\ar[d]^-{P_J\beta^\ast}&P_J\Phi
\ar[l]_-{p_JP_J\Phi}^-{\simeq}\ar[d]^{\beta^\ast}
\\
P_K\Phi\ar[r]^-{\simeq}_-{p_JP_K\Phi}&P_JP_K\Phi\ar@{=}[r]&P_JP_K\Phi\ar@{=}[r]&P_JP_K\Phi&P_K\Phi\ar[l]_-{\simeq}^-{p_JP_K\Phi}
}\]
from the proof of \cite[1.8]{calcIII}. The horizontal arrows are equivalences because both $P_J\Phi$ and $P_K\Phi$ are $J$-excisive.
\end{rem}

The maps $\alpha^\ast\colon P_J\Phi\to P_K\Phi$ define a contravariant diagram of functors indexed on the category of finite $G$-sets and $G$-maps that are injective on orbits. By the Remark above several of the maps in this diagram induce the same map in the homotopy category. We remove this redundancy by letting $\mathcal{F}^{[inj]}_{G}$ be the poset of finite $G$-sets, ordered by the relation $K\leq J$ if there is a $G$-map $K\to J$ that descends to an injective map $K/G\to J/G$. The various $J$-excisive approximations of $\Phi$ and the maps $\alpha^\ast$ define a diagram
\[(\mathcal{F}^{[inj]}_{G})^{op}\longrightarrow Fun_h(\mathscr{C},\mathscr{D}^G)\]
to the category of homotopy functors, which is commutative in the homotopy category. 

\begin{rem}
When we say that the diagram commutes in the homotopy category of homotopy functors we mean that for every $H$-object $c\in\mathscr{C}^H$ the diagram 
\[(P_{(-)}\Phi)(c)\colon (\mathcal{F}^{[inj]}_{G})^{op}\longrightarrow h\mathscr{D}^H\]
commutes in the homotopy category of $\mathscr{D}^H$, and that the zig-zags are natural in $c$. If one can localize the equivalences of $Fun_h(\mathscr{C},\mathscr{D}^G)$ this would give an actually commutative diagram in the homotopy category of $Fun_h(\mathscr{C},\mathscr{D}^G)$. This localization is constructed classically in \cite{BCR} and \cite{BR} by defining a model structure on the category of homotopy functors $\mathscr{C}\to\mathscr{D}$ which satisfy a certain accessibility condition.
\end{rem}

We observe that every map $\alpha\colon K\to J$ which is injective on orbits decomposes canonically as the composition of a $G$-map which is an isomorphism on orbits followed by an inclusion (through the image of $\alpha$). Conceptually, we want to distinguish these two kinds of maps. We will use the inclusions for taking homotopy limits and to talk about convergence, and we will use the isomorphisms on orbits to compare the convergence of different collections of inclusions. Let $\mathcal{F}^{incl}_{G}$ be the subposet of $\mathcal{F}^{[inj]}_{G}$ where the unique arrow $K\leq J$ belongs to $\mathcal{F}^{incl}_{G}$ if $K$ is a $G$-subset of $J$.

\begin{defn}
The Goodwillie tree of a homotopy functor $\Phi\colon \mathscr{C}\to\mathscr{D}^G$ is the functor 
\[\mathcal{T}_G\Phi\colon (\mathcal{F}_{G}^{incl})^{op}\longrightarrow Fun_h(\mathscr{C},\mathscr{D}^G)\]
which sends $J$ to $P_J\Phi$ and an inclusion of $G$-sets $\iota\colon K\to J$ to the map $\iota^{\ast}\colon P_J\Phi\to P_K\Phi$. We observe that the Goodwillie tower of $\Phi$ is the subdiagram of $\mathcal{T}_G\Phi$ on the sets $J=\underline{n}$ with the trivial $G$-action, where $n$ runs through the natural numbers.
\end{defn}

Let us discuss the role of the $G$-maps $K\to J$ that induce isomorphisms on orbits. Let $\underline{S}=\{S_{n}\}_{n\geq 1}$ be a sequence of transitive $G$-sets, and let $\iota_{\underline{S}}\colon\mathbb{N}\to \mathcal{F}_{G}^{incl}$ be the corresponding inclusion, that sends an integer $m$ to the disjoint union
\[\iota_{\underline{S}}(m)=\coprod_{n\leq m}S_n.\]
Here $\mathbb{N}$ is the poset of natural numbers with the standard order and the maps $\iota_{\underline{S}}(m)\to \iota_{\underline{S}}(k)$, for $m\leq k$, are the inclusions of coproduct summands.
We define $\mathcal{T}_G\Phi|_{\underline{S}}$ to be the tower
\[\mathcal{T}_G\Phi|_{\underline{S}}\colon\mathbb{N}^{op}\stackrel{\iota^{op}_{\underline{S}}}{\longrightarrow} (\mathcal{F}_{G}^{incl})^{op}\longrightarrow Fun_h(\mathscr{C},\mathscr{D}^G)\]
obtained by restricting the Goodwillie tree of $\Phi$ along $\iota^{op}_{\underline{S}}$. A collection of $G$-maps $\underline{\alpha}=\{\alpha_n\colon S_n\to R_n\}$ between two sequences of transitive $G$-sets $\underline{S}$ and $\underline{R}$ induces a map of towers 
\[\underline{\alpha}^{\ast}\colon \mathcal{T}_G\Phi|_{\underline{R}}\longrightarrow \mathcal{T}_G\Phi|_{\underline{S}}.\]
This is because the map $\amalg_{n\leq m}\alpha_n\colon \iota_{\underline{S}}(m)\to \iota_{\underline{R}}(m)$ is an isomorphism on orbits for every $m\in\mathbb{N}$. The map $\underline{\alpha}^\ast$ further induces a map on the homotopy limits of the towers
\[\holim_{\mathbb{N}^{op}}\mathcal{T}_G\Phi|_{\underline{R}}\longrightarrow \holim_{\mathbb{N}^{op}}\mathcal{T}_G\Phi|_{\underline{S}}.\]
We will show in \ref{convergence} that for certain functors (for example the identity on pointed $G$-spaces) this map is in fact a $G$-equivalence.
\begin{ex}
The projection $G\to 1$ induces a map from the constant sequence $\underline{G}$ with value $G$ to the constant sequence $\underline{1}$ with value $\{1\}$. This gives a natural transformation
\[\holim_{\mathbb{N}^{op}}P_n\Phi=\holim_{\mathbb{N}^{op}}\mathcal{T}_G\Phi|_{\underline{1}}\longrightarrow\holim_{\mathbb{N}^{op}}\mathcal{T}\Phi|_{\underline{G}}=\holim_{\mathbb{N}^{op}}P_{n\times G}\Phi\]
from the limit of the standard Goodwillie tower of $\Phi$ to the limit of the tower on free finite $G$-sets.
\end{ex}

\begin{ex}
Let $I\colon \Top_{\ast}\to \Top_{\ast}^G$ be the inclusion of pointed spaces with the trivial $G$-action (whose canonical extension to $\Top_{\ast}^G$ is the identity functor). We recall from \ref{ex:transapprox} that for every transitive $G$-set $T$ the $T$-excisive approximation of $I$ is naturally equivalent to
\[P_TI(X)\simeq \Omega^{\infty T}\Sigma^{\infty T}X:=\hocolim_k\Omega^{k T}\Sigma^{k T}X\]
where $\Omega^{k T}$ and $\Sigma^{k T}$ are the loop and suspension operators for the permutation representation sphere $\mathbb{R}[kT]^+$.
The map $P_1 I\to P_{G}I $ induced by the projection $G\to 1$ is the standard map
\[P_1I(X)\simeq
\Omega^\infty\Sigma^{\infty}X\longrightarrow 
\Omega^{\infty G}\Sigma^{\infty G}X\simeq P_{G}I(X)\]
from the na\"{i}ve to the genuine stable equivariant homotopy of $X$, induced by suspending with the reduced regular representation of $G$. This map has a retraction on fixed points, and $(P_{1}I(X))^G$ splits out of $(P_{G}I(X))^G$ as a summand, forming part of the Tom Dieck-splitting. In \S\ref{tomdieckspl} we will assemble the maps $P_S I\to P_{T}I $ induced by the $G$-maps $T\to S$ of transitive $G$-sets into the full Tom Dieck-splitting. Our technique generalizes to all the $G$-sets with a prescribed number of orbits and to a more general class of functors. 
\end{ex}

\subsection{Change of group}\label{sec:changegroup}

We study the behavior of the equivariant excisive approximations under the restriction of the $G$-action to a subgroup $H\leq G$. Let $\iota_{H}\colon H\to G$ be a subgroup inclusion, and let $\iota_{H}^\ast\colon \mathscr{D}^G\to \mathscr{D}^H$ be the corresponding restriction functor. Given a homotopy functor $\Phi\colon \mathscr{C}\to \mathscr{D}^G$ we write $\Phi|_H$ for the composite
\[\Phi|_H\colon \mathscr{C}\stackrel{\Phi}{\longrightarrow}\mathscr{D}^G\stackrel{\iota_{H}^\ast}{\longrightarrow} \mathscr{D}^H.\]
Notice that a $G$-model category restricts canonically to an $H$-model category by considering only the categories $\mathscr{C}^L$ for subgroups $L\leq H$, and that $\Phi|_H$ is still a homotopy functor. Given a finite $G$-set $J$ we use the notation $J=\coprod_{o\in J/G}o$ for the canonical decomposition into $G$-orbits, and we write
\[J=\coprod_{o\in J/G}\coprod_{w\in o/H}w\]
for its further decomposition into $H$-orbits. An element $\underline{w}$ of the set $\prod_{o\in J/G}o/H$ is a choice of $H$-orbit $w_o\in o/H$ for every $G$-orbit $o\in J/G$. We will denote the union of its components by $\amalg\underline{w}=\amalg_{o\in J/G}w_o$, which is an $H$-subset of $J$ that intersects every $G$-orbit. Given a finite collection of $H$-sets $\{K_i\}_{i\in I}$, we denote by 
\[\bigcirc_{i\in I}P_{K_i}(\Phi|_H):=P_{K_{i_1}}\dots P_{K_{i_k}}(\Phi|_H)\]
the iteration of the $K_i$-excisive approximations, in any chosen order (different orders give equivalent functors by \ref{PJcommute}).

\begin{theorem}\label{restriction}
Let $J$ be a finite $G$-set and let $\Phi\colon\mathscr{C}\to \mathscr{D}^G$ be a homotopy functor. For every subgroup $H\leq G$ there is a natural zig-zag of equivalences under $\Phi|_H$
\[(P_J\Phi)|_H\simeq  \iter_{\underline{w}\in\!\!\!\prod\limits_{o\in J/G}\!\!\! o/H}P_{\amalg\underline{w}}(\Phi|_H).\]
\end{theorem}

\begin{rem} Theorem \ref{restriction} specializes to the following equivalences:
\begin{enumerate}
\item If $H$ is the trivial group, there is a natural equivalence
\[(P_J\Phi)|_{1}\simeq \iter_{\underline{w}\in\!\!\!\prod\limits_{o\in J/G}\!\!\! o}P_{|J/G|}(\Phi|_1)\simeq P_{|J/G|}(\Phi|_1)\]
that is $P_J\Phi$ and $P_{J/G}\Phi$ have the same non-equivariant homotopy type.
\item If $J=nG$ is a free $G$-set, there is a natural equivalence
\[(P_{nG}\Phi)|_H\simeq  \iter_{\substack{\underline{w}\in\prod\limits_{i=1}^n G/H}}P_{nH}(\Phi|_H)\simeq P_{nH}(\Phi|_H).\]
In particular if $\Phi$ is $nG$-excisive $\Phi|_H$ is $nH$-excisive.
\item
If $J=T$ is a transitive $G$-set with $H$-orbits decomposition $T=T_1\amalg\dots\amalg T_n$, there is an equivalence
\[(P_T\Phi)|_H\simeq  P_{T_1}\dots P_{T_n}(\Phi|_H).\]
If $\mathscr{C}$ and $\mathscr{D}$ are pointed and $\Phi$ is reduced this corresponds to the equivalence
\[\Omega^{\infty T|_H}\Phi(\Sigma^{\infty T|_H}c)\simeq \Omega^{\infty T_1}\dots \Omega^{\infty T_n}\Phi(\Sigma^{\infty T_1}\dots\Sigma^{\infty T_n}c)\]
(cf. \ref{ex:transapprox}).
\end{enumerate}
\end{rem}

The proof of Theorem \ref{restriction} will occupy the rest of the section.

\begin{lemma}\label{blowupcubes}
Let $J$ and $K$ be finite $G$-sets, and let $p\colon\mathcal{P}(K)\to \mathcal{P}(J)$ be a $G$-equivariant functor with the following properties:
\begin{enumerate}
\item $p(U)=\emptyset$ if and only if $U=\emptyset$,
\item $p(U\cup V)=p(U)\cup p(V)$ for every $U,V\subset K$,
\item For every $j\in J$ there is $k\in K$ such that $p(\{k\})=\{j\}$.
\end{enumerate}
Then the functor $p^\ast\colon C^{\mathcal{P}(J)}_a\to C^{\mathcal{P}(K)}_a$ preserves and detects cartesian cubes, and it preserves cocartesian cubes. 
\end{lemma}

\begin{proof}
This lemma is a slight generalization of \cite[A.1]{Gdiags} and \cite[A.3]{Gdiags}, where this result is proved when $p\colon\mathcal{P}(K)\to \mathcal{P}(J)$ is the image functor of a surjective $G$-map $K\to J$. The difference in our case is that $p$ might send one-point sets to sets which have more than one element. The proof of \ref{blowupcubes} is identical to \cite[A.1]{Gdiags} and \cite[A.3]{Gdiags}, by replacing the preimage of a subset $W\subset J$ by a map $K\to J$ with the ``preimage'' of $W$ by the functor $p$, which is defined as
\[p^{-1}(W):=\{k\in K\ | \ p(\{k\})\subset W\}.\]
The fundamental properties satisfied by $p^{-1}(-)$ needed in the proof are the following:
\begin{itemize}
\item $p^{-1}(W)\cup p^{-1}(Z)\subset p^{-1}(W\cup Z)$
\item $p p^{-1}(W)=W$,
\item $U\subset p^{-1}p(U)$.
\end{itemize}
\end{proof}

\begin{prop}\label{excres}
Let $\underline{w}\in \prod_{o}o/H$ be a choice of $H$-orbit of $o$ for every $G$-orbit $o\in J/G$. If $\Phi\colon\mathscr{C}\to \mathscr{D}^G$ is $J$-excisive, the restriction $\Phi|_H\colon \mathscr{C}\to \mathscr{D}^H$ is $\amalg\underline{w}$-excisive.
\end{prop}

\begin{proof}
Let us fix $\underline{w}\in \prod_{o}o/H$. We need to show that for every subgroup $L$ of $H$ the functor $\Phi|_H$ sends $H$-strongly cocartesian $\amalg\underline{w}_+|_L$-cubes in $\mathscr{C}$ to cartesian cubes in $\mathscr{D}$. We prove this by turning $\amalg\underline{w}_+$-cubes into $J_+|_H$-cubes by means of a certain functor $p\colon \mathcal{P}(J_+|_H)\to \mathcal{P}(\amalg\underline{w}_+)$ which satisfies the properties of Lemma \ref{blowupcubes}.

If the $G$-orbits of $J$ where of the form $o=G/K$ for normal subgroups $K$ of $G$, one can define an $H$-equivariant retraction $p\colon J_+|_H\to \amalg\underline{w}_+$ by sending each $H$-orbit $w\in o/H$ isomorphically to $w_o$ (the $H$-set $G/K$ decomposes into isomorphic $H$-orbits if $K$ is normal in $G$). In general such a retraction does not exist at the level of sets, but it is still possible to define a functor $p\colon \mathcal{P}(J_+|_H)\to \mathcal{P}(\amalg\underline{w}_+)$ which is a retraction for the inclusion functor, and which satisfies the hypotheses of \ref{blowupcubes}. It is defined as follows.
For every $o\in J/G$ choose a point $a_o\in w_o$, and for every $v\in o/H$ different than $w_o$, choose an $a_v\in v$. For every $v\in o/H$ define a functor $f_v\colon \mathcal{P}(v)\to \mathcal{P}(w_o)$ by
\[f_v(U)=\{h\cdot a_o\ | \ h\in H, \ h\cdot a_v\in U\}.\]
This functor is $H$-equivariant, and it defines a retraction \[p\colon \mathcal{P}(J_+|_H)= \mathcal{P}\big((\coprod_{o\in J/G}\coprod_{v\in o/H}v)_+\big)\cong \mathcal{P}(+)\times\prod_{o\in J/G}\prod_{v\in o/H}\mathcal{P}(v)\longrightarrow \mathcal{P}(\amalg\underline{w}_+)\]
where the last map is the union of the maps $f_w$ for $w\neq w_o$, and of the identity otherwise. Explicitly it is defined by the formula
\[p(U)=(+\cap U)\amalg \coprod_{o\in J/G}\big( (U\cap w_o)\amalg\coprod_{\substack{v\in o/H\\ v\neq w_o}}f_v(U\cap v)\big).\]
This is an $H$-equivariant functor which satisfies the conditions of Lemma \ref{blowupcubes}.

Let us show that if $X\in \mathscr{C}^{\mathcal{P}(\amalg\underline{w}_+|_L)}_a$  is $H$-strongly cocartesian, so is $p^\ast X\in \mathscr{C}^{\mathcal{P}(J_+|_L)}_a$. For any subset $U\subset J_+$ which is not contained in $St(J_+)=\bigcup_{o}\mathcal{P}_{1}(o_+)$ we need to show that $(p^{\ast}X)|_{\mathcal{P}(U)}=p^{\ast}(X|_{\mathcal{P}(p(U))})$ is cocartesian. Since $p^{\ast}$ preserves cocartesian cubes it is sufficient to show that $p(U)$ does not belong to 
\[St(\amalg\underline{w}_+)=\bigcup_{z\in (\amalg\underline{w})/H}\mathcal{P}_{1}(\amalg\underline{w}\cap z_+)=\bigcup_{o\in J/G}\mathcal{P}_{1}((w_o)_+)\]
where the last equality holds because each $w_o$ is an $H$-orbit of $\amalg\underline{w}$.
Suppose for the sake of contradiction that $p(U)$ belongs to $St(\amalg\underline{w}_+)$. Then there is a $G$-orbit $o\in J/G$ such that $p(U)$ is a proper subset of $(w_o)_+$. By the definition of $p$ this forces $U$ to be a proper subset of $o_+$, which is absurd since $U$ does not belong to $St(J_+)$.

Let us finally prove that $\Phi|_H$ is $\amalg\underline{w}$-excisive. Since $\Phi$ is $J$-excisive, for every $H$-strongly cocartesian $X\in \mathscr{C}^{\mathcal{P}(\amalg\underline{w}_+|_L)}_a$ the $J_+|_L$-cube 
\[\Phi(p^\ast X)=\Phi|_H(p^\ast X)=p^\ast \Phi|_H(X)\] is cartesian. Thus $\Phi|_H(X)$ is cartesian by \ref{blowupcubes}.
\end{proof}

\begin{lemma} There is a weak natural transformation under $\Phi|_H$
\[F_\Phi\colon (T_J\Phi)|_H\longrightarrow  \iter_{\substack{\underline{w}\in\!\!\!\prod\limits_{o\in J/G}\!\!\! o/H}}T_{\amalg\underline{w}}(\Phi|_H)\]
which is natural in $\Phi$.
\end{lemma}

\begin{proof}
Since the target of this map is an iterated homotopy limit, we need to produce a weak natural transformation
\[F_\Phi\colon (\holim_{\mathcal{P}_0(J_+)}\Phi\big(\Lambda^J\big))|_H=
\holim_{\mathcal{P}_0(J_+|_H)}\Phi|_H\big(\Lambda^J|_H\big)
\longrightarrow  \holim_{\prod\limits_{\underline{w}\in\!\!\!\prod\limits_{o\in J/G}\!\!\! o/H}\mathcal{P}_0(\amalg\underline{w}_+)}\Phi|_H\big(\iter_{\underline{w}}\Lambda^{\amalg\underline{w}}\big)\]
where the $\iter_{\underline{w}}$ in the target denotes the iteration of the diagrams $\Lambda^{\amalg\underline{w}}(-)$.
Define a functor $\pi\colon \prod\limits_{\underline{w}\in \prod_o o/H}\mathcal{P}_0(\amalg\underline{w}_+)\to \mathcal{P}_0(J_+|_H)$ by sending $\underline{A}=\{A_{\underline{w}}\in\mathcal{P}_0(\amalg\underline{w}_+)\}_{\underline{w}}$ to
\[\pi(\underline{A})=\left\{\begin{array}{ll}\bigcup_{\underline{w}}A_{\underline{w}}& \mbox{if}\ +\in \bigcap_{\underline{w}}A_{\underline{w}}\\
\big(\bigcup_{\underline{w}}A_{\underline{w}}\big)\backslash +& \mbox{otherwise}
\end{array}\right.\]
Notice that this functor is surjective on objects. 
There is a natural zig-zag of functors
\[St(J_+)|_H\stackrel{\gamma}{\longleftarrow} \Big(\coprod_{o\in J/G}\prod_{w\in o/H}\mathcal{P}_{1}(w_+)\Big)/_{\underline{\emptyset}\to\underline{+}}
\stackrel{\delta}{\longrightarrow}\prod\limits_{\underline{w}\in\!\!\!\prod\limits_{o\in J/G}\!\!\! o/H}St(\amalg\underline{w}_+).\]
The middle category consists of the disjoint union of the posets $\prod_{w}\mathcal{P}_{1}(w_+)$ glued along the morphism from the constant collection $\underline{\emptyset}_w=\emptyset$ to the constant collection $\underline{+}_w=+$.
The left-pointing arrow $\gamma$ is the union of the right-cofinal functors $\prod_{w}\mathcal{P}_{1}(w_+)\to \mathcal{P}_{1}(o_+)$ that take a collection of subset to the union of its components inside $o_+$ (see \cite[A.2-A.3]{Gdiags}), and it is also right-cofinal. The right-pointing arrow $\delta$ sends an object $(v\in J/G,\ \underline{U}=\{U_w\subset w_+\}_{w\in v/H})$ to the collection of sets with $\underline{w}=\{w_o\}_{o\in J/G}$-component
\[(w_{v}, U_{w_{v}}\subset (w_{v})_+)\in St(\amalg\underline{w}_+).\]
This restricts to a natural zig-zags of functors
\[ St(\pi(\underline{A}))\stackrel{\gamma_{\underline{A}}}{\longleftarrow} \Big(\coprod_{o\in J/G}\prod_{w\in o/H}\mathcal{P}(\pi(\underline{A})\cap w_+)\backslash w_+\Big)/_{\underline{\emptyset}\to\underline{+}}\stackrel{\delta_{\underline{A}}}{\longrightarrow} \prod\limits_{\underline{w}\in\!\!\!\prod\limits_{o\in J/G}\!\!\! o/H}St(A_{\underline{w}})\]
for every $\underline{A}$, where $\gamma_{\underline{A}}$ is right-cofinal. Let us denote the middle poset by $St(\underline{A})$. Observe that for any object $c\in \mathscr{C}$ there is an equality of diagrams $St(\underline{A})\longrightarrow \mathscr{C}$
\[St^{\pi(\underline{A})}(c)\circ\gamma_{\underline{A}}=
\big(\iter_{\underline{w}}St^{A_{\underline{w}}}(c)\big)\circ\delta_{\underline{A}}
\ \ .\]
On homotopy colimits this gives a zig-zag of compatible natural transformations
\[\xymatrix{\Lambda^J_{\pi(\underline{A})}(c)& \displaystyle\hocolim_{St(\underline{A})}
St^{\pi(\underline{A})}(c)\circ\gamma_{\underline{A}}
\ar[l]_-{(\gamma_{\underline{A}})_{\ast}}^-{\simeq}
\ar[r]^-{(\delta_{\underline{A}})_\ast}
& \iter_{\underline{w}}\Lambda^{\amalg\underline{w}}_{A_{\underline{w}}}(c)}.\]
Let $\Lambda^J_{\underline{A}}(c)$ be the middle homotopy colimit.
We define the map $F_\Phi$ as the composition of the maps on homotopy limits
\[\xymatrix@C=18pt{\displaystyle
F_\Phi\colon(\holim_{\mathcal{P}_0(J_+)}\Phi\big(\Lambda^J\big))|_H\ar[r]^-{\pi^\ast}
&\displaystyle
\holim_{\underline{A}}
\Phi|_H\big(\Lambda_{\pi(\underline{A})}^J\big)
&\displaystyle
\holim_{\underline{A}}\Phi|_H\big(\Lambda^J_{\underline{A}}(c)\big)\ar[r]^-{\delta_{\ast}}
\ar[l]_-{\gamma_\ast}^-\simeq
& \displaystyle
\holim_{\underline{A}}\Phi|_H\big(\iter_{\underline{w}}
\Lambda_{A_{\underline{w}}}^{\amalg\underline{w}}\big)}.\]
\end{proof}

\begin{proof}[Proof of \ref{restriction}]
By naturality of the weak map $F_\Phi$ the diagram
\[\xymatrix{
(T_J\Phi)|_H\ar[rrr]^-{F_\Phi}\ar[d]_{T_J(t_J\Phi)}&&& \iter_{\underline{w}}T_{\amalg\underline{w}}(\Phi|_H)
\ar[d]^-{\iter_{\underline{w}}T_{\amalg\underline{w}}(\circ_{\underline{x}}t_{\underline{x}})}
\ar[dll]_-{\iter_{\underline{w}}T_{\amalg \underline{w}}(t_J\Phi)}
\\
(T^{(2)}_J\Phi)|_H\ar[r]_-{F_{T_J\Phi}}&
\iter_{\underline{w}}T_{\amalg\underline{w}}(T_J\Phi)|_H\ar[rr]_-{\iter_{\underline{w}}T_{\amalg\underline{w}}(F)}
&&\iter_{\underline{w}}T_{\amalg\underline{w}}
\iter_{\underline{x}}T_{\amalg\underline{x}}(\Phi|_H)
}\]
commutes, defining a weak natural transformation $F_{\Phi}\colon(P_J\Phi)|_H\to \iter_{\underline{w}}P_{\amalg\underline{w}}(\Phi|_H)$. We use Proposition \ref{excres} to show that this map is an equivalence. Morally this is because $(P_{J}\Phi)|_H$ is $\amalg\underline{w}$-excisive for every $\underline{w}$, and the universal properties of the $\amalg\underline{w}$-excisive approximations should produce a canonical map $ \iter_{\underline{w}}P_{\amalg\underline{w}}(\Phi|_H)\to (P_{J}\Phi)|_H$ which is an inverse of $F_\Phi$. The be precise, consider the diagram

\[\xymatrix@C=21pt{
(P_{J}\Phi)|_{H}\ar[r]_-{F_\Phi}\ar@/^2pc/[rrr]
&\iter_{\underline{w}}P_{\amalg\underline{w}}(\Phi|_H)
\ar[rr]^-{
\iter_{\underline{w}}P_{\amalg\underline{w}}p_J\Phi}
\ar@/_2pc/[rrrr]_-{\iter_{\underline{w}}P_{\amalg\underline{w}}
\iter_{\underline{x}}t_{\amalg\underline{x}}\Phi}&&
\iter_{\underline{w}}P_{\amalg\underline{w}}(P_J\Phi|_H)
\ar[rr]^-{\iter_{\underline{w}}P_{\amalg\underline{w}}F_\Phi}&&
\iter_{\underline{w}}P_{\amalg\underline{w}}
\iter_{\underline{x}}P_{\amalg\underline{x}}(\Phi|_H)
}.\]
We show that the two curved arrows are equivalences, and by the $2$ out of $6$ property every map in the diagram, including $F_\Phi$, is an equivalence. The lower-right triangle commutes because $F_{\Phi}$ is a map under $\Phi|_{H}$. The lower curved arrow is an equivalence because $\iter_{\underline{w}}P_{\amalg\underline{w}}(P_J\Phi|_H)$ is $\amalg\underline{x}$-excisive for every $\underline{x}$ in $\prod_o o/H$. The upper curved map is an equivalence because it compares via the following zig-zag of equivalences
\[\xymatrix{(P_{J}\Phi)|_{H}\ar[rr]_-{\simeq}^-{p_JP_J\Phi}\ar[dd]&&(P_{J}P_J\Phi)|_{H}\ar[dd]^{F_{P_J\Phi}}&&(P_{J}\Phi)|_{H}\ar[ll]_-{P_Jp_J\Phi}^{\simeq}\ar[d]^{F_{\Phi}}\\
&&&& \iter_{\underline{w}}P_{\amalg\underline{w}}(\Phi|_H)\ar[d]^{\iter_{\underline{w}}P_{\amalg\underline{w}}p_J\Phi}\\
\iter_{\underline{w}}P_{\amalg\underline{w}}(P_J\Phi|_H)\ar@{=}[rr]&&\iter_{\underline{w}}P_{\amalg\underline{w}}(P_J\Phi|_H)&& \iter_{\underline{w}}P_{\amalg\underline{w}}(P_J\Phi|_H)\ar@{=}[ll]
}\]
to the composition of the maps
\[\xymatrix@=10.5pt{
(P_{J}\Phi)|_{H}\ar[rrr]_-{\simeq}^-{p_{\underline{w}^1}(P_{J}\Phi)|_H}&&&
P_{\amalg\underline{w}^1}(P_J\Phi|_H)\ar[rrrr]_-{\simeq}^-{p_{\underline{w}^2}(P_{\underline{w}^1}(P_{J}\Phi)|_H)}&&&&
P_{\amalg\underline{w}^2}P_{\amalg\underline{w}^1}(P_J\Phi|_H)\ar[r]_-{\simeq}&\dots\ar[r]_-{\simeq}&
\displaystyle\iter_{\underline{w}}P_{\amalg\underline{w}}(P_J\Phi|_H)
}\] 
which are equivalences by Proposition \ref{excres}. The right-hand square commutes by naturality of $F_{\Phi}$ in the functor $\Phi$, and the left-hand square commutes by construction.
\end{proof}

\subsection{On the convergence of the Goodwillie tree}\label{sec:convergence}

Let $G$ be a finite group. We discuss the convergence of the Goodwillie tree for a certain class of homotopy functors $\Phi\colon \Top_\ast\to \Top_{\ast}^G$.
%For every subgroup $H$ of $G$ and every pointed $H$-space $X$ the inclusion $\iota_H\colon X^H\to X$ is $H$-equivariant, and thus it defines an $H$-equivariant map $(\iota_H)_\ast\colon \Phi(X^H)\to \Phi(X)$.

\begin{defn}\label{def:split}
We say that a homotopy functor $\Phi\colon \Top_\ast\to \Top_{\ast}^G$ commutes with fixed-points if for every subgroup $H$ of $G$ and every $H$-equivariant map $f\colon X\to Y$ which is an isomorphism on $H$-fixed-points, the map $(f_\ast)^H\colon \Phi(X)^H\to \Phi(Y)^H$ is an isomorphism.
\end{defn}

\begin{ex}\
\begin{enumerate}
\item The inclusion of $H$-fixed-points $X^H\to X$ is an isomorphism on $H$-fixed-points for every pointed $H$-space $X$. Thus if $\Phi$ commutes with fixed-points $\Phi(X^H)^H\cong \Phi(X)^H$.
\item A $sSet$-enriched homotopy functor $\Phi\colon \Top_\ast\to \Top_{\ast}^G$ determines an orthogonal $G$-spectrum $\Phi(\mathbb{S})$ with $n$-th level the $G$-space $\Phi(\mathbb{S})_n=\Phi(S^n)$ (in fact the finitary ones define a model category Quillen equivalent to orthogonal $G$-spectra, by \cite{Blumberg}). The geometric, genuine and na\"{i}ve $G$-fixed points of  $\Phi(\mathbb{S})$ are respectively defined in degree $n$ by the spaces $\Phi(S^{nG})^G$, $(\Omega^{n\overline{G}}\Phi(S^{nG}))^G$ and $\Phi(S^n)^G$, and they fit into a commutative diagram
\[\xymatrix{
\Phi(S^{nG})^G&&\Phi(S^n)^G\ar[ll]_{(\iota_\ast)^G}\ar[dl]^-{\eta}\\
&(\Omega^{n\overline{G}}\Phi(S^{nG}))^G\ar[ul]_-{ev}
}\]
where $\iota\colon S^n=(S^{nG})^G\to S^{nG}$ is the inclusion, $\eta$ is the adjoint assembly map and $ev$ is the evaluation at $0\in S^{n\overline{G}}$. Observe that $\iota$ is an isomorphism on $G$-fixed-points, and if $\Phi$ commutes fixed-points the map $(\iota_\ast)^G$ is an isomorphism. Thus the evaluation map from the genuine to the geometric $G$-fixed points of $\Phi(\mathbb{S})$ is a split epimorphism. The $G$-spectra with this property are called geometrically split in \cite{Lewis}, and they are the spectra that admit a Tom Dieck-splitting. The relationship with these splittings will be further investigated in \S\ref{tomdieckspl}.

\item Let $\Phi\colon \Top_\ast\to \Top_{\ast}^G$ be a homotopy functor whose image lies in the subcategory of spaces with the trivial $G$-action. Suppose that $\Phi\colon \Top_\ast\to \Top_{\ast}$ preserves limits with shape $H$, for every subgroup $H$ of $G$ (that is $\Phi$ ``commutes with fixed-points''). Then if $f\colon X\to Y$ is an isomorphism on $H$-fixed-points, the induced map
\[(f_\ast)^{H}\colon \Phi(X)^H\cong\Phi(X^H)\stackrel{(f^H)_\ast}{\longrightarrow} \Phi(Y^H)\cong \Phi(Y)^H\]
is an isomorphism, and $\Phi$ commutes with fixed-points in the sense of definition \ref{def:split}.

Examples of such functors include the inclusion $I\colon \Top_\ast\to \Top_{\ast}^G$ of spaces with the trivial $G$-action (which extends to the identity functor on $\Top_{\ast}^G$), the ``na\"{i}ve'' algebraic $K$-theory of spaces functor $A(-)$ of \ref{ex:htpyfctrs}(iii), and the algebraic $K$-theory functor $\widetilde{K}(R\ltimes M(-))$ of \cite{DM} associated to a ring $R$ and an $R$-bimodule $M$.

\item Let $K$ be a $G$-CW-complex. The homotopy functor $K\wedge (-)\colon \Top_\ast\to \Top_{\ast}^G$ commutes with fixed-points. If $f\colon X\to Y$ is an isomorphism on $H$-fixed-points the map
\[(f_\ast)^{H}\colon (K\wedge X)^H\cong K^H\wedge X^H\stackrel{K^H\wedge f^H}{\xrightarrow{\hspace*{2cm}}}K^H\wedge Y^H\cong (K\wedge Y)^H\]
is an isomorphism.
\end{enumerate}
\end{ex}
%We restrict our attention to homotopy functors $\Phi\colon Top_\ast\to Top_{\ast}^G$ on the $G$-category of pointed $G$-spaces.
%We say that a $\Phi$ commutes with fixed points if the following diagram commutes
%\[\xymatrix{Top_{\ast}^G\ar[r]^-{\Phi}\ar[d]_{(-)^H}&Top_{\ast}^G\ar[d]^{(-)^H}\\
%Top_{\ast}\ar[r]_\Phi& Top_{\ast}
%}\]
%for every subgroup $H$ of $G$. In particular by taking $H=1$ we see that $\Phi\colon Top_{\ast}\to Top_{\ast}^G$ takes values in spaces with the trivial $G$-action. Examples of such functors include the identity functor on pointed spaces, which extends to the identity functor on pointed $G$-spaces (see \ref{sec:Id}).
A proper filtration of finite $G$-sets is an infinite countable sequence of proper inclusions of finite $G$-sets
\[\underline{J}=\big(J_1\subset \dots\subset J_n\subset J_{n+1}\subset\dots\big).\]
A map of proper filtrations $\underline{f}\colon \underline{K}\to \underline{J}$ is a sequence of $G$-equivariant maps $f_n\colon K_n\to J_n$ which commute with the respective inclusions.

\begin{theorem}\label{convergence}
Let $\Phi\colon \Top_\ast\to \Top_{\ast}^G$ be a homotopy functor which commutes with fixed-points, and let $\underline{f}\colon \underline{K}\to \underline{J}$ be a map of proper filtrations such that $f_n\colon K_n\to J_n$ is a bijection on $G$-orbits, for every $n$. Then the maps $f_{n}^\ast\colon P_{J_n}\Phi\to P_{K_n}\Phi$ induce a $G$-equivalence on homotopy limits
\[\underline{f}^\ast\colon\holim_{n}P_{J_n}\Phi\stackrel{\simeq}{\longrightarrow} \holim_{n}P_{K_n}\Phi.\]
In particular there is a $G$-equivalence $\holim_{n}P_{n}\Phi\to \holim_{n}P_{nG}\Phi$.
\end{theorem}

\begin{rem}\label{Carlssonthing}
This result might seems surprising, but for the identity functor on pointed spaces it is closely related to a result of Carlsson about $\mathbb{Z}$-completion which plays a crucial role in the proof of Sullivan's conjecture (\cite{Carlsson}). Let $Q=\Omega^\infty\Sigma^\infty\colon \Top_\ast\to \Top_\ast$ be the stable homotopy functor. This functor has the structure of a monad, and the totalization of the corresponding cosimplicial space $Q^ \bullet$ is equivalent to the limit of the Goodwillie tower of the identity functor on pointed spaces
\[\holim_n P_n\id(X)\simeq Tot Q^\bullet(X).\]
Equivariantly there are two similar constructions. The first is the canonical extension of $Q$ to pointed $G$-spaces $Q=\Omega^\infty\Sigma^\infty\colon \Top^{G}_\ast\to \Top^{G}_\ast$, which is the ``na\"{i}ve'' stable homotopy functor. The second is the ``genuine'' stable homotopy functor $Q_G=\Omega^{\infty G}\Sigma^{\infty G}\colon \Top^{G}_\ast\to \Top^{G}_\ast$ obtained by stabilizing with respect to the regular representation of $G$. An easy fixed-points argument shows that for every pointed $G$-space $X$ there is a $G$-equivalence $\holim_n P_nI(X)\simeq Tot Q^\bullet(X)$. Carlsson proves in \cite[II.7]{Carlsson} that there is a natural $G$-equivalence $Tot Q^\bullet(X)\simeq Tot Q_{G}^\bullet(X)$. The equivalence of \ref{convergence}
\[\holim_{n}P_nI\stackrel{\simeq}{\longrightarrow}\holim_{n}P_{nG}I\]
for the inclusion $I\colon \Top_\ast\to \Top_{\ast}^G$ is a calculus analogue of Carlsson result. It follows that the pointed $G$-spaces $\holim_n P_nI(X)$, $\holim_n P_{nG}I(X)$, $Tot Q^\bullet(X)$ and $Tot Q_{G}^\bullet(X)$ are all equivalent.
\end{rem}

\begin{cor}\label{convanalytic}
Let $\Phi\colon \Top_\ast\to \Top^{G}_{\ast}$ be a homotopy functor which commutes with fixed-points, and such that for every subgroup $H$ of $G$ the composite
\[\Top_\ast\stackrel{\Phi}{\longrightarrow} \Top^{G}_{\ast}\stackrel{(-)^H}{\xrightarrow{\hspace*{1.2cm}}}\Top_\ast\]
is $\rho_H$-analytic in the sense of \cite[4.2]{calcII}, for some integer $\rho_H$. Then for every proper filtration of $G$-sets $\underline{J}$
the map $\Phi(X)\to \holim_{n}P_{J_n}\Phi(X)$
is a $G$-equivalence for every $G$-space $X$ such that $X^H$ is $\rho_H$-connected for every subgroup $H$ of $G$.
\end{cor}
\begin{proof}
Since $\Phi$ commutes with fixed-points, the inclusion $X^H\to X$ induces an isomorphism
\[(P_n\Phi(X))^H\cong P_n\Phi(X)^H\stackrel{\cong}{\longleftarrow} P_n\Phi(X^H)^H\]
for every natural number $n$. By the convergence Theorem \cite[1.13]{calcIII}
\[\Phi(X)^H\cong \Phi(X^H)^H\longrightarrow \holim_{n}P_{n}\Phi(X^H)^H\cong(\holim_{n}P_{n}\Phi(X))^H\]
is an equivalence if $X^H$ is $\rho_H$-connected. By \ref{convergence} the map of filtrations $f_n\colon J_n\to J_n/G$ induces an equivalence between 
the right-hand side and the $H$-fixed-points of $\holim_{n}P_{J_n}\Phi(X)$.
\end{proof}

\begin{rem}
It is unclear to the author if a result similar to \ref{convanalytic} holds for a general homotopy functor $\Phi\colon \Top_\ast\to \Top_{\ast}^G$ which satisfies a suitable equivariant version of analyticity (similar to the one suggested in \cite{Gcalc} or \cite{GdiagTop}).
\end{rem}

Our proof of \ref{convergence} is close in spirit to the proof of \cite[II.7]{Carlsson}, and it is based on the existence of a natural retraction
\[r\colon (P_{J}\Phi)^G\longrightarrow P_{J/G}\Phi^G\]
for the restriction on $G$-fixed points of the map $p^\ast\colon P_{J/G}\Phi\to P_{J}\Phi$ from \S\ref{sec:tree}, induced by the projection $p\colon J\to J/G$ (notice that $P_{J/G}$ commutes with fixed-points, because the $G$-action on $J/G$ is trivial). We recall that the $G$-fixed-points category of $I=\mathcal{P}(J_+)$ is canonically isomorphic to $\mathcal{P}(J/G_+)$. Hence the fixed points diagram $X^G$ of a $J_+$-cube $X$ is a $J/G_+$-cube (see \ref{ex:fixedcube}).
The retraction $r\colon (P_{J}\Phi)^G\to P_{J/G}\Phi^G$ is induced on homotopy colimits by the composition
\[
\xymatrix{
(T^{(k)}_{J}\Phi(c))^G\ar@{-->}[dr]\ar@{=}[r]&\big(\holim_{\mathcal{P}_0(J_+)^k}\Phi\big(\Lambda^{J,k}(c)\big)\big)^G\ar[r]& \holim_{\mathcal{P}_0(J/G_+)^k}\Phi\big(\Lambda^{J,k}(c)\big)^G\\
&T_{J/G}^{(k)}\Phi(c)^G\ar@{=}[r]& \holim_{\mathcal{P}_0(J/G_+)^k}\Phi\big(\Lambda^{J/G,k}(c)\big)^G\ar[u]_{(\eta_p)^{G}_\ast}^{\cong}
}\]
where $\Lambda^{J,k}(c)$ denotes the $k$-fold iteration  $\Lambda^{J}(\dots\Lambda^{J}(c))$.
The first map is the fixed-points restriction map (\ref{fixedptsdiag}) of \S\ref{sec:fixedpts}.
The isomorphism is induced by the map $\eta_p\colon p^\ast\Lambda^{J/G}(c)\to \Lambda^{J}(c)$ defined at the beginning of \S\ref{sec:tree}. It is an isomorphism because $\Phi$ commutes with fixed-points, and because for every $G$-invariant subset $U\subset J_+$ the restriction of $(\eta_p)_{U}$ on $G$-fixed-points is the isomorphism
\[\Lambda^{J}_U(c)^G=\big(\hocolim_{St(U)}St^U(c)\big)^G\underset{(\ref{fixedptsdiag})}{\cong} \hocolim_{St(U)^G}(St^U(c))^G\cong \hocolim_{St(U/G)}St^{U/G}(c^G)=\Lambda_{U/G}^{J/G}(c^G)\]
where the isomorphism $St(U)^G\cong St(U/G)$ is the restriction of $\mathcal{P}(J_+)^G\cong \mathcal{P}(J/G_+)$. 

\begin{proof}[Proof of \ref{convergence}]
Since $f_n\colon K_n\to J_n$ induces an isomorphism $\overline{f}_n\colon K_n/G\to J_n/G$ on $G$-orbits, there is a commutative square
\[\xymatrix{
P_{J_n}\Phi\ar[r]^-{f_{n}^\ast}& P_{K_n}\Phi\\
P_{J_n/G}\Phi\ar[u]\ar[r]_{\overline{f}^{\ast}_n}^-{\cong}& P_{K_n/G}\Phi\ar[u]
}\]
for every natural number $n$, where the vertical maps are induced by the projections $J_n\to J_n/G$ and $K_n\to K_n/G$. It is therefore sufficient to show that for every proper filtration $\underline{J}$ the map
\[\holim_{n}P_{J_n/G}\Phi\longrightarrow\holim_{n}P_{J_n}\Phi\]
induced by the projection $p\colon J_n\to J_n/G$ is an equivalence. Moreover by cofinality we can assume that $J_n$ has $n$-orbits, and we write $n=J_n/G$.

Let $j\colon n\to n+1$ be the standard inclusion, and let $r\colon (P_{J_n}\Phi)^G\to P_{n}\Phi^G$ be the retraction of $p^\ast$ constructed above. Consider the following commutative diagram
\[\xymatrix@R=11pt@C=15pt{
(P_{n-1}\Phi)^G&&&(P_{n}\Phi)^G\ar@{=}[dll]\ar[ddl]^{p^{\ast}}\ar[lll]_-{j^\ast}
&&&(P_{n+1}\Phi)^G\ar@{=}[dll]\ar[ddl]^{p^{\ast}}\ar[lll]_-{j^\ast}\\
&(P_{n}\Phi)^G\ar[ul]^-{j^\ast}&&&(P_{n+1}\Phi)^G
\ar[ul]^-{j^\ast}\\
&&(P_{J_n}\Phi)^G\ar[ul]^-{r}&&&(P_{J_{n+1}}\Phi)^G\ar@{-->}[lll]\ar[ul]^-{r}
}\]
By cofinality, $(\holim_{n}P_n\Phi)^G$ is equivalent to the homotopy limit of the $(P_{J_n}\Phi)^G$'s along the composite maps $p^\ast j^\ast r$. It then suffices to show that the map $p^\ast j^\ast r$ compares, via a natural zig-zag of equivalences, to the map $\iota^{\ast}\colon (P_{J_{n+1}}\Phi)^G\to (P_{J_n}\Phi)^G$ induced by the inclusion $\iota\colon J_{n}\to J_{n+1}$. Morally this is the case because the map $\iota^{\ast}$ is unique in the homotopy category, but we need to take extra care because we are working on fixed-points. An explicit zig-zag is provided by the following diagram
\[\xymatrix{
(P_{J_{n+1}}\Phi)^G\ar[d]_{\iota^{\ast}}\ar[drrr]|-{p_{J_{n}}P_{J_{n+1}}\Phi}
\ar[rrr]^-{p_{J_{n+1}}P_{J_{n+1}}\Phi}&&&
(P_{J_{n+1}}P_{J_{n+1}}\Phi)^G\ar[d]^{p^\ast j^\ast r}
&&&(P_{J_{n+1}}\Phi)^G\ar[lll]_-{P_{J_{n+1}}p_{J_{n+1}}\Phi}\ar[d]^{p^\ast j^\ast r}\\
(P_{J_{n}}\Phi)^G\ar[rrr]_-{P_{J_{n}}p_{J_{n+1}}\Phi}
&&&(P_{J_{n}}P_{J_{n+1}}\Phi)^G&&&(P_{J_{n}}\Phi)^G
\ar[lll]^-{P_{J_{n}}p_{J_{n+1}}\Phi}
}\]
The mid vertical map is the map $p^\ast j^\ast r$ for the functor $P_{J_{n+1}}\Phi$.
The horizontal arrows are equivalences because $P_{J_{n}}\Phi$ and $P_{J_{n+1}}\Phi$ are both $J_{n+1}$-excisive. The bottom left triangle commutes by the properties of the excisive approximations. The upper-left triangle commutes because the natural transformation $p^\ast \iota^{\ast}r\colon (P_{J_{n+1}}\Phi)^G\to (P_{J_{n}}\Phi)^G$ is a map under $\Phi^G$. The right-hand square commutes by naturality of $p^\ast \iota^\ast=(\iota p)^\ast$ and $r$.  

%The remaining square
%decomposes as
%\[\xymatrix{(P_{nG}\Phi)^G\ar[rrr]^r\ar[d]^-{P_{nG}p_{nG}\Phi}&&&
%(P_n\Phi)^G\ar[d]_-{P_np_{nG}\Phi}\ar[rrr]^{p^{\ast}u}
%&&&(P_{(n-1)G}\Phi)^G\ar[d]^-{P_{(n-1)G}p_{nG}\Phi}\\
%(P_{nG}P_{nG}\Phi)^G\ar[rrr]^r&&&
%(P_nP_{nG}\Phi)^G\ar[rrr]^{p^{\ast}u}
%&&&(P_{(n-1)G}P_{nG}\Phi)^G
%}\]
%and both square commute by naturality of $r$ and $p^\ast u$.

This argument shows that the map $p^\ast\colon (\holim_{n}P_n)^G\to(\holim_{n}P_{J_n})^G$ is an equivalence on $G$-fixed-points, for every finite group $G$. We use the restriction formula of \ref{restriction} to show that this map is an equivalence also on $H$-fixed-points, for every subgroup $H$ of $G$. By naturality of the zig-zag of \ref{restriction} it is sufficient to show that the map
\[P_{n}\Phi^H\simeq\big(\iter_{\underline{w}\in\!\!\!\prod\limits_{o\in J_n/G}\!\!\! o/H}P_{n}(\Phi|_H)\big)^H\stackrel{\iter p_{n}^\ast}{\xrightarrow{\hspace*{.8cm}}}\big(\iter_{\underline{w}\in\!\!\!\prod\limits_{o\in J_n/G}\!\!\! o/H}P_{\amalg\underline{w}}(\Phi|_H)\big)^H\simeq (P_{J_n}\Phi)^H\]
induced by the projections $\amalg\underline{w}\to n$ becomes an equivalence after taking the homotopy limit over $n$. Recall that for every functor $\Psi\colon\Top_\ast\to \Top^{H}_\ast$ and any $H$-set $A$ we constructed a restriction map
\[(P_{A}\Psi)^H\longrightarrow\hocolim_k\holim_{\mathcal{P}_0(A/H_+)^k}\Psi(\Lambda^{A,k}(-))^H.\]
By iterating this construction we obtain a map
\[\big(\!\!\iter_{\underline{w}\in\!\!\!\!\prod\limits_{o\in J_n/G}\!\!\!\! o/H}P_{\amalg\underline{w}}(\Phi|_H)\big)^H\longrightarrow \hocolim_{k_1}\!\! \holim_{\mathcal{P}_0(\amalg\underline{w}_1/H_+)^{k_1}}\!\!\dots \hocolim_{k_l}\!\! \holim_{\mathcal{P}_0(\amalg\underline{w}_l/H_+)^{k_l}}\!\!\Phi(\Lambda^{\amalg\underline{w}_1,k_1}\dots \Lambda^{\amalg\underline{w}_l,k_l})^H
\]
where $l$ is the cardinality of $\prod_{o\in J_n/G} o/H$. The composition of this map with the isomorphism 
\[\Phi(\Lambda^{\amalg\underline{w}_1,k_1}\dots \Lambda^{\amalg\underline{w}_l,k_l}(-))^H\cong \Phi(\Lambda^{n,k_1}\dots \Lambda^{n,k_l}(-))^H\]
gives a retraction of $\iter p_{n}^\ast$. The existence of these retractions show that $\holim_n\iter p_{n}^\ast$ is an equivalence on $H$-fixed-points, by an argument which is formally identical to the one given above for the $G$-fixed-points.
\end{proof}

\subsection{The genuine tower of the stable indexed powers}\label{indexedpowers}

Let $E$ be an orthogonal $G$-spectrum and $K$ a finite $G$-set.
The aim of this section is to describe the genuine subtower of the Goodwillie tree (determined by the free $G$-sets) for the indexed power functor
$\mathcal{T}_{E}^ K\colon \Top_\ast\to \Sp_{O}^ G$ defined by
\[\mathcal{T}_{E}^ K(X)=E\wedge X^{\wedge K},\]
and some variations of this functor involving families of subgroups. These results generalize immediately to the indexed power functors on $G$-spectra $(-)^{\wedge K}\colon \Sp_{O}^ G\to \Sp_{O}^G$ of \cite{HHR} and their family versions, but it is the functors from pointed spaces that will play a role in the calculation of the derivatives of the identity of \S\ref{sec:Id}.

We showed in Proposition \ref{smashfctr} that $\mathcal{T}_{E}^K$ is $kG$-excisive, where $k=|K|$. Thus by \ref{inclusions} it is $nG$-excisive for every $n\geq |K|$, and its genuine Goodwillie tower stabilizes at the $|K|$-th stage. Unlike in the non-equivariant case though, the lower excisive approximations of $\mathcal{T}_{E}^ K$ are usually non-trivial. They are described in the following proposition and its corollary.

\begin{prop}\label{PnGsympowers}
Let $H$ be a subgroup of $G$. The geometric $H$-fixed-points of the $nG$-excisive approximation of $\mathcal{T}_{E}^K$ are respectively
\[
\Phi^H(P_{nG}\mathcal{T}_{E}^K)\simeq \left\{\begin{array}{ll}
\ast&\mbox{if}\ n<|K|\ \mbox{and}\ n\leq |K/H|\\
\Phi^H(\mathcal{T}_{E}^K) &\mbox{otherwise}.
\end{array}
\right.
\]
\end{prop}

\begin{proof}
Let us first observe that \ref{restriction} provides an equivalence
\[\Phi^H(P_{nG}\mathcal{T}_{E}^K)=\Phi^H(P_{nG}\mathcal{T}_{E}^K)|_H\simeq \Phi^H(P_{nH}\mathcal{T}_{E}^{K|_H}),\]
and we can therefore assume that $H=G$.

Suppose that $n<|K|$ and that $n\leq |K/G|$, and let us examine the connectivity of $\Phi^G(T^{(m)}_{nG}\mathcal{T}_{E}^K)$, for some positive integer $m$. By the equivalence of \ref{prop:geomholim} this spectrum is the homotopy limit
\[\begin{array}{ll}\Phi^G(T^{(m)}_{nG}\mathcal{T}_{E}^K(X))\simeq \holim_{\prod\limits_m\mathcal{P}_0(nG_+)^G}(\Phi^{G}E)\wedge (\underbrace{\Lambda^{nG}\dots\Lambda^{nG}}_{m}(X)^{\wedge K})^G\\\simeq (\Phi^{G}E)\wedge\holim_{\prod\limits_m\mathcal{P}_0(nG_+)^G} \mathbb{S}\wedge (\underbrace{\Lambda^{nG}\dots\Lambda^{nG}}_{m}(X)^{\wedge K})^G.
\end{array}\]
The connectivity of the homotopy limit of this formula is bounded below by
\begin{equation}\label{connTnG}\min_{U_1,\dots,U_n\in \mathcal{P}_0(n_+)}
\conn ((\Lambda^{nG}_{p^{-1}U_1}\dots \Lambda^{nG}_{p^{-1}U_n})(X)^{\wedge K})^G-\sum_{i=1}^m(|U_i|-1)
\end{equation}
(see e.g. \cite[A.1.2]{GdiagTop}), where $p\colon nG_+\to n_+$ is the projection map. In order to estimate this quantity we need to understand the connectivity of the $G$-fixed-points of $\Lambda^{nG}_{p^{-1}(U)}(X)^{\wedge K}$, for every subset $U$ of $n_+$. We use the equivalences of \ref{spelloutlambda} to describe these spaces, depending on weather or not $U$ contains the basepoint. When $U$ is a subset of $n$, there is an equivalence
\[(\Lambda^{nG}_{p^{-1}(U)}(X)^{\wedge K})^G\underset{\ref{spelloutlambda}}{\simeq} ((\bigvee_{|U|-1}\Sigma X)^{\wedge K})^G\cong (\bigvee_{(|U|-1)^{\times K}}\Sigma^K(X^{\wedge K}))^G\cong \bigvee_{(|U|-1)^{\times K/G}}\Sigma^{K/G}(X^{\wedge K})^G,\]
and the connectivity of this space is $\conn (X^{\wedge K})^G+|K/G|$.
Now suppose that $U=V_+$ for some subset $V$ of $n$. In this case there is an equivalence
\[(\Lambda^{nG}_{p^{-1}(V_+)}(X)^{\wedge K})^G\underset{\ref{spelloutlambda}}{\simeq} ((\bigvee_{|V|-1}\Sigma^G X)^{\wedge K})^G\cong (\bigvee_{(|V|-1)^{\times K}}\Sigma^{K\times G}(X^{\wedge K}))^G\cong \bigvee_{(|V|-1)^{\times K/G}}\Sigma^{|K|}(X^{\wedge K})^G,\]
and this has connectivity $\conn (X^{\wedge K})^G+|K|$.
By induction, we see that he connectivity range of (\ref{connTnG}) is
\[\min_{U_1,\dots,U_n}
\conn (X^{\wedge K})^G+\sum_{i=1}^m p(U_i)-\sum_{i=1}^m(|U_i|-1)=\min_{U_1,\dots,U_n}
\conn (X^{\wedge K})^G+\sum_{i=1}^m (p(U_i)-|U_i|+1)
\]
where $p(U_i)=|K/G|$ if $U_i$ does not contain $+$, and $p(U_i)=|K|$ otherwise. By our assumption $p(U_i)-|U_i|\geq 0$, and therefore
\[\conn \holim_{\prod\limits_m\mathcal{P}_0(nG_+)^G} \mathbb{S}\wedge (\underbrace{\Lambda^{nG}\dots\Lambda^{nG}}_{m}(X)^{\wedge K})^G\geq \conn (X^{\wedge K})^G+m\]
diverges with $m$. It follows that $\Phi^GP_{nG}\mathcal{T}_{E}^K$ is contractible.

Now let us suppose that either $n\geq |K|$, or that $n>|K/G|$. In the first case we know by \ref{smashfctr} that $\mathcal{T}_{E}^K$ is $K\times G=|K|\times G$-excisive, and therefore that $\mathcal{T}_{E}^K$ and $P_{nG}\mathcal{T}_{E}^K$ are equivalent. It remains to show that the map $\mathcal{T}_{E}^K\to P_{nG}\mathcal{T}_{E}^K$ is an equivalence on geometric $G$-fixed-points when $n>|K/G|$. The argument of \ref{smashfctr} shows that this is the case if the functor 
$p\colon \wr St(-)^{K}\to St(nG_+)^K$ restricts to a right-cofinal functor on $G$-fixed-points. We recall that $\wr St(-)^{K}$ is the Grothendieck construction of the functor
\[St(-)^K\colon \mathcal{P}_1(nG_+)\longrightarrow Cat,\]
and that $p$ sends a pair $(U\subsetneq nG_+,\underline{V}\in St(U)^K)$ to $\underline{V}$ viewed as a collection of subsets of $nG_+$. Let $\underline{W}$ in $St(nG_+)^K$ be a collection of subsets which is invariant for the $G$-action on $St(nG_+)^K$. We show that the category $(\underline{W}/p)^G$ has an initial object.The pair $(\cup_k W_k, \underline{W})$ is $G$-invariant, and it defines an initial object provided the subset $\cup_k W_k$ of $nG_+$ is proper. We observe that since $\underline{W}$ is $G$-invariant, if the component $W_k$ is a subset of $\{i\}\times G_+$, the component $W_{gk}$ must be a subset of the same orbit $\{i\}\times G_+$. Hence the union $\cup_k W_k$ can cover at most $|K/G|$ copies of $G_+$, and since $n>|K/G|$ it must be a proper subset of $nG_+$.
\end{proof}

We can now repackage the information of Proposition \ref{PnGsympowers} to describe the full functor $P_{nG}\mathcal{T}_{E}^ K$. 
%\begin{rem}\label{symvsnot} We observe that there is a natural equivalence
%\[\mathcal{T}_{E}^K(X)=E\wedge X^{\wedge K}\simeq ((Aut_K)_+\wedge E\wedge X^{\wedge K})_{hAut_K}\simeq \mathcal{S}^{K}_{(Aut_K)_+\wedge E}(X)\]
%where $Aut_K$ acts on the first factor of $C=(Aut_K)_+\wedge E$, and $G$ on the second. Thus we will state the results only for the symmetric indexed power functors $\mathcal{S}^{K}_{C}$.
%\end{rem}
Given a set of subgroups $\mathcal{R}$ of a finite group $\Sigma$ closed by conjugation, we let $\overline{E}\mathcal{R}$ be a pointed $\Sigma$-space whose fixed-points by the subgroups $\Gamma$ of $\Sigma$ are characterized by the property
\[\overline{E}\mathcal{R}^\Gamma=\left\{\begin{array}{ll}
 S^0&\mbox{if $\Gamma\in\mathcal{R}$}\\
\ast&\mbox{otherwise}
\end{array}\right.\]
(This space can be constructed by closing $\mathcal{R}$ to a family of subgroups $\mathcal{F}$, and by defining $\overline{E}\mathcal{R}$ as the cofiber of the map between universal $\Sigma$-spaces $E(\mathcal{F}\backslash \mathcal{R})_+\to E\mathcal{F}_+$ induced by the inclusion of families.) If $\mathcal{R}$ is not closed by conjugation, we still write $\overline{E}\mathcal{R}$ for the pointed $\Sigma$-space associated to its closure by conjugation.

Now let $\mathcal{R}$ be a set of subgroups of $G\times\Sigma_k$, such that every element of $\mathcal{R}$ intersects $1\times\Sigma_k$ trivially, or in other words a collection of graphs of group homomorphisms $\rho\colon H\to\Sigma_k$ where $H$ is a subgroup $G$. We think of $\mathcal{R}$ as a set of actions of subgroups of $G$ on a $k$-elements set. Given an orthogonal $G\times\Sigma_k$-spectrum $C$, we define the corresponding indexed symmetric power functor $\mathcal{S}^{\mathcal{R}}_C\colon \Top_\ast\to \Sp_{O}^ G$ by
\[\mathcal{S}_{C}^{\mathcal{R}}(X)=(C\wedge X^{\wedge k})\wedge_{\Sigma_k}\overline{E}\mathcal{R}\]
where $\Sigma_k$ acts both on $C$ and on $X^{\wedge k}$ by permuting the smash factors.

\begin{rem}\label{handy} We will make repeated use of the following well-known formula. If $\mathcal{R}$ is closed under conjugation, the are natural equivalences
\[\Phi^H\mathcal{S}_{C}^{\mathcal{R}}(X)\simeq \big(\bigvee_{\substack{\rho\colon H\to\Sigma_k\\ \Gamma_\rho\in\mathcal{R}}}\Phi^{H}(\rho^{\ast}(E)\wedge X^{\wedge \rho^\ast k})\big)_{h\Sigma_k}\simeq \bigvee_{\substack{[\rho\colon H\to\Sigma_k]\\ \Gamma_{\rho}\in\mathcal{R}}}\Phi^{H}(\rho^{\ast}(E)\wedge X^{\wedge \rho^\ast k})_{hAut_{\rho^{\ast}k}}
\]
where $\Gamma_\rho$ is the graph of $\rho$, the $H$-spectrum $\rho^{\ast}E$ is the restriction of $E$ along the inclusion $H\to G\times \Sigma_k$ which sends $h$ to $(h,\rho(h))$, and $\rho^{\ast}k$ is the $H$-set defined by the homomorphism $\rho$. Moreover $\Sigma_k$ acts on the set indexing the first wedge by pointwise conjugation of group homomorphisms, and the second wedge runs over the equivalence classes of this action. A proof of this formula for pointed spaces can be found in \cite[I.4]{Carlsson}, and it can be readily extended to spectra using the model for the geometric fixed points chosen in \ref{prop:geomholim}.
\end{rem}

\begin{ex}\label{indexedfromfamily}
Let $K$ be a finite $G$-set with $k$-elements, and $\rho_K\colon G\to\Sigma_k$ its corresponding group homomorphism (upon a choice of total order on the set underlying $K$).
Let $\mathcal{R}_K$ be the set of all the subgroups of the graph of $\rho_K$. Then for every orthogonal $G$-spectrum $E$ the canonical map
\[\mathcal{T}_{E}^{K}(X)=E\wedge X^{\wedge K}\stackrel{\simeq}{\longleftarrow} ((\Sigma_k)_+\wedge E\wedge X^{\wedge k})\wedge_{\Sigma_k}\overline{E}\mathcal{R}_K=\mathcal{S}_{(\Sigma_k)_+\wedge E}^{\mathcal{R}_K}(X)\]
is an equivalence, where $G$ acts on $\Sigma_k$ by left multiplication through $\rho_K\colon G\to\Sigma_k$. Indeed by the decomposition formula \ref{handy} the geometric $H$-fixed-points of the right-hand side are equivalent to
\[\bigvee_{\{\rho_K|_H\}}((\rho_{K}^{\ast}\Sigma_k)_{+}^H\wedge\Phi^{H}(E\wedge X^{\wedge K}))_{hAut_{K|_H}}\]
where the wedge run through the unique graph in $\mathcal{R}_K$ with domain $H$.
By definition of the $G$-action the $H$-fixed-points set of  $\rho_{K}^{\ast}\Sigma_k$ is the group of equivariant automorphisms $Aut_{K|_H}$, and the map above corresponds to the canonical equivalence
\[((Aut_{K|_H})_{+}^H\wedge\Phi^{H}(E\wedge X^{\wedge K}))_{hAut_{K|_H}}\stackrel{\simeq}{\longrightarrow}\Phi^{H}(E\wedge X^{\wedge K}).\]
\end{ex}

Given a set $\mathcal{R}$ of subgroups of $G\times\Sigma_k$ which intersect $1\times \Sigma_k$ trivially and a positive integer $n<k$, we define
\[\mathcal{R}(<n)=\{\Gamma_\rho\in\mathcal{R}\ |\ \rho\colon H\to\Sigma_k\ \mbox{with}\ |(\rho^\ast k)/H|<n\},\]
the set of actions in $\mathcal{R}$ with less than $n$-orbits.

\begin{cor}\label{symtree}
For any orthogonal $G\times\Sigma_k$-spectrum $C$ and positive integer $n<k$, there is a natural equivalence
\[P_{nG}\mathcal{S}_{C}^{\mathcal{R}}(X)\simeq (C\wedge X^{\wedge k})\wedge_{\Sigma_k}\overline{E}\mathcal{R}(<n).\]
In particular if $K$ is a finite $G$-set with $k$ elements and $E$ is an orthogonal $G$-spectrum, by setting $C=(\Sigma_k)_+\wedge E$ as in \ref{indexedfromfamily} we obtain an equivalence
\[P_{nG}\mathcal{T}_{E}^K(X)\simeq ((\Sigma_{k})_+\wedge E\wedge X^{\wedge k})\wedge_{\Sigma_k}\overline{E}\mathcal{R}_K(<n).\]
\end{cor}

\begin{rem}
We observe that since $\mathcal{R}_K(<n)$ does not contain the trivial homomorphisms $H\to\Sigma_k$ the non-equivariant homotopy type of $\overline{E}\mathcal{R}_K(<n)$ is contractible. Indeed non-equivariantly $P_n (E\wedge X^{\wedge k})$ is contractible when $n<k$.
\end{rem}

\begin{proof}[Proof of \ref{symtree}]
If the set of subgroups $\mathcal{R}$ decomposes as $\mathcal{A}\amalg\mathcal{B}$ (where all the sets are closed by conjugation), then $\overline{E}\mathcal{A}\vee \overline{E}\mathcal{B}$ is a model for $\overline{E}\mathcal{R}$. Thus if we denote by $\mathcal{L}$ the complement of $\mathcal{R}(<n)$ in $\mathcal{R}$,
the functor $\mathcal{S}_{C}^{\mathcal{R}}$ splits as
\[\mathcal{S}_{C}^{\mathcal{R}}\simeq \mathcal{S}_{C}^{\mathcal{R}(<n)}\vee \mathcal{S}_{C}^{\mathcal{L}}.\]
Hence the corollary holds provided we can prove that $\mathcal{S}_{C}^{\mathcal{R}(<n)}$ is $nG$-excisive and $P_{nG}\mathcal{S}_{C}^{\mathcal{L}}$ is trivial.

We use \ref{prop:geomholim} to commute $P_{nG}$ and geometric fixed-points, and the decomposition formula \ref{handy} gives an equivalence
\[
\Phi^HP_{nG}\mathcal{S}_{C}^{\mathcal{R}}
\simeq
\big(\bigvee_{\substack{\rho\colon H\to\Sigma_k\\ \Gamma_\rho\in \mathcal{R}}}\Phi^{H}P_{nG}\mathcal{T}_{\rho^\ast C}^{\rho^{\ast }k}\big)_{h\Sigma_k}
.\]
which is compatible with the splitting above. Here $\rho^\ast C$ is the $H$-spectrum obtained by restricting $C$ along the map $H\to G\times \Sigma_k$ which sends $h$ to $(h,\rho(h))$. By \ref{PnGsympowers} one immediately sees that $\Phi^{H}P_{nG}\mathcal{T}_{\rho^\ast C}^{\rho^{\ast }k}$ is equivalent to $\Phi^{H}\mathcal{T}_{\rho^\ast C}^{\rho^{\ast }k}$ when $\Gamma_\rho$ belongs to $\mathcal{R}(<n)$, and that it is contractible if it belongs to $\mathcal{L}$. 
\end{proof}

\section{Homogeneous functors}

Let $G$ be a finite group, let $J$ be a finite $G$-set, and let $\mathscr{C}$ and $\mathscr{D}$ be $G$-model categories. We will assume from now on that $\mathscr{C}$ and $\mathscr{D}$ are pointed.
We define the notion of $J$-homogeneous homotopy functor, and we construct a supply of examples by taking the ``layers'' of the Goodwillie tree of any homotopy functor $\Phi\colon \mathscr{C}\to \mathscr{D}^G$ (in \S\ref{layers}). We study the behavior of the cross-effect functor on $J$-homogeneous functors in \S\ref{multilinear}, and we classify ``strongly homogeneous functors'' as equivariant spectra with a na\"{i}ve $\Sigma_n$-action in \S\ref{sec:class}.

\begin{defn}
Let $J$ be a finite $G$-set. A homotopy functor $\Phi\colon \mathscr{C}\to \mathscr{D}^G$ is $J$-reduced if $P_K\Phi$ is contractible for every proper $G$-subset $K$ of $J$. We say that $\Phi$ is $J$-homogeneous if it is $J$-excisive and $J$-reduced. We say that $\Phi$ is strongly $J$-homogeneous if it is $J$-excisive and $n$-reduced, where $n=|J/G|$.
\end{defn}

\begin{rem}
A strongly $J$-homogeneous functor is in particular $J$-homogeneous, since
for every proper $G$-subset $K\subset J$ there are equivalences
\[P_{K}\Phi\simeq P_{K/G}P_{K}\Phi\simeq P_{K}P_{K/G}\Phi\simeq P_{K/G}\ast\simeq\ast.\]
This is because $P_K\Phi$ is $K/G$-excisive by \ref{isoonorb}, and because $|K/G|< |J/G|$.   
\end{rem}

\begin{ex}
If $J=T$ is a transitive $G$-set, the only proper subset $K$ of $T$ is the empty-set. A functor is $T$-reduced if and only if $P_{\emptyset}\Phi=\Phi(\ast)$ is equivalent to the zero-object of $\mathscr{D}^G$. In this case strongly $T$-homogeneous functors and $T$-homogeneous functors agree. We refer to these functors as $T$-linear. It is a theorem of \cite{Blumberg} that there is a model structure on the category of finitary enriched reduced homotopy functors on pointed $G$-spaces which is Quillen equivalent to the model category of $G$-spectra on the universe $\bigoplus_k\mathbb{R}[kT]$, and where the fibrant objects are the $T$-linear functors. This Theorem is an equivariant version of Goodwillie's classification of $1$-homogeneous functors of \cite{calcIII}. 
\end{ex}

\begin{ex}\label{orbitsstronglyhom}
For any orthogonal $G$-spectrum $E$, the functor $\mathcal{T}^{n}_E\colon \Top_\ast\to \Sp_{O}^G$ of Proposition \ref{smashfctr} defined by
\[\mathcal{T}^{n}_E(X)=E\wedge X^{\wedge n}\]
is strongly $nG$-homogeneous. We showed in \ref{smashfctr} that $\mathcal{T}^{n}_E$ is $nG$-excisive, and it follows by \cite[3.1-3.2]{calcIII} that $\mathcal{T}^{n}_E$ is $n$-reduced. As an immediate consequence we see that the functor $\mathcal{S}_{E}^n(X)=(E\wedge X^{\wedge n})_{h\Sigma_n}$ is also strongly $nG$-homogeneous, where $E$ is an orthogonal $G\times\Sigma_n$-spectrum. We observe that a map $f\colon E\to W$ of $G\times \Sigma_n$-spectra induces an equivalence $f_\ast\colon \mathcal{S}^{n}_E\to \mathcal{S}^{n}_W$ if $f$ is an equivalence of $G$-spectra after forgetting the $\Sigma_n$-action. Thus we think of $E$ as an object of the model category of genuine $G$-spectra with na\"{i}ve $\Sigma_n$-action.
\end{ex}

\begin{ex}
The $n$-th layer of the genuine tower of a homotopy functor $\Phi\colon \mathscr{C}\to \mathscr{D}^G$ between pointed $G$-model categories
\[D_{nG}\Phi=\hofib(P_{nG}\Phi\longrightarrow P_{(n-1)G}\Phi)\]
is $nG$-homogeneous, but generally not strongly $nG$-homogeneous. We will generalize this construction to an arbitrary $G$-set $J$ in \S\ref{layers}.
\end{ex}

\begin{ex}\label{DnGsym}
The calculation of \ref{symtree} allows us to see which symmetric power functors $\mathcal{S}^{\mathcal{R}}_{C}$ are $nG$-homogeneous. We recall that $C$ is an orthogonal $G\times\Sigma_k$-spectrum and that $\mathcal{R}$ is a set of subgroups of $G\times \Sigma_k$ which intersect $1\times\Sigma_k$ trivially. Given a pair of integers $n< k$ we let
\[\mathcal{R}(n)=\{\Gamma_\rho\leq G\times \Sigma_k\ |\ \rho\colon H\to\Sigma_k \ \mbox{with}\ \rho^{\ast}k/H=n-1\},\]
and for $n=k$ we set
\[\mathcal{R}(n)=\{\Gamma_\rho\leq G\times \Sigma_n\ |\ \rho\colon H\to\Sigma_n \ \mbox{with}\ \rho=1 \ \mbox{or}\ \rho^{\ast}n/H= n-1\}.\]
These collections record the actions of $\mathcal{R}$ which are either trivial or non-trivial with $n-1$ orbits. We observe that $\mathcal{R}(n)$ is non-empty only when $k\leq (n-1)|G|$ or $k=n$.
It follows immediately from \ref{symtree} that the $n$-th layer of the genuine tower of $\mathcal{S}^{\mathcal{R}}_{C}$ is
\[D_{nG}\mathcal{S}^{\mathcal{R}}_{C}\simeq \mathcal{S}^{\mathcal{R}(n)}_{C},\]
and that $\mathcal{S}^{\mathcal{R}(n)}_{C}$ is $nG$-homogeneous.
\end{ex}

\subsection{The layers of the tree}\label{layers}

Let $\Phi\colon \mathscr{C}\to \mathscr{D}^G$ be a homotopy functor between pointed $G$-model categories, and let $J$ be a finite $G$-set. We construct a $J$-homogeneous homotopy functor and a strongly $J$-homogeneous homotopy functor from the Goodwillie tree of $\Phi$ defined in \S\ref{sec:tree}. They are analogous to the layers of the classical Taylor tower. 

Given an inclusion of finite $G$-sets $\iota\colon K\to J$ we let $D_\iota\Phi$ be the homotopy fiber
\[D_\iota\Phi=\hofib\big(P_J\Phi\stackrel{\iota^\ast}{\longrightarrow} P_K\Phi\big).\]
For every $G$-orbit $o$ of $J$ let $\iota_o\colon J\backslash o\to J$ denote the inclusion of the complement of $o$ in $J$.
\begin{defn}
The $J$-differential of a homotopy functor $\Phi\colon \mathscr{C}\to \mathscr{D}^G$ is the homotopy functor $D_J\Phi\colon \mathscr{C}\to \mathscr{D}^G$ defined as the iteration
\[D_J\Phi= D_{\iota_{o_1}}\dots D_{\iota_{o_n}}\Phi\]
for a choice of order $J/G=\{o_1,\dots,o_n\}$ on the set of orbits of $J$. Different orders give equivalent differentials by \ref{PJcommute}.
The strong $J$-differential of $\Phi$ is the homotopy functor $\overline{D}_J\Phi\colon \mathscr{C}\to \mathscr{D}^G$ defined as
\[\overline{D}_J\Phi=P_JD_n\Phi\]
where $n=|J/G|$.
\end{defn}

\begin{rem}
The $J$-differential can equivalently be defined as the total homotopy fiber of the $|J/G|$-cube in $\mathscr{D}^G$
\[D_J\Phi\simeq Tfib\left(\begin{array}{l}
\mathcal{P}(J/G)\longrightarrow \mathscr{D}^G\\
W\longmapsto P_{J\backslash w_1}\dots P_{J\backslash w_l}\Phi
\end{array}\right)\]
where $W=\{w_1,\dots,w_l\}$. 
\end{rem}

%Given a finite collection of $G$-sets $\underline{J}=\{J_a\}_{a\in A}$, we want a definition of the iterated $J_a$-excisive approximations $P_{J_{a_1}}\dots P_{J_{a_k}}\Phi$ that does not depend on a choice of order on $A$. To this end, define a functor $T_{\underline{J}}\Phi\colon C\to D^ G$ as the homotopy limit
%\[T_{\underline{J}}\Phi(c)=\holim_{\mathcal{P}_0\big((\coprod\limits_{a\in A}J_a)_+\big)}\Phi(\Lambda^{\underline{J}}(c))\]
%where the diagram $\Lambda^{\underline{J}}(c)$ has vertices
%\[\Lambda^{\underline{J}}(c)_{U}=\hocolim_{\prod\limits_{a\in A}St(U\cap J_a)}St^{U}(c)\]
%As usual $St^{U}(c)$ has $\emptyset$-vertex $c$ and the terminal object everywhere else. The functor $T_{\underline{J}}\Phi$ is equivalent to $T_{J_{a_1}}\dots T_{J_{a_k}}\Phi$, and the functor
%\[P_{\underline{J}}\Phi:=\hocolim\big(T_{\underline{J}}\Phi\to T_{\underline{J}}^2\Phi\to\dots\big)\]
%is equivalent to $P_{J_{a_1}}\dots P_{J_{a_k}}\Phi$. For every $G$-orbit $z$ of $J$ define the $G$-set $J_z=J\backslash z$.

\begin{ex}
If the orbits of $J$ are all isomorphic, that is if $J\cong nT$ for some transitive $G$-set $T$, the $nT$-differential of $\Phi$ is equivalent to
\[D_{\iota_T}\dots D_{\iota_T}\simeq D_{\iota_T}=\hofib \big(P_{nT}\to P_{(n-1)T}\big).\]
In particular
\[D_n\Phi\simeq \hofib\big(P_n\Phi\to P_{n-1}\Phi\big)\ \ \ \ \ \ \ \ \mbox{and}\ \ \ \ \ \ \
 D_{nG}\Phi\simeq \hofib\big(P_{nG}\Phi\to P_{(n-1)G}\Phi\big).\]
\end{ex}

\begin{prop}
For any homotopy functor $\Phi\colon \mathscr{C}\to \mathscr{D}^G$ and any finite $G$-set $J$, the functor $D_J\Phi$ is $J$-homogeneous, and the functor $\overline{D}_J\Phi$ is strongly $J$-homogeneous.
\end{prop}

\begin{proof}
The functor $\overline{D}_J\Phi$ is clearly $J$-excisive. It is $n$-reduced since $P_k\overline{D}_J\Phi\simeq P_JP_kD_n\Phi\simeq \ast$ if $k<n$.

The $J$-differential is a homotopy limit of $J$-excisive functors (by \ref{inclusions}), and therefore it is $J$-excisive.
Let us prove that $D_J\Phi$ is $J$-reduced. A proper $G$-subset $K$ of $J$ is necessarily included in $J\backslash o$ for some orbit $o$. Since the $D_{\iota_o}$'s commute with each other we may assume that $o=o_n$, and 
\[P_KD_J\Phi\simeq D_{\iota_{o_1}}\dots D_{\iota_{o_n}}P_{K}\Phi=D_{\iota_{o_1}}\dots D_{\iota_{o_{n-1}}}\hofib(P_{J}P_{K}\to P_{J\backslash o_n}P_{K}\Phi).\]
Since $K$ is included in $J\backslash o_n$ the functor $P_{K}\Phi$ is both $J$-excisive and $J\backslash o_n$-excisive, and the map 
\[P_K\Phi\simeq P_{J}P_K\Phi\longrightarrow P_{J\backslash o_n}P_{K}\Phi\simeq P_K\Phi\]
is an equivalence. Thus $P_KD_J$ is contractible.
\end{proof}

The differentials of $\Phi$ are functorial with respect to the $G$-equivariant maps of finite $G$-sets that are bijective on orbits. Such a map $\alpha\colon K\to J$ induces a map on homotopy fibers
\[\xymatrix{
D_{\iota_o}\Phi\ar[d]\ar@{-->}[r]^{\alpha_{o}^{\ast}}&D_{\iota_{\alpha(o)}}\Phi\ar[d]\\
P_{J}\Phi\ar[r]^{\alpha^\ast}\ar[d]&P_{K}\Phi\ar[d]\\
P_{J\backslash o}\Phi\ar[r]^{\alpha^\ast}&P_{K\backslash \alpha(o)}
}\]
for every orbit $o$ in $J/G$, and every homotopy functor $\Phi$. 
By iterating this construction we obtain a map
\[\alpha^\ast=\alpha_{o_1}^{\ast}\dots \alpha_{o_n}^\ast\colon D_J\Phi\longrightarrow D_K\Phi.\]

\subsection{Equivariant deloopings of homogeneous functors}\label{delooping}

Let $J$ be a finite $G$-set and let $\Phi\colon \Top_\ast\to \Top_{\ast}^{G}$ be a homotopy functor on the $G$-model category of pointed $G$-spaces. We show that if $\Phi$ is $J$-homogeneous it lifts to the category of ``$J$-spectra'' along the $\Omega^{\infty J}$-functor. This is the category obtained from $\Top_{\ast}^G$ by inverting the permutation representation sphere $S^J$.

We choose the category $\Sp_{\Sigma}^G$ of $G$-objects in symmetric spectra as a point-set model for this category of $J$-spectra. The model structure on this category is defined by stabilizing with respect to the $G$-sets in the $G$-set universe $\amalg_{n\in\mathbb{N}}J$, and it is denoted $\Sp_{J}^{G}$. This model structure is a symmetric analogue of the model structure on orthogonal $G$-spectra induced by the universe $\bigoplus_n\mathbb{R}[J]$. It is studied in great details in \cite{Markus}, and it is Quillen equivalent to the model category of \cite{Mandell} when the orbits of $J$ are of the form $G/N$ for normal subgroups $N\lhd G$ (see \cite[Thm A]{Markus}). 

\begin{ex}
The category $\Sp^{G}_{\underline{n}}$ is a model for the category of na\"{i}ve symmetric $G$-spectra, and $\Sp^{G}_{nG}$ is a model for the category of genuine symmetric $G$-spectra, for every $n\in\mathbb{N}$. 
\end{ex}

We promote $\Sp_{J}^G$ to a $G$-model category, by considering the collection of categories of $H$-objects in symmetric spectra $\{\Sp^{H}_\Sigma\}_{H\leq G}$, where each category is equipped with the model structure induced by the universe of the $H$-sets that are restriction of the $G$-sets in $\amalg_{n\in\mathbb{N}}J$. 
There is an equivariant infinite loop space functor 
\[\Omega^{\infty J}\colon \Sp_{J}^G\longrightarrow \Top^{G}_\ast\]
that sends a symmetric $G$-spectrum $E$ to $\Omega^{\infty J}E=\hocolim_m\Omega^{mJ}E_{mJ}$, where $E_{mJ}$ is the value of $E$ at the $G$-set $mJ$
\[E_{mJ}:=E_{m|J|}\wedge_{\Sigma_{m|J|}}Bij(m|J|,mJ)_+.\]
Here $\Sigma_{m|J|}$ and $G$ act on the set of bijections $Bij(m|J|,mJ)$, respectively on the right and on the left, in the obvious way. The functor $\Omega^{\infty J}$ preserves equivalences of semi-stable $G$-spectra, and in particular of fibrant $G$-spectra (see \cite[Thm B]{Markus}).

\begin{theorem}\label{Jdelooping}
A $J$-homogeneous homotopy functor $\Phi\colon \Top_\ast\to \Top_{\ast}^{G}$ lifts functorially to a $J$-homogeneous homotopy functor $\widehat{\Phi}$ to the category of $J$-spectra
\[\xymatrix{&\Sp_{J}^G\ar[d]^{\Omega^{\infty J}}\\
\Top_\ast\ar@{-->}[ur]^-{\widehat{\Phi}}\ar[r]_{\Phi}&\Top_{\ast}^G
}\]
which has fibrant values.
\end{theorem}

We start by constructing functorial $J$-fold deloopings $\Phi\simeq \Omega^JR_J\Phi$ for the $J$-homogeneous functor $\Phi$. This is where the real content of Theorem \ref{Jdelooping} is. The lift $\widehat{\Phi}$ to our point-set model $\Sp_{J}^G$ is constructed formally from these deloopings at the end of the section. The proof of Theorem \ref{Jdelooping} will occupy the rest of Section \ref{delooping}. 

We construct equivariant deloopings $P_J\Phi\simeq\Omega^oR_o\Phi$ of $P_J\Phi$ one orbit $o\in J/G$ at the time. Since $\Phi$ is $J$-excisive the map $p_J\colon \Phi\to P_J\Phi$ is an equivalence, and this provides us with an $o$-fold delooping of $\Phi$. We will show in \ref{RJreduced} and \ref{RJexc} that this construction can be iterated for every orbit providing us with a $J$-fold delooping of $\Phi$.
Fix a $G$-orbit $o$ of $J$ and a positive integer $k$. We construct the $o$-delooping of $P_J\Phi$ following the idea of \cite[2.2]{calcIII}.
We decompose $T_{J}^k\Phi$ as the homotopy limit of a punctured $o_+$-cube by defining an equivariant cover $\{\mathcal{A}^{k}_l\}_{l\in o_+}$ of the category with $G$-action $\mathcal{P}_0(J_+)^k$. The cover consists of the subposets
\[\mathcal{A}^{k}_j=\mathcal{P}_0(J_+)^k\backslash\mathcal{P}_0(J_+\backslash\{j\})^k \ \ \ \ \ \ \ \ \ \ \ \ \ \ \ \mathcal{A}^{k}_+=\big(\mathcal{P}_0(J_+)\backslash\mathcal{P}_0(o)\big)^k\]
of $\mathcal{P}_0(J_+)^k$, where $j$ runs through the elements of $o$. By ``equivariant cover'' we mean that the $G$-action functors $g\colon \mathcal{P}_0(J_+)^k\to \mathcal{P}_0(J_+)^k$ restrict to functors $g\colon\mathcal{A}^{k}_l\to \mathcal{A}^{k}_{gl}$ for every $l\in J_+$.
%An object of $\mathcal{A}^{k}_+$ is a collection of non-empty subsets $\underline{U}=\{U_l\}_{l\in\underline{k}}$ of $J_+$ such that no component $U_l$ is included in $o$. An object of $\mathcal{A}^{k}_j$ is a collection $\underline{U}$ such that at least one of the components $U_l$ contains $j$.
For every subset $W$ of $o_+$ we denote the intersection of the subposets corresponding to the elements of $W$ by 
\[\mathcal{A}^{k}_W=\bigcap_{j\in W}\mathcal{A}^{k}_j.\]
The cover induces an $o_+$-cube $R^{\mathcal{A}^{k}}\Phi(X)\in (\Top_\ast)^{\mathcal{P}(o_+)}_a$ for every pointed space $X$, with vertices
\[R^{\mathcal{A}^{k}}_W\Phi(X)=\holim_{\mathcal{A}^{k}_W}\Phi\big(\Lambda^{J,k}(X)\big)\]
where $\Lambda^{J,k}$ denotes the $k$-fold iteration $\Lambda^J(\dots\Lambda^{J}(X))$.
By the Covering Lemma \ref{coverings} this cube is homotopy cartesian, that is $T_{J}^{k}\Phi=R^{\mathcal{A}^{k}}_{\emptyset}\Phi$ decomposes as the homotopy limit
\[T_{J}^k\Phi=\holim_{\mathcal{P}_0(J_+)^k}\Phi\big(\Lambda^{J,k}(-)\big)\stackrel{\simeq}{\longrightarrow}\holim_{W\in \mathcal{P}_0(o_+)}R^{\mathcal{A}^{k}}_W\Phi.\]
The plan is to show that for every non-empty subset $W$ of $o_+$ different than $\{+\}$ or $o_+$, the vertex $R^{\mathcal{A}^{k}}_W\Phi$ is contractible (see \ref{Aast}), and that $R^{\mathcal{A}^{k}}_{\{+\}}\Phi$ is equivalent to $T_{J\backslash o}^k\Phi$ ( see \ref{+vertex}).
Roughly speaking, also the $+$-vertex of $R^{\mathcal{A}^{k}}\Phi$ becomes contractible after taking the homotopy colimit over $k$, because $\Phi$ is $J$-reduced.
The only non-trivial vertices of the cube $R^{\mathcal{A}^{k}}\Phi$ are the initial vertex $T^{k}_{J}\Phi$, and the final vertex
\[R^{\mathcal{A}^{k}}_{o_+}\Phi=\holim_{\mathcal{A}_{o_+}^k}\Phi\big(\Lambda^{J,k}(-)\big).\]
After stabilizing over $k$ we will get an equivalence
\[P_J\Phi=\hocolim_k T^{k}_J\Phi\stackrel{\simeq}{\longrightarrow}\hocolim_k\holim_{\mathcal{P}_0(o_+)}R^{\mathcal{A}^{k}}\Phi
\stackrel{\simeq}{\longrightarrow}\holim_{\mathcal{P}_0(o_+)}\hocolim_kR^{\mathcal{A}^{k}}\Phi
\stackrel{\simeq}{\longleftarrow}\Omega^o \hocolim_kR^{\mathcal{A}^{k}}_{o_+}\Phi\]
which defines an $o$-delooping of $\Phi\simeq P_J\Phi$. Just as in \cite{calcIII} there is an issue with the fact that there are no direct stabilization maps $R^{\mathcal{A}^{k}}_W\Phi\to R^{\mathcal{A}^{k+1}}_W\Phi$, and we need a bit of care in taking the homotopy colimit over $k$. This issue is addressed in \ref{wrongway} below, after the analysis of the vertices of the covering cube $R^{\mathcal{A}^{k}}\Phi$ of the next two lemmas.

\begin{lemma}\label{Aast}
Let $\{+\}\neq W\subsetneqq o_+$ be non-empty. The functor $R^{\mathcal{A}^{k}}_W\Phi=\holim_{\mathcal{A}_W^k}\Phi\big(\Lambda^{J,k}(-)\big)$ is $G_W$-contractible.
\end{lemma}

\begin{proof}
Let us define subposets
\[\mathcal{A}^{\ast}_W=\{\underline{U}\in \mathcal{P}_0(J_+)^k| \ \mbox{for every}\ j\in W \ , \ j\in U_l\subset W\ \mbox{for some }\ l\}\subset\mathcal{A}^{k}_W\] 
\[\mathcal{A}^{\ast}_{W_+}=\Big(\{\underline{U}\in \mathcal{P}_0(J_+)^k| \ \mbox{for every}\ j\in W \ , \ j\in (U_l\cap z_+)\subset W_+\ \mbox{for some }\ l\}\cap\mathcal{A}^{k}_+\Big)
\subset \mathcal{A}^{k}_{W_+}.\]
We claim that the inclusions $\mathcal{A}^{\ast}_W\subset \mathcal{A}^{k}_W$ and $\mathcal{A}^{\ast}_{W_+}\subset \mathcal{A}^{k}_{W_+}$ are both left cofinal. This will show that the canonical maps
\[\holim_{\mathcal{A}^{k}_W}\Phi\big(\Lambda^{J,k}(-)\big)\stackrel{\simeq}{\longrightarrow} \holim_{\underline{U}\in\mathcal{A}^{\ast}_W}\Phi\big(\Lambda^{J,k}_{\underline{U}}(-)\big)\ \ \ \ \ \mbox{and} \ \ \ \ \ \holim_{\mathcal{A}^{k}_{W_+}}\Phi\big(\Lambda^{J,k}(-)\big)\stackrel{\simeq}{\longrightarrow} \holim_{\underline{U}\in\mathcal{A}^{\ast}_{W_+}}\Phi\big(\Lambda^{J,k}_{\underline{U}}(-)\big)\]
are equivalences.
If $j\in U_l\subset W$ the space $\Lambda^{J}(X)_{U_l}$ is a $U_l$-fold cone and thus contractible, and for every $\underline{U}\in\mathcal{A}^{\ast}_{W}$ the space $\Lambda^{J,k}(X)_{\underline{U}}$ must be contractible. It follows that the first homotopy limit is contractible. Similarly, for a $\underline{U}\in\mathcal{A}^{\ast}_{W_+}$ the space $\Lambda^{J}(X)_{U_l}$ corresponding to a component $U_l$ with $j\in (U_l\cap o_+)\subset W_+$ is a $U_l$-fold cone, since $W_+\neq o_+$. This shows that the second homotopy limit is also contractible.

Let us prove that the inclusion $\iota\colon \mathcal{A}^{\ast}_W\to \mathcal{A}^{k}_W$ is left cofinal. For a fixed object $\underline{V}$ in $\mathcal{A}^{k}_W$, choose for every $j$ in $W$ an index $l_j$ with $j\in V_{l_j}$. Notice that we might have $l_j=l_{j'}$ for distinct $j$ and $j'$ in $W$. Define $\mathring{\underline{V}}$  to be the collection of subsets $\underline{V}$ with the component $V_{l_j}$ replaced by $V_{l_j}\cap W$. The set of subsets $\mathring{\underline{V}}$ is a well defined object in the over category $\iota/_{\underline{V}}$. The zig-zag of natural transformations
\[\underline{U}\to\Big(\underline{U}\ \mbox{with}\ U_{l_j}\ \mbox{replaced by }\ U_{l_j}\cup (V_{l_j}\cap W)\Big)\leftarrow \Big(\underline{U}\ \mbox{with}\ U_{l_j}\ \mbox{replaced by }\ (V_{l_j}\cap W)\Big)\to \mathring{\underline{V}}\]
defines a contraction of $\iota/_{\underline{V}}$ onto $\mathring{\underline{V}}$.

Similarly, for a fixed object $\underline{V}$ in $\mathcal{A}^{k}_{W_+}$ we define $\check{\underline{V}}$ by replacing each component $V_{l_j}$ with $(V_{l_j}\backslash o_+)\cup(V_{l_j}\cap W_+)$. A similar zig-zag defines a contraction of the over category of the inclusion $\mathcal{A}^{\ast}_{W_+}\subset \mathcal{A}^{k}_W$ onto $\check{\underline{V}}$.
\end{proof}

\begin{lemma}\label{+vertex}
There is a natural equivalence $R^{\mathcal{A}^{k}}_+\Phi=\holim_{\mathcal{A}^{k}_+}\Phi\big(\Lambda^{J,k}(-)\big)\stackrel{\simeq}{\longrightarrow}T_{J\backslash o}\Phi$.
\end{lemma}

\begin{proof}
The poset $\mathcal{P}_{0}(J_+\backslash o)^k$ clearly includes in $\mathcal{A}^{k}_+=\big(\mathcal{P}_0(J_+)\backslash\mathcal{P}_0(o)\big)^k$. It is sufficient to show that for $k=1$ the inclusion $\iota\colon \mathcal{P}_{0}(J_+\backslash o)\to \mathcal{P}_0(J_+)\backslash\mathcal{P}_0(o)$ is left cofinal. Let $U$ be an object of the target. This is a subset $U$ of $J_+$ that is not included in $o$. Hence the intersection $U\cap (J_+\backslash o)$ must be non-empty, and $U\cap (J_+\backslash o)$ is final in $\iota/_{U}$.
\end{proof}

We explain how to stabilize with respect to $k$, using the technique of \cite[2.12]{calcIII}. We define a functorial zig-zag of maps of $o_+$-cubes 
\begin{equation}\label{stabzigzag}
R^{\mathcal{A}^{k}}\Phi\longrightarrow T_JR^{\mathcal{A}^{k}}\Phi\longleftarrow \mathcal{O}^k\Phi\longrightarrow R^{\mathcal{A}^{k+1}}\Phi
\end{equation}
and we show in \ref{wrongway} that the left pointing map is an equivalence. The $o$-delooping of $P_J\Phi$ is then defined as the homotopy colimit of the final vertices
\[R_o\Phi:=\hocolim\big(R^{\mathcal{A}^{1}}_{o_+}\Phi\longrightarrow\dots\longrightarrow R^{\mathcal{A}^{k}}_{o_+}\Phi\longrightarrow T_JR^{\mathcal{A}^{k}}_{o_+}\Phi\stackrel{\simeq}{\longleftarrow} \mathcal{O}_{o_+}^k\Phi\longrightarrow R^{\mathcal{A}^{k+1}}_{o_+}\Phi\longrightarrow \dots \big).\]
The zig-zag (\ref{stabzigzag}) is induced by the following zig-zag of $o_+$-cubes of categories. Under the isomorphism $\mathcal{P}(o_+)\cong \mathcal{P}(o)\times\mathcal{P}(\{+\})$ we see an $o_+$-cube as a map of $o$-cubes, and we define the zig-zag at a subset $W\subset o$ by
\[\vcenter{\hbox{\xymatrix{\mathcal{A}^{k}_W\\
\mathcal{A}^{k}_{W_+}\ar[u]}}}
\stackrel{pr}{\longleftarrow}
\vcenter{\hbox{\xymatrix{\mathcal{P}_0(J_+)\times\mathcal{A}^{k}_{W}\\
\mathcal{P}_0(J_+)\times\mathcal{A}^{k}_{W_+}\ar[u]}}}
\longrightarrow
\vcenter{\hbox{\xymatrix{\mathcal{A}^{k+1}_{W}\\
\mathcal{A}^{k+1}_{W}\cap(\mathcal{P}_0(J_+)\times\mathcal{A}^{k}_+)\ar[u]
}}}
\longleftarrow
\vcenter{\hbox{\xymatrix{\mathcal{A}^{k+1}_{W}\\
\mathcal{A}^{k+1}_{W_+}\ar[u]
}}}
.\]
The first map projects off the $\mathcal{P}(J_+)$-factor, and the other maps are inclusions. The $o_+$-cube of functors $\mathcal{O}^k\Phi$ of (\ref{stabzigzag}) is defined as
\[\mathcal{O}^{k}_{W}\Phi=
\holim_{\mathcal{A}^{k+1}_{W}}\Phi\big(\Lambda^{J,k+1}(-)\big)\ \ \ \ \ \ \ \ \ \ \ \ \ \ \ \mathcal{O}^{k}_{W_+}\Phi=\holim_{\mathcal{A}^{k+1}_{W}\cap(\mathcal{P}_0(J_+)\times\mathcal{A}^{k}_+)}\Phi\big(\Lambda^{J,k+1}(-)\big)\] 
for every $W\subset o$.

\begin{lemma}\label{wrongway}
The middle map of the zig-zag (\ref{stabzigzag}) induces a weak equivalence $T_JR^{\mathcal{A}^k}\Phi\stackrel{\simeq}{\longleftarrow} \mathcal{O}^k\Phi$ of $o_+$-cubes of functors.
\end{lemma}

\begin{proof}
The argument is similar to \cite[2.12]{calcIII}. For a subset $W\subset o$, the map $T_JR^{\mathcal{A}^k}\Phi\stackrel{}{\leftarrow} \mathcal{O}^k\Phi$ is induced by the inclusion \[\mathcal{C}=\mathcal{P}_0(J_+)\times\mathcal{A}^{k}_{W}\to\mathcal{A}^{k+1}_{W}=\mathcal{B}.\]
Let $\mathcal{D}$ be the sub $G$-poset of $\mathcal{B}$ defined by
\[\mathcal{D}=\bigcup_{j\in W}\big(\mathcal{A}^{1}_j\times (\mathcal{P}_0(J_+)\backslash\{j\})^k\big).\]
The $G$-posets $\mathcal{C}$ and $\mathcal{D}$ cover $\mathcal{B}$, with intersection
\[\mathcal{C}\cap\mathcal{D}=\bigcup_{j\in W}\big(\mathcal{A}_{j}^1\times ((\mathcal{P}_0(J_+)\backslash\{j\})^k\cap\mathcal{A}_{j}^k)\big).\]
By the (non-equivariant) Covering Lemma the induced square
\[\xymatrix@=13pt{\mathcal{O}^k\Phi\ar[r]\ar[d]&T_JR^{\mathcal{A}^k}\Phi\ar[d]\\
\displaystyle\holim_{\mathcal{D}}\Phi\big(\Lambda^{J,k}(-)\big)\ar[r]&\displaystyle\holim_{\mathcal{C}\cap\mathcal{D}}\Phi\big(\Lambda^{J,k}(-)\big)}
\]
is cartesian. Is is therefore enough to show that the lower horizontal map is a $G_W$-equivalence. We claim that both the source and the target of this map are $G_W$-contractible. The argument is the same as in \ref{Aast}, by showing that the inclusions
\[\bigcup_{j\in W}\big(\mathcal{A}^{1,\ast}_j\times (\mathcal{P}_0(J_+)\backslash\{j\})^k\big)\to \mathcal{D}\ \ \ \ \ \ \ \ \ \ \ \bigcup_{j\in W}\big(\mathcal{A}_{j}^{1,\ast}\times ((\mathcal{P}_0(J_+)\backslash\{j\})^k\cap\mathcal{A}_{j}^k)\big)\to \mathcal{C}\cap\mathcal{D}\]
are left $G_W$-cofinal.
A similar argument shows that $T_JR^{\mathcal{A}^k}_{W_+}\Phi\stackrel{}{\leftarrow} \mathcal{O}_{W_+}^k\Phi$ is an equivalence.
\end{proof}

Lemmas \ref{Aast}, \ref{+vertex} and \ref{wrongway} show that if $\Phi$ is $J$-homogeneous and $o$ is a $G$-orbit of $J$, the functor $R_o\Phi$ is an $o$-fold delooping of $\Phi$. The next two Lemmas show that the construction $R_o$ can be iterated over all the orbits $o_1,\dots,o_n$ of $J$, defining a $J$-fold delooping of $\Phi$.

\begin{lemma}\label{RJexc}
For every orbit $o\in J/G$ and every $J$-homogeneous functor $\Phi$, the functor $R_o\Phi$ is $J$-excisive.
\end{lemma}

\begin{proof}
The functor $R_o\Phi$ is defined as the homotopy colimit
\[R_o\Phi:=\hocolim\big(R^{\mathcal{A}^{1}}_{o_+}\Phi\longrightarrow\dots\longrightarrow R^{\mathcal{A}^{k}}_{o_+}\Phi\stackrel{t_J}{\longrightarrow} T_JR^{\mathcal{A}^{k}}_{o_+}\Phi\stackrel{\simeq}{\longleftarrow} \mathcal{O}_{o_+}^k\Phi\longrightarrow R^{\mathcal{A}^{k}}_{o_+}\Phi\longrightarrow \dots \big).\]
If $X$ is a $G$-strongly cocartesian $J_+|_H$-cube, the map $t_J\colon R^{\mathcal{A}^{k}}_{o_+}(X)\to T_JR^{\mathcal{A}^{k}}_{o_+}(X)$ factors through a cartesian $J_+|_H$, by the main argument in the proof of \ref{exisiveapprox}. Thus $R_o\Phi(X)$ is a sequential homotopy colimit of cartesian cubes, and it is therefore cartesian. 
\end{proof}

\begin{lemma}\label{RJreduced}
Let $\Phi$ be $J$-homogeneous. For every orbit $o\in J/G$ and every subset $K\subset J\backslash o$ (possibly non-proper) the functor $P_KR_{o}\Phi$ is contractible.
\end{lemma}

\begin{proof}
Since $P_K$ commutes with $R_{o}$ it is sufficient to show that $R_o P_K\Phi$ is contractible. We show that $R^{\mathcal{A}^k}_{o_+} P_K\Phi$ is contractible for every integer $k$, which proves that $R_o P_K\Phi$ is a homotopy colimit of contractible functors. Let $\mathcal{S}_k$ be the subposet of $\mathcal{P}_0(k\times o)$ of subsets $S$ with the property that the composition with the projection $S\to k\times o\to o$ is surjective. There is an equivariant isomorphism of posets
\[\mathcal{A}_{o_+}\cong \mathcal{P}_0(J_+\backslash o)^k\times\mathcal{S}^k\]
that sends a $k$-tuple of subsets $\underline{U}$ of $J_+$ to the pair $(\underline{U}\backslash o,\{(l,j)|j\in U_l\})$. Here $o$ is removed from $\underline{U}$ componentwise. The inverse sends a pair $(\underline{V},S)$ to the collection $\Gamma(\underline{V},S)$ with $l$-component $V_l\cup p^{-1}(l)$, where $p\colon S\to k\times o\to k$ is the projection. Hence $R^{\mathcal{A}^k}_{o_+}$ decomposes as
\[R^{\mathcal{A}^k}_{o_+}P_K\Phi\cong\holim_{S\in\mathcal{S}_k}
\holim_{\underline{V}\in\mathcal{P}_0(J_+\backslash o)^k}P_K\Phi\big(\Lambda^{J,k}(-)_{\Gamma(\underline{V},S)}\big).\]
We prove that for every fixed $S\in\mathcal{S}_k$, the inner homotopy limit is contractible. By \ref{faces} the $(J_{+}\backslash o)$-cube $\Lambda^{J}(X)_{(-)\cup p^{-1}(l)}$ is strongly cocartesian for every $l$, and by \ref{inclusions} the functor $P_K\Phi$ is $J\backslash o$-excisive ($K$ is a $G$-subset of $J\backslash o$). Hence the inner homotopy limit is equivalent to 
\[P_K\Phi\big(\Lambda^{J,k}(-)_{\Gamma(\underline{\emptyset},S)}\big)=P_K\Phi\big(\Lambda^{J,k}(-)_{\underline{S}}\big)\]
where $\underline{S}$ has $l$-component $p^{-1}(l)$. Now each $p^{-1}(l)$ is a subset of $o$, and at least one is non-empty. For this component $\Lambda^{J}(-)_{S_l}$ is a $S_l$-fold cone, and it is therefore contractible.
\end{proof}

We iterate the constructions $R_o$ over the orbits $o_1,\dots,o_n$ of $J$, defining a functor
\[R_J\Phi:=R_{o_1}\dots R_{o_n}\Phi.\]
The previous Lemmas show that $R_J\Phi\colon \Top_\ast\to \Top_{\ast}^G$ is a $J$-homogeneous $J$-fold delooping of $\Phi$.

The last thing we need to do in order to complete the proof of Theorem \ref{delooping} is to build the lifting $\widehat{\Phi}\colon \Top_\ast\to \Sp_{J}^G$ of $\Phi$ to our category of symmetric $G$-spectra. This is a purely formal maneuver. First of all notice that there is a wrong-way pointing equivalence in our construction of the $J$-delooping
\[\Phi\stackrel{\simeq}{\rightarrow}\holim_{W_1\in\mathcal{P}_0((o_{1})_+)}
\!\!\!\!\hocolim_{k_1}R^{\mathcal{A}^{k_1}}_{W_{1}}
\dots\!\!\!\!
\holim_{W_n\in\mathcal{P}_0((o_{n})_+)}\!\!\!\!\hocolim_{k_n}R^{\mathcal{A}^{k_n}}_{W_n}
\Phi\stackrel{\simeq}{\longleftarrow}
\Omega^{o_1}R_{o_1}
\dots
\Omega^{o_n}R_{o_n}
\Phi\stackrel{\simeq}{\rightarrow}
\Omega^JR_J\Phi.\]
Let us denote the second functor from the left $B_J\Phi$, and the third one $E_J\Phi$. 
The trick of \cite[(0.1)]{calcI} solves this issue, by defining a new sequence of functors
\[\overline{R}_{mJ}\Phi:=\holim\big(
R^{m}_{J}\Phi\stackrel{\simeq}{\longrightarrow}B_{J}R^{m}_J\Phi\stackrel{\simeq}{\longleftarrow}E_{J}R^{m}_J\stackrel{\simeq}{\longrightarrow}\Omega^{J}R^{m+1}_J\Phi
\stackrel{\simeq}{\longrightarrow}\Omega^{J}B_{J}R^{m+1}_J\Phi\stackrel{\simeq}{\longleftarrow}\dots
\big)\]
where the superscript $m$ denotes iteration. This new sequence supports actual structure maps \[\sigma_m\colon\overline{R}_{mJ}\Phi\wedge S^J\stackrel{}{\to}\overline{R}_{(m+1)J}\Phi.\] Observe that $\overline{R}_{mJ}\Phi$ has a canonical $\Sigma_m$-action induced by permuting the iterations of $\overline{R}_{J}$. We want to define $\widehat{\Phi}\colon \Top_\ast\to \Sp_{J}^G$ in such a way that the $G$-space $\overline{R}_{mJ}\Phi(X)$ is the value of the $G$-spectrum $\widehat{\Phi}(X)$ at the $G$-set $mJ$. Again this is formal, by sending a pointed space $X$ to the sequence of $G\times\Sigma_m$-spaces
\[\widehat{\Phi}(X)_m=\Omega^{m\overline{J}}\overline{R}_{mJ}\Phi(X)\]
where $\Omega^{m\overline{J}}$ is the loop space of the direct sum of $m$ copies of the reduced permutation representation of $J$, with $G\times\Sigma_m$-acting by conjugation. The structure maps of this spectrum are the composites
\[(\Omega^{m\overline{J}}\overline{R}_{mJ}\Phi)\wedge S^1\longrightarrow\Omega^{m\overline{J}}(\overline{R}_{mJ}\Phi\wedge S^1)\stackrel{(-)\wedge S^{\overline{J}}}{\longrightarrow} \Omega^{(m+1)\overline{J}}(\overline{R}_{mJ}\Phi\wedge S^J)\stackrel{\sigma_m}{\longrightarrow}\Omega^{(m+1)\overline{J}}\overline{R}_{(m+1)J}\Phi.\]
Since the values of $\widehat{\Phi}$ are fibrant, it follows immediately that $\widehat{\Phi}$ is $J$-homogeneous, concluding the proof of Theorem \ref{delooping}.

\subsection{Equivariant multilinear symmetric functors}\label{multilinear}

We study the behavior of Goodwillie's cross-effect and diagonal functors with respect to equivariant excision. We will use it in the next section to classify strongly homogeneous functors.

Let $J$ be a finite $G$-set, and let us order its orbits $J/G=\{o_1,\dots,o_n\}$. We consider multivariable homotopy functors $M\colon \mathscr{C}^{\times n}\to \mathscr{D}^G$, where $\mathscr{C}^{\times n}$ has the $G$-model structure defined by the product model structure on the categories of $H$-objects $(\mathscr{C}^{\times n})^H\cong (\mathscr{C}^H)^{\times n}$. We will further assume that both $\mathscr{C}$ and $\mathscr{D}$ are pointed.

\begin{defn} Let 
Let $M\colon \mathscr{C}^{\times n}\to \mathscr{D}^G$ be a homotopy functor and let $K$ be a finite $G$-set. We say that $M$ is $K$-excisive (resp. $K$-reduced) in the $i$-variable if for every object $\underline{c}\in \mathscr{C}^{\times(n-1)}$ the functor
\[M(c_1,\dots,c_{i-1},(-),c_i,\dots,c_{n-1})\colon\mathscr{C}{\longrightarrow} \mathscr{D}^G\]
is $K$-excisive (resp. $K$-reduced).
We say that $M\colon \mathscr{C}^{\times n}\to \mathscr{D}^G$ is $J$-multilinear if it is $o_i$-excisive and $1$-reduced in the $i$-variable, for every orbit $o_i$ of $J$.
\end{defn}

Let $\Delta\colon \mathscr{C}\to \mathscr{C}^{\times n}$ be the diagonal functor. Precomposition by $\Delta$ defines a functor
\[\Delta^\ast\colon Fun_h(\mathscr{C}^{\times n},\mathscr{D}^{G})\longrightarrow Fun_h(\mathscr{C},\mathscr{D}^G)\]
between the categories of homotopy functors.
The goal of the next results is to prove that $\Delta^\ast$ sends $J$-multilinear functors to $J$-homogeneous functors.

\begin{prop}\label{diag} Let $J$ be a finite $G$-set, and let $K_1,\dots,K_n$ be a collection of finite $G$-sets indexed over the orbits $J/G=\{o_1,\dots,o_n\}$. Let $M\colon \mathscr{C}^{\times n}\to \mathscr{D}^G$ be a homotopy functor which is $K_i$-excisive in the $i$-variable for every $1\leq i\leq n$. Then the diagonal functor $\Delta^\ast M\colon \mathscr{C}\to \mathscr{D}^G$ is $(K_1\amalg\dots\amalg K_n)$-excisive. In particular if $M$ is $o_i$-excisive in the $i$-variable for every orbit $o_i\in J/G$, the functor $\Delta^\ast M$ is $J$-excisive.
\end{prop}
The proof of this proposition is based on the following lemma, analogous to \cite[3.3]{calcII}.

\begin{lemma}\label{lemmamulti}
Let $K$ be a proper $G$-subset of $J$, let $\Phi\colon \mathscr{C}\to \mathscr{D}^G$ be a $K$-excisive homotopy functor and let $X\in\mathscr{C}_{a}^{\mathcal{P}(J_+|_H)}$ be a $G$-strongly cocartesian $J_+|_H$-cube, for some subgroup $H\leq G$. Then for every subset $T\subset J_+$, the canonical map
\[\Phi(X_T)\stackrel{\simeq}{\longrightarrow}\holim_{\substack{ T\subset U\subset J_+|\\ |U|\geq |J\backslash K|+1}}\Phi(X_U)\]
is an $H_T$-equivalence.
\end{lemma}

\begin{proof}
The proof is analogous to \cite[3.3]{calcII}, by induction on the cardinality of the complement $J_+\backslash T$ and by replacing the Covering Lemma \cite[1.9]{calcII} with the Equivariant Covering Lemma \ref{coverings}. We repeat the argument. Let $\mathcal{E}_T$ be the subposet of $\mathcal{P}(J_+)$
\[\mathcal{E}_T=\{T\subset U\subset J_+\ |\ |U|\geq |J\backslash K|+1\}.\]
When $|T|\geq |J\backslash K|+1$ the set $T$ is initial in $\mathcal{E}_T$. Hence the map of the statement is a $G_T$-equivalence by cofinality \cite[2.25]{Gdiags}. Now suppose that $|T|\leq|J\backslash K|$ and that
\[\Phi(X_S)\stackrel{\simeq}{\longrightarrow}\holim_{\substack{ S\subset U\subset J_+|\\ |U|\geq |J\backslash K|+1}}\Phi(X_U)\]
is a $G_S$-equivalence for every subset $S$ of $J_+$ with $|S|>|T|$. Consider the commutative diagram
\[\xymatrix{\Phi(X_T)\ar[r]\ar[d]_{\simeq}&\displaystyle\holim_{U\in\mathcal{E}_T}\Phi(X_U)\ar[d]\\
\displaystyle\holim_{T\subsetneqq S\subset T\cup K_+}\Phi(X_S)\ar[r]_-{\simeq}&\displaystyle\holim_{T\subsetneqq S\subset T\cup K_+}\holim_{V\in \mathcal{E}_S}\Phi(X_V)
}.\]
The bottom horizontal map is an equivalence by the inductive hypothesis. The left vertical map is an $H_T$-equivalence because $\Phi(X_{T\cup(-)})$ is a homotopy cartesian $K_+|_{H_T}$-cube in $\mathscr{D}$. This can be seen as follows. If $T$ intersects $K_+$ the maps $\Phi(X_{T\cup U})\to \Phi(X_{T\cup U\cup k})$ are equivalences for a chosen $k\in K_+\cap T$, and thus $\Phi(X_{T\cup(-)})$ is cartesian. If $T$ does not intersect $K_+$, the cube $X_{T\cup(-)}$ is $G$-strongly cocartesian by \ref{facesstronglycocart}, and its image is cartesian because $\Phi$ is $K$-excisive.

It remains to prove that the right vertical map is an equivalence.
The $G_T$-poset $\{T\subsetneqq S\subset T\cup K_+\}$ is isomorphic to $\mathcal{P}_0(K_+)$, and $\mathcal{E}_S=\bigcap_{k\in S\backslash T}\mathcal{A}_k$, where $\mathcal{A}_k$ is the subposet of $\mathcal{P}(J_+)$ defined by
\[\mathcal{A}_k=\{V\subset J_+\ | \ \ |V|\geq |J\backslash K|+1\ , \ T\cup k\subset V \}.\]
The right vertical map is then an equivalence by the Covering Lemma \ref{coverings}, provided the collection of posets $\{\mathcal{A}_k\}_{k\in K_+}$ covers $\mathcal{E}_T$. This is indeed the case since for cardinality reasons, if $V$ belongs to $\mathcal{E}_T$ it must intersect $K_+$. Hence $V$ contains both $T$ and at least one element $k$ of $K_+$, and thus $V\in \mathcal{A}_k$.
\end{proof}

\begin{proof}[Proof of \ref{diag}]
Let us denote $\amalg\underline{K}:=K_1\amalg\dots \amalg K_n$. Let $X$ be a $G$-strongly cocartesian $(\amalg\underline{K})_+|_H$-cube in $\mathscr{C}$ for some subgroup $H$ of $G$. By applying Lemma \ref{lemmamulti} to every variable of $M$, the homotopy limit of $\Delta^\ast M(X)$ over $\mathcal{P}_0(\amalg \underline{K}_+)$ decomposes as
\[\holim_{T\in\mathcal{P}_0(\amalg \underline{K})}M(X_T,\dots,X_T)\stackrel{\simeq}{\longrightarrow}\holim_{T\in\mathcal{P}_0(\amalg \underline{K})}\holim_{(U_{1},\dots, U_{n})\in \mathcal{E}_{\underline{T}}}M(X_{U_1},\dots X_{U_n})\]
where the inner homotopy limit on the right is taken over the sub-poset $\mathcal{E}_{\underline{T}}$ of $\mathcal{P}_0(\amalg \underline{K})^{n}$ of subsets $U_{1},\dots, U_{n}$ with the property that $T\subset U_i\subset \amalg \underline{K}_+$ and $|U_{i}|\geq |\amalg \underline{K}_+\backslash K_{i}|$ for every $1\leq i\leq n$. It follows that $\Delta^\ast M$ is $\amalg \underline{K}$-excisive if we can prove that the composition
\[M(X_{\emptyset},\dots,X_{\emptyset})\longrightarrow\holim_{(U_{1},\dots, U_{n})\in \mathcal{E}_{\underline{\emptyset}}}M(X_{U_1},\dots X_{U_n})\longrightarrow\holim_{T\in\mathcal{P}_0(\amalg \underline{K})}\holim_{(U_{1},\dots, U_{n})\in \mathcal{E}_{\underline{T}}}M(X_{U_1},\dots X_{U_n})\]
is an equivalence.
The first map is an equivalence again by Lemma \ref{lemmamulti}. The second map is induced by the cover $\{\mathcal{U}_k\}_{k\in \amalg \underline{K}_+}$ of $\mathcal{E}_{\underline{\emptyset}}$ defined by the posets
\[\mathcal{U}_k=\{(U_1,\dots,U_n)\in \mathcal{E}_{\underline{\emptyset}}| \ k\in U_i \ \mbox{for all} \ 1\leq i\leq n\}.\]
It is an equivalence by \ref{coverings} if we can show that the collection of the $\mathcal{U}_k$'s covers $\mathcal{E}_{\underline{\emptyset}}$. This is the case if every element $(U_1,\dots,U_n)\in \mathcal{E}_{\underline{\emptyset}}$ has non-empty intersection $\bigcap_{i=1}^nU_i$. This is equivalent to proving that the complement
\[\amalg\underline{K}_+\backslash\bigcap_{i=1}^nU_i=\bigcup_{i=1}^n(\amalg\underline{K}_+\backslash U_i)\]
is a proper subset of $\amalg\underline{K}_+$. Since $(U_1,\dots,U_n)$ belongs to $\mathcal{E}_{\underline{\emptyset}}$, the cardinality of the union is at most
\[|\bigcup_{i=1}^n(\amalg\underline{K}_+\backslash U_i)|\leq \sum_{i=1}^n|\amalg\underline{K}_+\backslash U_i|\leq \sum_{i=1}^n|K_{i}|=|\amalg\underline{K}|<|\amalg\underline{K}_+|.\]
\end{proof}

\begin{cor}\label{diagJlin}
Let $M\colon \mathscr{C}^{\times n}\to \mathscr{D}^G$ be a homotopy functor. If $M$ is $J$-multilinear the diagonal $\Delta^\ast M\colon \mathscr{C}\to \mathscr{D}^G$ is strongly $J$-homogeneous.
\end{cor}

\begin{proof}
The diagonal functor $\Delta^\ast M$ is $J$-excisive by \ref{diag}. The argument of \cite[3.1]{calcIII} shows that $P_{n-1}\Delta^\ast M$ is contractible, and thus that $\Delta^\ast M$ is $n$-reduced.
\end{proof}

\begin{ex} Corollary \ref{diagJlin} gives another proof that the functor $\mathcal{T}_{E}^{n}(X)=E\wedge X^{\wedge n}$ from Proposition \ref{smashfctr} is $nG$-homogeneous, which does not involve the calculation of its $nG$-excisive approximation.
\end{ex}

Goodwillie's cross-effect defines a functor in the opposite direction of the diagonal. Let $S=\{s_1,\dots,s_n\}$ be a finite set. The $S$-cross effect is the functor
\[cr_{S}\colon Fun_h(\mathscr{C},\mathscr{D}^G)\longrightarrow Fun_h(\mathscr{C}^{\times S},\mathscr{D}^G)\]
that sends $\Phi\colon \mathscr{C}\to\mathscr{D}^G$ to the multivariable functor $cr_S\Phi$ defined at a collection of objects $c_{s_1},\dots,c_{s_n}\in\mathscr{C}$ by the total homotopy fiber
\[cr_S\Phi(c_{s_1},\dots,c_{s_n})=\hofib\big(\Phi(\bigvee_{s\in S}c_s)\longrightarrow
\holim_{U\in\mathcal{P}_0(S)}\Phi(\bigvee_{s\in S\backslash U}c_s)\big).\]
Indexing the cross-effect on a general finite set $S$ makes it easier to state the following proposition. Let us remark that $cr_S\Phi$ is $1$-reduced in every variable.

\begin{prop}\label{crossexcisive}
Let $J$ be a finite $G$-set with $n$-orbits $J/G=\{o_1,\dots,o_n\}$, 
let $\Phi\colon \mathscr{C}\to\mathscr{D}^G$ be a $J$-excisive homotopy functor and let $W$ be a subset of $J/G$. Then $cr_W\Phi$ is $(J\backslash \amalg_{w\in W}w)\amalg o$-excisive in the $o$-variable for every $o\in W$, and $cr_{n+1}\Phi$ is trivial. In particular $cr_{J/G}\Phi$ is $J$-multilinear.
\end{prop}

\begin{proof}
We prove the proposition by induction on the cardinality of $W$. If $W$ consists of only one orbit $o$, then
\[cr_{\{o\}}\Phi(c)=\hofib\big(\Phi(c)\to \Phi(\ast)\big)\]
is $J$-excisive in its unique variable.

Now suppose that $W$ contains at least two orbits, and let us study the $o$-variable of $cr_W\Phi$. Choose an orbit $w\in W$ different than $o$. By the proof of \cite[3.3]{calcIII} for every object $\underline{c}$ of $\mathscr{C}^{\times W}$ there is an equivalence
\begin{equation}\label{crind}
(cr_W\Phi\big)(\underline{c})\simeq \Big(cr_{W\backslash w}\hofib\big(\Phi(c_w\vee (-))\to\Phi(-)\big)\Big)(\underline{c}\backslash c_w)
\end{equation}
where $\underline{c}\backslash c_w$ is the object of $\mathscr{C}^{\times (W\backslash w)}$ obtained by removing the $w$-component from $\underline{c}$. We show that the functor $\Phi_{c_w}=\hofib\big(\Phi(c_w\vee (-))\to\Phi(-)\big)$ is $(J\backslash w)$-excisive. By the inductive hypothesis and (\ref{crind}) this shows that $cr_W\Phi$ is
\[((J\backslash w)\backslash \coprod\limits_{x\in W\backslash w}x)\amalg o=J\backslash (\coprod\limits_{x\in W}x)\amalg o\]
excisive in the $o$-variable, which concludes the proof.

Let $H$ be a subgroup of $G$ and let $X$ be a $G$-strongly cocartesian $(J_+\backslash w)|_H$-cube. We need to show that $\Phi_{c_w}(X)$ is cartesian. Let $Y$ be the $(J_+\backslash w\amalg 1)|_H$-cube defined by the pinch map $c_w\vee X\to X$. We start by showing that $Y$ is $G$-strongly cocartesian. Let $T$ be a subset of $J_+\backslash w\amalg 1$ which does not belong to the star category $St(J_+\backslash w\amalg 1)$. If $T$ does not contain $1$, the restriction $Y|_{\mathcal{P}(T)}$ is equal to $X|_{\mathcal{P}(T)}\vee c_w$ which is cocartesian. If $T$ does contains $1$, write $T=S\amalg 1$ for a subset $S$ of $J_+\backslash w$. The homotopy colimit of $Y$ over $\mathcal{P}_1(S\amalg 1)$ is the homotopy pushout
\[
%\hocolim\Big(\hocolim_{\mathcal{P}_1(S)}Y_{(-)\amalg 1}\leftarrow \hocolim_{\mathcal{P}_1(S)}Y\rightarrow Y_{S} 
%\Big)=
\hocolim\Big(\hocolim_{\mathcal{P}_1(S)}X\longleftarrow \hocolim_{\mathcal{P}_1(S)}c_w\vee X\longrightarrow c_w\vee X_S 
\Big)\]
which is equivalent to $X_S=Y_{S\amalg 1}$. Hence $Y|_{\mathcal{P}(S)}$ is cocartesian, showing that $Y$ is a $G$-strongly cocartesian $(J_+\backslash w\amalg 1)|_H$-cube.

Since the map $J\to J\backslash w\amalg 1$ that collapses $w$ to $1$ is bijective on orbits, $\Phi$ is $(J_+\backslash w\amalg 1)$-excisive (see \ref{isoonorb}). Therefore $\Phi(Y)$ is cartesian, and so is the square
\[\xymatrix{\Phi(c_w\vee X_{\emptyset})\ar[r]\ar[d]&
\Phi(X_{\emptyset})\ar[d]\\
\displaystyle\holim_{\mathcal{P}_0(S)}\Phi(c_w\vee X)\ar[r]&\displaystyle\holim_{\mathcal{P}_0(S)}\Phi(X)
}.\]
The map between the horizontal homotopy fibers of the square $\Phi_{c_w}( X_{\emptyset})\to\holim_{\mathcal{P}_0(S)}\Phi_{c_w}(X)$ is then an equivalence, proving that $\Phi_{c_w}(X)$ is cartesian.

It remains to prove that $cr_{n+1}\Phi$ is trivial. This follows immediately from \cite[3.3]{calcIII}, since $\Phi$ is $n$-excisive by \ref{isoonorb}. 
\end{proof}

We conclude the section by discussing symmetric homotopy functors. This works exactly as in \cite[\S 3]{calcIII}.
The group of automorphisms $\Sigma_{n}$ of the orbits set $J/G=\{o_1,\dots,o_n\}$ acts on the category $\mathscr{C}^{\times n}$ by permuting the components. This defines a category with $\Sigma_n$-action $\Sigma_{n}\to Cat$, and we denote $\Sigma_{n}\wr \mathscr{C}^{\times n}$ its Grothendieck construction. We recall that a symmetric functor is a functor $M\colon \Sigma_{n}\wr \mathscr{C}^{\times n}\to \mathscr{D}^G$, or in the language of equivariant diagrams it is a $\Sigma_n$-diagram in $\mathscr{D}^G$ shaped over $\mathscr{C}^{\times n}$. For example, the $n$-cross-effect of $\Phi\colon \mathscr{C}\to \mathscr{D}^G$ is a symmetric functor, where the $\Sigma_n$-structure is induced by the maps
\[\holim_{U\in\mathcal{P}_0(\underline{n})}\Phi(\bigvee_{i\in \underline{n}\backslash U}c_i)\longrightarrow\holim_{U\in\mathcal{P}_0(\underline{n})}\Phi(\bigvee_{i\in \underline{n}\backslash \sigma^{-1}(U)}c_i)\]
induced by a permutation $\sigma\in\Sigma_n$.
\begin{defn}
A $J$-multilinear symmetric homotopy functor is a symmetric functor $M\colon \Sigma_{n}\wr \mathscr{C}^{\times n}\to \mathscr{D}^G$ such that the underlying functor 
\[\mathscr{C}^{\times n}\longrightarrow\Sigma_{n}\wr \mathscr{C}^{\times n}\longrightarrow \mathscr{D}^G\]
is a $J$-multilinear homotopy functor. The category of $J$-multilinear homotopy functors is denoted $Lin_J(\Sigma_{n}\wr \mathscr{C}^{\times n},\mathscr{D}^G)$.
\end{defn}

The diagonal $\Delta^\ast M$ of a symmetric functor $M\colon \Sigma_{n}\wr \mathscr{C}^{\times n}\to \mathscr{D}^G$ inherits a $\Sigma_n$-action, and we write $\Delta_{h\Sigma_n}^\ast M\colon \mathscr{C}\to\mathscr{D}^G$ for the functor obtained by taking homotopy orbits
\[(\Delta_{h\Sigma_n}^\ast M)(c)=M(c,\dots,c)_{h\Sigma_n}.\]
If the target $G$-model category $\mathscr{D}$ is $G$-stable (for example $G$-spectra), homotopy orbits preserve homogeneous functors (see \ref{PJcommute}). By Corollary \ref{diagJlin} and Proposition \ref{crossexcisive} the functor $\Delta_{h\Sigma_n}^\ast$ and the cross-effect restrict to functors
\[\xymatrix@C=35pt{\Delta_{h\Sigma_n}^\ast\colon Lin_J(\Sigma_{n}\wr \mathscr{C}^{\times n},\mathscr{D}^G)\ar@<.8ex>[r]&\ar@<.8ex>[l] H_{sJ}(\mathscr{C},\mathscr{D}^G)\colon cr_n}\]
between the category of $J$-multilinear symmetric homotopy functors to the category of strongly $J$-homogeneous functors.

\subsection{The classification of strongly homogeneous functors}\label{sec:class}

Let $J$ be a finite $G$-set with $n$-orbits $J/G=\{o_1,\dots,o_n\}$. We combine the results of \S\ref{delooping} and \S\ref{multilinear} to classify strongly $J$-homogeneous functors $\Top_\ast\to\Top_\ast^G$ by $J$-spectra with na\"{i}ve $\Sigma_n$-action. The category of $J$-spectra $\Sp_{J}^G$ is the one introduced in \S\ref{delooping}.

We start by observing that by the results of \cite{Blumberg} a $1$-homogeneous functor $\Phi\colon \Top_\ast\to\Top_\ast^G$ which is finitary and $sSet$-enriched is classified by a na\"{i}ve $G$-spectrum only on the full subcategory of spaces with trivial $G$-action. By this we mean that there is a natural equivalence
\[\Phi(X){\simeq}\Omega^\infty(\Phi(\mathbb{S})\wedge X)\]
if $X$ has the trivial $G$-action. Here $\Phi(\mathbb{S})\in \Sp_{1}^G$ is the $G$-spectrum defined by the $G$-spaces $\Phi(\mathbb{S})_n=\Phi(S^n)$. However, a finitary enriched $G$-homogeneous functor $\Phi\colon \Top_\ast\to\Top_\ast^G$ admits an equivalence
\[\Phi(X)\simeq\Omega^{\infty G}(\Phi(\mathbb{S})\wedge X)\]
for every pointed $G$-space $X$ (the value of the $G$-spectrum $\Phi(\mathbb{S})\in \Sp_{G}^G$ at the $G$-set $G$ is $\Phi(S^G)$, the value of $\Phi$ at the regular representation). This is telling us that the $G$-set $J$ has an effect both on the type of equivariant deloopings and on the subcategory of $\Top_{\ast}^G$ over which the functor is determined by the coefficients spectrum.

\begin{defn}
A $J$-CW-complex is a $G$-CW-complex which is built out of equivariant cells of the form $T_+\wedge D^k$, where $T$ is a transitive $G$-set which admits a $G$-equivariant map $J\to T$. We let $\Top_\ast\langle J\rangle$ be the full subcategory of $\Top_{\ast}^G$ of pointed spaces which are weakly $G$-equivalent to finite $J$-CW-complexes. In particular $\Top_\ast\langle \underline{n}\rangle=\Top_\ast$ and $\Top_\ast\langle nG\rangle=\Top^{G}_\ast$ for every $n\in\mathbb{N}$.
\end{defn}

Let $H_{sJ}(\Top_\ast,\Top^{G}_\ast)$ be the category of strongly $J$-homogeneous homotopy functors $\Top_\ast\to \Top^{G}_\ast$, and let $\Sigma_n$-$\Sp_{J}^G$ be the category of functors $\Sigma_n\to \Sp_{J}^G$ with the projective model structure. An object of $\Sigma_n\to \Sp_{J}^G$ is a symmetric $G\times \Sigma_n$-spectrum, and a map in $\Sigma_n$-$\Sp_{J}^G$ is an equivalence if it is an equivalence when it is restricted to a map of $G$-spectra. We define a pair of equivalence preserving functors
\[\xymatrix@C=35pt{\Sigma_n\mbox{-}\Sp_{J}^G\ar@<.8ex>[r]&\ar@<.8ex>[l] H_{sJ}(\Top_\ast,\Top^{G}_\ast)}.\]
One sends a $G\times \Sigma_n$ spectrum $E$ to the functor $\Omega^{\infty J}(E\wedge X^{\wedge n})_{h\Sigma_n}$. The other sends a homotopy functor $\Phi$ to the spectrum $E_{\Phi}$ defined in level $k$ by the $G\times \Sigma_n$-space
\[(E_{\Phi})_k=\Omega^{k\overline{V}_{n}}cr_{n}\Phi(\Delta S^k)\]
where $\overline{V}_{n}$ is the reduced standard representation of $\Sigma_{n}$.

\begin{theorem}\label{classification} Let $J$ be a finite $G$-set with $n$-orbits. For every $sSet$-enriched finitary strongly $J$-homogeneous homotopy functor $\Phi\colon \Top_\ast\to \Top_{\ast}^G$ there is a natural equivalence
\[\Phi(X)\simeq \Omega^{\infty J}(E_\Phi\wedge X^{\wedge n})_{h\Sigma_{n}}\]
for every object $X\in \Top_\ast\langle J\rangle$. In particular if $\Phi$ is $nG$-homogeneous, there is a $G$-equivalence
\[\Phi(X)\simeq \Omega^{\infty G}(E_\Phi\wedge X^{\wedge n})_{h\Sigma_{n}}\]
for every pointed $G$-space $X$. 
Conversely, for every $E\in \Sigma_n$-$\Sp_{J}^G$ there is a natural equivalence
\[E\simeq cr_n ((E\wedge (-)^{\wedge n})_{h\Sigma_n})(\Delta S^0).\]
\end{theorem}

\begin{rem}\label{rem:classification}
\begin{enumerate}
\item For transitive $G$-sets, Theorem \ref{classification} is part of the classification of $G$-linear functors on non-complete universes of \cite[\S 3.4]{Blumberg}.
\item Theorem \ref{classification} suggests a Quillen equivalence $\Sigma_n\mbox{-}\Sp_{J}^G\rightleftarrows H_{sJ}(\Top_\ast,\Top^{G}_\ast)$, where the equivalences on $H_{sJ}(\Top_\ast,\Top^{G}_\ast)$ are the natural transformations $\Phi\to\Psi$ that induce equivalences $\Phi(X)\to \Psi(X)$ for every $X\in \Top_\ast\langle J\rangle$. A precise model categorical formulation of this Quillen equivalence for the trivial group can be found in \cite[\S 8]{BCR} and \cite[6.17]{BR}.
\item Strongly $J$-homogeneous functors are in particular $n$-homogeneous, and $J$-multilinear functors are $n$-multilinear, by \ref{isoonorb}. Goodwillie's proof of the classification of $n$-homogeneous functors \cite{calcIII} shows that for every strongly $J$-homogeneous functor $\Phi\colon\Top_\ast\to\Top_{\ast}^G$ there is an equivalence
\[\Phi(X)\simeq \Omega^{\infty}(E_\Phi\wedge X^{\wedge n})_{h\Sigma_{n}}\]
on the category $\Top_\ast$ of pointed $G$-spaces with trivial $G$-action. The extra property that $\Phi$ is $J$-excisive gives us $J$-equivariant deloopings, and extends this equivalence to the category $\Top_\ast\langle J\rangle$.
\item At this point the reader might be wondering about the classification of $J$-homogeneous functors (which are not necessarily $n$-reduced). The calculation of the derivatives of the symmetric indexed powers of \ref{DnGsym} and of the identity functor suggests that $nG$-homogeneous functors should be classified by a homotopy theory of ``equivariant'' spectra $Sp^{\mathscr{F}_{G,n}}$ parametrized by the Burnside category of the category $\mathscr{F}_{G,n}$ of $H$-sets with $n-1$ orbits and the trivial $H$-set $n$, where $H$ runs through the subgroups of $G$. Barwick's theory of spectral Mackey functors of \cite{Mackey} provides a convenient model for this homotopy theory in the infinity-categorical context. An $nG$-homogeneous functor $\Phi$ should then decompose as
\[\Phi(X)\simeq\Omega^{\infty G}(C_{\Phi}\wedge \mathcal{N}(\mathbb{S}_G\wedge X))_{h\mathscr{F}_{G,n}}\]
where $C_{\Phi}$ is such a Mackey functor, and $\mathcal{N}\colon Sp^{G}\to Sp^{\mathscr{F}_{G,n}}$ and  $(-)_{h\mathscr{F}_{G,n}}\colon Sp^{\mathscr{F}_{G,n}}\to Sp^{G}$ are respectively the generalizations of the multiplicative norm of \cite{HHR} and of the homotopy orbits functor. In the case of the identity functor the coefficient spectrum $C_\Phi$ decomposes as a consequence of the Snaith-splitting (see \ref{derI}), but we do not expect this to happen in general. Similarly, $J$-homogeneous functors should be classified by an analogous construction for a version of $\mathscr{F}_{G,n}$ where the orbit types of the $H$-sets depend on $J$.
This classification will be investigated in later work.
\end{enumerate}
\end{rem}

%\begin{ex}\label{coeffindexedsmash} Consider the $(K\times G)$-homogeneous functor $\mathcal{S}^{K}(X)=E\wedge X^{\wedge K}$ from pointed spaces to $G$-spectra from example. The $G$-set $K\times G$ is canonically isomorphic to the $G$-set $|K|\times G$ (via the map $(k,g)\to (g^{-1}k,g)$), and the orbit set $(K\times G)/_G$ is canonically identified with the underlying set with trivial action $|K|$.
%By the classification theorem, there is a coefficient $G$-spectrum $E_{\mathcal{S}^{K}}$ and a weak equivalence 
%\[E\wedge X^{\wedge K}\simeq (E_{\mathcal{S}^{K}}\wedge X^{|K|})_{h\Sigma_{|K|}}\]
%for every pointed $G$-space $X$. It is easy to verify that the equivalence above holds for the $G$-spectrum with $\Sigma_{|K|}$-action
%\[E_{\Phi^{K}}=(\Sigma_{|K|})_+\wedge E\]
%Here an automorphism $\delta$ in $\Sigma_{|K|}$ acts on the left smash factor by $\delta\cdot \sigma=\delta\circ\sigma$, and $G$ acts diagonally, with action on $\Sigma_{|K|}$ defined by sending an automorphism $\sigma$ to $\sigma\circ g^{-1}$.
%
%More generally, for any fixed subgroup $A$ of the group $Aut_G(K)$ of equivariant automorphisms of $K$, and any $G$-spectrum $E$ with $A$-action, the functor 
%\[(E\wedge X^{\wedge K})_{h A}\]
%is also $K\times G$-excisive. Here $A$ acts diagonally, by permuting the $K$-indexed smash factors. The classifying spectrum for this functor is the $G$-spectrum with $\Sigma_{|K|}$-action
%\[(\Sigma_{|K|})_+\wedge_A E\]
%where both $A$ and $G$ act diagonally, on $\Sigma_{|K|}$ respectively by $a\circ\sigma$ and $\sigma\circ g^{-1}$.
%\end{ex}

The proof of \ref{classification} will occupy the rest of the section. We start by classifying strongly $J$-homogeneous homotopy functors $\Top_\ast\to \Sp^{G}_J$ to the category of $J$-spectra, and we will later deduce the result for functors to $\Top^{G}_\ast$ by means of the delooping theorem \ref{delooping}. Let us remind that $J$ has $n$-orbits.

\begin{lemma}\label{assembly}
Let $M\colon (\Top_\ast)^{\times n}\to \Sp_{J}^G$ be a $sSet$-enriched finitary homotopy functor which is $J$-multilinear. Then the assembly map
\[M(S^0,\dots,S^0)\wedge X^{\wedge n}\stackrel{\simeq}{\longrightarrow} M(X,\dots,X)\]
is an equivalence for every $X\in \Top_\ast\langle J\rangle$.
\end{lemma} 

\begin{proof}
By looking at the variables of $M$ one at the time, it is sufficient to show that the assembly map
\[\Phi(S^0)\wedge X\stackrel{\simeq}{\longrightarrow} \Phi(X)\]
is an equivalence for every $X\in \Top_\ast\langle J\rangle$, where $\Phi\colon \Top_\ast\to \Sp_{J}^G$ is a $sSet$-enriched finitary homotopy functor which is $o$-homogeneous, for an orbit $o\in J/G$. Since $\Phi$ is a finitary homotopy functor we can assume that $X$ is a finite $J$-CW-complex. The proof is by induction on the skeleton of $X$.

If $X$ is zero-dimensional, it is a pointed finite $G$-set $X=K_+$. The assembly map fits into a commutative diagram
\[\xymatrix{\Phi(S^0)\wedge K_+\ar[r]\ar[d]& \Phi(K_+)\ar[dl]\\
\prod_K\Phi(S^0)
}.\]
We show that the diagonal and the vertical maps in this triangle are equivalences.
Since $K_+$ is a $J$-CW-complex the orbit decomposition of $K$ is of the form $K=T_1\amalg\dots\amalg T_k$ where each $T_i$ receives a $G$-equivariant map form $J$. By composing with the inclusion of $o$ into $J$ we obtain $G$-equivariant maps $o\to T_i$ for every $1\leq i\leq k$. It follows that $\Phi$ is $T_i$-excisive for every $i$. The diagonal map above decomposes as
\[\Phi(K_+)=\Phi(\bigvee_{i=1}^n\bigvee_{T_i}S^0)\longrightarrow
\prod_{i=1}^n\Phi(\bigvee_{T_i}S^0)\longrightarrow \prod_{i=1}^n\prod_{T_i}\Phi(S^0)=\prod_K\Phi(S^0).\]
The second map is an equivalence because $\Phi$ is in particular $1$-excisive, and therefore it sends wedges to products. The first map is an equivalence because $\Phi$ is $T_i$-excisive for every $i$, and therefore it sends $T_i$-indexed wedges to $T_i$-indexed products (see e.g. \cite[3.17]{Gdiags}).
Similarly, the vertical map of the triangle
\[\Phi(S^0)\wedge K_+\longrightarrow \prod_{K}\Phi(S^0)\]
is an equivalence because the identity functor on $\Sp^{G}_{J}$ is $T_i$-homogeneous for every $i$. This is because the permutation representation spheres $S^{T_i}$ are invertible in $\Sp^{G}_{J}$ (see \cite[3.26]{Gdiags}).

Now suppose that the assembly map is an equivalence on $J$-CW-complexes of dimension $n$, and let $X$ have dimension $n+1$. The inclusion of the $n$-skeleton $X^{(n)}\to X$ induces a commutative diagram
\[\xymatrix{\Phi(S^0)\wedge X^{(n)}\ar[r]\ar[d]&\Phi(S^0)\wedge X\ar[r]\ar[d]&\Phi(S^0)\wedge S^{n+1}\wedge K_+\ar[d]\\
\Phi(X^{(n)})\ar[r]&\Phi(X)\ar[r]&\Phi(S^{n+1}\wedge K_+)
}\]
where the vertical maps are assemblies, and $K$ is some finite $G$-set whose orbits receive a $G$-map from $J$. Both rows are fiber sequences. The top one because $\Sp_{J}^G$ is a stable category, and the bottom one because $\Phi$ is $1$-homogeneous. The left vertical map is an equivalence by the inductive hypothesis, and it is therefore sufficient to show that the right vertical map is an equivalence. This map decomposes as
\[\Phi(S^0)\wedge S^{n+1}\wedge K_+\longrightarrow \Phi(K_+)\wedge S^{n+1}\longrightarrow \Phi(S^{n+1}\wedge K_+)\]
where the first map is an equivalence by the same argument used in the zero-dimensional case above. The second map is an equivalence by an easy induction on $n$, using that $\Phi$ is $1$-homogeneous.
\end{proof}

\begin{prop}\label{classspectra}
Let $\Phi\colon \Top_\ast\to \Sp_{J}^G$ be a $sSet$-enriched finitary strongly $J$-homogeneous homotopy functor. There is a natural equivalence
\[\Phi(X)\simeq ((cr_n\Phi)(S^0,\dots,S^0)\wedge X^{\wedge n})_{h\Sigma_{n}}\]
for every object $X\in \Top_\ast\langle J\rangle$.
\end{prop}

\begin{proof}
Consider the commutative diagram
\[\xymatrix@C=50pt{H_{sJ}(\Top_\ast,\Sp_{J}^G)\ar@<1ex>[r]^-{cr_{n}}\ar[d]& Lin_J(\Sigma_{n}\wr \Top_{\ast}^{\times n}, \Sp_{J}^G)\ar@<1ex>[l]-^{\Delta^{\ast}_{h\Sigma_{n}}}\ar[d]\\
H_{n}(\Top^{G}_\ast,\Sp_{J}^G)\ar@<1ex>[r]^-{cr_{n}}& Lin_{n}(\Sigma_{n}\wr (\Top^{G}_\ast)^{\times n}, \Sp_{J}^G)\ar@<1ex>[l]^-{\Delta^{\ast}_{h\Sigma_{n}}}
}\]
where the vertical maps are inclusions of subcategories by \ref{isoonorb}.
The fundamental property of the category of spectra allowing Goodwillie's classification of $n$-homogeneous functors $\Top_\ast\to \Sp$ is that $\Sp$ is a stable model category. Since the category of $J$-spectra $\Sp_{J}^G$  is also stable, the proof of \cite[3.5]{calcIII} shows that any strongly $J$-homogeneous $\Phi\colon \Top_\ast\to \Sp^G$ is of the form
\[\Phi(X)\simeq (cr_n\Phi)(X,\dots X)_{h\Sigma_n}\]
for every pointed $G$-space $X$. By Lemma \ref{assembly} this decomposes further as
\[\Phi(X)\simeq ((cr_n\Phi)(S^0,\dots,S^0)\wedge X^{\wedge n})_{h\Sigma_{n}}\]
if $X$ belongs to $\Top_\ast\langle J\rangle$.
\end{proof}

%\begin{rem}
%A na\"{i}ve attempt of classification of $J$-homogeneous functors $\Phi\colon \Top_\ast\to Sp_{J}^G$ would be
%\[\Phi(X)\simeq(C_\Phi\wedge X^{\wedge J})_{hAut_G(J)}\]
%where $Aut_G(J)$ is the group of equivariant automorphisms of $J$, and $C_\Phi$ is a $J$-spectrum with $Aut_G(J)$-action. This equivalence cannot hold unless $J$ has trivial action. This is because the right-hand functor is only $(G\times J)$-excisive, and $\Phi$ is $J$-excisive (non-equivariantly, the left-hand functor is $J/G$-excisive, but the right hand functor is only $(G\times J)/G=|J|$-excisive). One can still wonder if the above equivalence holds for $(J\times G)$-excisive functors. However, this decomposition would essentially not give us anything new that the classification theorem \ref{classification}. This is because Example \ref{coeffindexedsmash} would give a further decomposition
%\[\Phi(X)\simeq(C_\Phi\wedge X^{\wedge J})_{hAut_G(J)}\simeq\big((\bigvee_{\Sigma_{|J|}}C_\Phi\wedge X^{|J|})_{h\Sigma_{|J|}}\big)_{hAut_G(J)}\]
%where the $Aut_G(J)$-action on $(\bigvee_{\Sigma_{|J|}}C_\Phi)\wedge X^{|J|}$ permutes the wedge summands by acting on ${\Sigma_{|J|}}$ by right multiplication. Since this action commutes with the $\Sigma_{|J|}$-action, we get an equivalence
%\[\Phi(X)\simeq\big((\Sigma_{|J|}/{{Aut_G(J)}}_+)\wedge C_\Phi\wedge X^{|J|}\big)_{h\Sigma_{|J|}}\]
%which amounts to a decomposition of the form of Theorem \ref{classification}.
%\end{rem}

\begin{proof}[Proof of \ref{classification}]
Let $\Phi\colon \Top_\ast\to \Top_{\ast}^G$ be a finitary $sSet$-enriched strongly $J$-homogeneous homotopy functor, and let $\widehat{\Phi}\colon \Top_\ast\to \Sp_{J}^{G}$ be the lift constructed in \ref{delooping}. By \ref{classspectra} there is a natural equivalence
\[\Phi(X)\simeq\Omega^{\infty J}\widehat{\Phi}(X)\simeq\Omega^{\infty J}((cr_n\widehat{\Phi})(\Delta S^0)\wedge X^{\wedge n})_{h\Sigma_{n}}\]
for every $X\in \Top_\ast\langle J\rangle$. Here $\Delta\colon \Top_\ast\to \Top^{\times n}_\ast$ is the diagonal functor.
We need to show that $(cr_{n}\widehat{\Phi})(\Delta S^0)$ is equivalent to the $J$-spectrum with $\Sigma_{n}$-action $E_{\Phi}$ defined in level $k$ by
\[(E_{\Phi})_k=\Omega^{k\overline{V}_{n}}cr_{n}\Phi(\Delta S^k)\]
where $\overline{V}_{n}$ is the reduced standard representation of $\Sigma_{n}$. These two spectra are equal in spectrum level $0$, and since $\widehat{\Phi}$ has fibrant values the spectrum $(cr_{n}\widehat{\Phi})(\Delta S^0)$ is fibrant in $\Sp_{J}^{G}$. Hence it is sufficient to show that $E_{\Phi}$ is fibrant. Its value at the $G$-set $kJ$ is
\[(E_{\Phi})_{kJ}=\Omega^{kJ\overline{V}_{n}}
cr_{n}\Phi(\Delta S^{kJ})\]
where $kJ\overline{V}_{n}=\mathbb{R}[kJ]\otimes \overline{V}_{n}$. We need to show that the map $\sigma_k\colon (E_{\Phi})_{kJ}\to \Omega^J(E_{\Phi})_{(k+1)J}$ is an equivalence of $G$-spaces, for every integer $k\geq 0$. The isomorphism of $G$-representations
\[\mathbb{R}[J]\oplus (\mathbb{R}[(k+1)J]\otimes\overline{V}_{n})\cong \mathbb{R}[nJ]\oplus (\mathbb{R}[kJ]\otimes\overline{V}_{n})\]
defines a commutative triangle
\[\xymatrix{\Omega^{kJ\overline{V}_{n}}cr_{n}\Phi(\Delta S^{kJ})\ar[dr]\ar[r]^-{\sigma_k}&\Omega^J\Omega^{(k+1)J\overline{V}_{n}}
cr_{n}\Phi(\Delta S^{(k+1)J})\ar[d]^{\cong}\\
&\Omega^{kJ\overline{V}_{n}}
\Omega^{nJ}cr_{n}\Phi(\Delta S^{(k+1)J})}.\]
Since the cross-effect functor is $J$-multilinear it is in particular $J$-linear in each variable, and the diagonal map is an equivalence, proving that $E_{\Phi}$ is fibrant.

It remains to show that every $E\in \Sigma_n$-$\Sp_{J}^G$ is equivalent to
\[E\simeq cr_n ((E\wedge (-)^{\wedge n})_{h\Sigma_n})(\Delta S^0).\]
This follows from the first part of the proof of \cite[3.5]{calcIII}, using that $\Sp_{J}^G$ is stable.
\end{proof}

\section{Generalized Tom Dieck-splittings}\label{tomdieckspl}

Let $J$ be a finite $G$-set, and let $\Phi\colon \Top_\ast\to \Top_{\ast}^G$ be a homotopy functor which commutes with fixed-points, in the sense of Definition \ref{def:split}. We show that the fixed points of the differential $D_J\Phi$ decompose as a product, whose factors depend on the differentials $D_{J/H}\Phi$ for normal subgroups $H$ of $G$. Here $J/H$ is always seen as a $G$-set. This decomposition is intimately related to the Tom Dieck-splitting for the identity functor, when $J=G$ (see \ref{transTD} and \ref{nonAb}).

If $H\lhd G$ is a normal subgroup, there is a $G$-equivariant projection map $p_{J,H}\colon J\to J/H$ which is a bijection on $G$-orbits. Let $p_{J,H}^\ast\colon P_{J/H}\Phi\to P_{J}\Phi$ be the induced map on excisive approximations constructed in \S\ref{sec:tree}. 
This map has a retraction on $G$-fixed-points which is constructed in a way similar to the one in the proof of \ref{convergence}, as follows. There is a $G$-equivariant map
\[(T_{J}\Phi(c))^H\stackrel{res}{\longrightarrow}\holim_{\mathcal{P}_0(J/H_+)}\Phi(\Lambda^J(c))^H
\overset{(\eta_p)^{H}_\ast}{\underset{\cong}\longleftarrow}\holim_{\mathcal{P}_0(J/H_+)}\Phi(\Lambda^{J/H}(c))^H = (T_{J/H}\Phi(c))^H.\]
The first map is the fixed-points restriction (\ref{fixedptsdiag}) from \S\ref{sec:fixedpts}. The second map is induced by the natural transformation $\eta_{p_{J,H}}\colon \Lambda_{U/H}^{J/H}(c)\to \Lambda_{U}^J(c)$ of \S\ref{sec:tree}, and it is an isomorphism because $\Phi$ commutes with fixed-points.
By iterating $T_J$ and taking $G$-fixed points this map gives a retraction \[r_{J,H}\colon (P_{J}\Phi)^G\to (P_{J/H}\Phi)^G\]
for the map $p_{J,H}^\ast\colon (P_{J/H}\Phi)^G\to (P_{J}\Phi)^G$. These maps satisfy the following compatibility conditions. If $H$ is a subgroup of a normal subgroup $L\lhd G$, the diagrams
\[\xymatrix{
(P_J\Phi)^G\ar[r]^{r_{J,H}}\ar[dr]_-{r_{J,L}}&(P_{J/H}\Phi)^G\ar[d]^{r_{J/H,L}}\\
&(P_{J/L}\Phi)^G}
\ \ \ \ \ \ \ \ \ \  \mbox{and}\ \ \ \ \ \ \ \ \ \
\xymatrix{
(P_{J/H}\Phi)^G\ar[d]_{r_{J/H,L}}\ar[r]^{p^{\ast}_{J,H}}&(P_J\Phi)^G\\
(P_{J/L}\Phi)^G\ar[ur]_{p^{\ast}_{J,L}}
}
\]
commute. A construction similar to the one at the end of \S\ref{layers} lifts these maps to the level of differentials $r_{J,H}\colon (D_{J}\Phi)^G\to (D_{J/H}\Phi)^G$.

\begin{theorem}\label{higertomdieck}
Let $\Phi\colon \Top_{\ast}\to \Top_{\ast}^{G}$ be a homotopy functor which commutes with fixed-points. For every finite $G$-set $J$ there is a natural equivalence of functors $\Top_{\ast}^G\to \Top_{\ast}$
\[(D_{J}\Phi)^G\simeq \prod_{H\lhd G}\hofib\big((D_{J/H}\Phi)^G\to \holim_{H<L\lhd G}(D_{J/L}\Phi)^G\big)\]
where the product runs over the normal subgroups of $G$, and the limit is taken over the retractions $r_{J/L,L'}\colon (D_{J/L}\Phi)^G\to (D_{J/L'}\Phi)^G$. In particular $(D_{J}\Phi(X))^G$ always splits off $(D_{J/G}\Phi(X))^G\cong D_{J/G}\Phi(X^G)^G$ as a factor.
\end{theorem}

Before proving this result we describe the summands of this splitting for the identity functor in more familiar terms, recovering the Tom Dieck-splitting. For any subgroup $H$ of $G$, we let $F^{G}_{G/H}\colon \Sp^{G}\to \Sp$ be the $G$-fixed points functor for the universe defined by the permutation representation of $G/H$. It sends a $G$-spectrum $E$ to the spectrum with level $k$ space
\[F^{G}_{G/H}(E)_k=(\Omega^{k\overline{G/H}}E_{kG/H})^G.\]

\begin{cor}\label{transTD}
Let $\Phi\colon \Top_{\ast}\to \Top_{\ast}^G$ be an enriched reduced homotopy functor which commutes with fixed-points, and let $\Phi(\mathbb{S}\wedge X)$ be the $G$-spectrum defined by the spaces $\{\Phi(S^n\wedge X)\}_{n\geq 0}$. The $G$-fixed points spectrum of $\Phi(\mathbb{S}\wedge X)$ splits naturally as
\[\Phi(\mathbb{S}\wedge X)^G\simeq \bigvee_{H\lhd G}\hofib\Big(F_{G/H}^G\big(\Phi(\mathbb{S}\wedge X)\big)\to \holim_{H< K\lhd G}F_{G/K}^G\big(\Phi(\mathbb{S}\wedge X)\big)\Big)\]
for every pointed $G$-space $X$.
When $\Phi\colon \Top_\ast\to\Top^{G}_\ast$ is the inclusion and $G$ is Abelian, this is the Tom Dieck-splitting 
\[(\Sigma^{\infty} X)^G\simeq \bigvee_{H\leq G}\Sigma^{\infty} X^{H}_{hG/H}\]
of \cite{dieck}, for every pointed $G$-space $X$.
\end{cor}

\begin{rem}
If $\Phi$ is $G$-excisive, the $G$-spectra $\Phi(\mathbb{S}\wedge X)$ are fibrant with respect to the complete $G$-universe, and the fixed point spectra $F_{G/H}^G\big(\Phi(\mathbb{S}\wedge X)\big)$ are equivalent for every $H$. The splitting of \ref{transTD} is in this case trivial.
\end{rem}

\begin{proof}[Proof of \ref{transTD}]
For every transitive $G$-set $T$, Example \ref{ex:transapprox} gives an equivalence
\[(D_T\Phi(X))^G\simeq (\Omega^{\infty T}\Phi(\Sigma^{\infty T}X))^G\cong \Omega^\infty F_{T}^G\Phi(\mathbb{S}\wedge X).\]
The splitting follows from Theorem \ref{higertomdieck}.

Now suppose that $\Phi\colon \Top_\ast\to\Top^{G}_\ast$ is the inclusion functor. We need to identify
\[\hofib\big(F_{G/H}^G(\Sigma^\infty X)\to\holim_{H<K}F_{G/K}^G(\Sigma^\infty X)\big)\]
with the suspension spectrum of $X^{H}_{hG/H}$, for every subgroup $H$ of $G$. 
The various fixed points spectra satisfy the identities
\[F_{G/H}^GE= F_{G/H}^{G/H}F_{G/G}^HE \ \ \ \ \ \ \ \ \ \ \ F_{G}^GE\stackrel{\simeq}{\longrightarrow} F^{G/H}_{G/H}F_{G}^HE\]
for every $G$-spectrum $E$. In particular
\[F_{G/H}^G(\Sigma^\infty X)= F_{G/H}^{G/H}F_{G/G}^H(\Sigma^\infty X)=F_{G/H}^{G/H}(\Sigma^\infty X^H).\]
Let $E\mathcal{F}_H$ be a classifying space for the family of subgroups of $G$ which do not contain $H$, and let $\widetilde{E\mathcal{F}}_H$ be the cofiber of $(E\mathcal{F}_H)_+\to S^0$. The $H$-geometric fixed points of $\Sigma^\infty X$ are 
\[F_{G}^H(\widetilde{E\mathcal{F}}_H\wedge \Sigma^\infty X)\simeq \Sigma^{\infty}X^H\] (see \cite[10.2]{Lewis}). Combining this equivalence with the identities above we obtain a natural equivalence
\begin{equation}\label{one}
F_{G/H}^G(\Sigma^\infty X)\simeq F_{G/H}^{G/H}(\Sigma^\infty X^H)\simeq 
F_{G/H}^{G/H}F_{G}^H(\widetilde{E\mathcal{F}}_H\wedge \Sigma^\infty X)\simeq F_{G}^{G}(\widetilde{E\mathcal{F}}_H\wedge \Sigma^\infty X)
\end{equation}
for every subgroup $H$ of $G$. It remains to identify
\[\hofib\big(F_{G}^{G}(\widetilde{E\mathcal{F}}_H\wedge \Sigma^\infty X)\to\holim_{H<K}F_{G}^{G}(\widetilde{E\mathcal{F}}_K\wedge \Sigma^\infty X)\big)\]
with $\Sigma^\infty X^{H}_{hG/H}$. By commuting the constructions appropriately we rewrite this homotopy fiber as
\[F_{G}^{G}\big(\Sigma^\infty X\wedge \holim_{H<K}\hofib(\widetilde{E\mathcal{F}}_H\to \widetilde{E\mathcal{F}}_K)
\big).\]
For every subgroup $L$ of $G$, the $L$-fixed-points of $\hofib(\widetilde{E\mathcal{F}}_H\to \widetilde{E\mathcal{F}}_K)$ are contractible if $L$ is a subgroup of $K$ or if $L$ is not a subgroup of $H$, and they are equivalent to $S^0$ otherwise. The maps induced by inclusions $K\to K'$ of subgroups containing $H$ are either identities or basepoint inclusions. It follows that the fixed-points of the homotopy limit over $K$ are contractible unless all the spaces in the diagram are $S^0$, that is
\[\holim_{H<K}\hofib(\widetilde{E\mathcal{F}}_H\to \widetilde{E\mathcal{F}}_K)^L\simeq\left\{\begin{array}{ll}
S^0&,\ \substack{\mbox{if $H\leq L$ and $K$ is not a}\\\mbox{subgroup of $L$, for all $H<K$}}\\
\\
\ast&, \ \mbox{otherwise}
\end{array}\right.\simeq \left\{\begin{array}{ll}
S^0&,\ \mbox{if $H= L$}
\\
\ast&, \ \mbox{otherwise}
\end{array}\right.
\]
(for the last identification we are using that every subgroup of $G$ is normal, see \ref{nonAb} below).
Another space which is characterized by this property is $\widetilde{E\mathcal{F}}_H\wedge EG/H_+$. It follows that the $H$-summand of the splitting is equivalent to
\[F_{G}^{G}(\Sigma^\infty X\wedge \widetilde{E\mathcal{F}}_H\wedge EG/H_+)\simeq_{(\ref{one})}F_{G/H}^{G/H}(\Sigma^\infty X^H\wedge EG/H_+)\simeq \Sigma^\infty X^{H}_{hG/H}\]
where the last equivalence is the Adams isomorphism.
\end{proof}

\begin{rem}\label{nonAb} If $G$ has a subgroup which is not normal, the decomposition of \ref{transTD} is still part of the Tom Dieck-splitting, in the sense that each $H$-summand of \ref{transTD} splits further. In trying to identify the summands of the decomposition, the argument of the proof of \ref{transTD} applies until the equivalence
\[\holim_{H<K\lhd G}\hofib(\widetilde{E\mathcal{F}}_H\to \widetilde{E\mathcal{F}}_K)^L\simeq\left\{\begin{array}{ll}
S^0&,\ \substack{\mbox{if $H\leq L$ is the largest normal}\\
\mbox{subgroup of $G$ contained in $L$}}\\
\\
\ast&, \ \mbox{otherwise}
\end{array}\right.\]
where now $K$ only runs through the normal subgroups of $G$ which contain $H$ as a proper subgroup. Smashing this space with $\Sigma^\infty X$ and taking fixed-points, we obtain a decomposition of the $H$-summand of \ref{transTD} into $\bigvee_{(S)}\Sigma^{\infty}X^{S}_{h WS}$. Here $(S)$ runs through the conjugacy classes of subgroups of $G$ such that $H\leq S$ is the largest normal subgroup of $G$ contained in $S$. Since every conjugacy class has a unique maximal subgroup (which is not necessarily proper) the set of pairs $(H\lhd G,(S))$ as above is in bijection with the set of conjugacy classes of subgroups of $G$, and this is the full Tom Dieck-splitting.
\end{rem}

The proof of \ref{higertomdieck} is based on the following lemma.

\begin{lemma}\label{decomp}
Let $I$ be a finite poset with an initial object $\emptyset$, and let $P\colon I\to D$ be a functor to a stable model category $D$. Suppose that every map in the diagram $P$ has a section, and that these sections assemble into a diagram $S\colon I^{op}\to D$ such that for every $i<j<k$ in $I$ the diagram
\[\xymatrix{
P_j\ar[d]_{P(j<k)}\ar[r]^{S(i<j)}&P_i\\
P_k\ar[ur]_{S(i<k)}
}\]
commutes. Then the initial vertex $P_\emptyset$ decomposes as
\[P_\emptyset\simeq\bigvee_{i\in I}\hofib\big(P_i\to \holim_{j> i}P_j \big).\]
\end{lemma}

\begin{proof}
Let the degree of $i$ be the length of the longest sequence of non-identity morphisms starting at $i$. We prove by induction on the degree that
\[P_i\simeq\bigvee_{j\geq i}\hofib\big(P_j\to \holim_{k> j}P_k \big).\]
Let us write $M_j=\hofib\big(P_j\to \holim_{k> j}P_k \big)$. If $i$ has degree zero there are no non-identity maps out of $i$, and the decomposition holds:
\[\bigvee_{j\geq i}M_j=M_i=\hofib\big(P_i\to \holim(\emptyset)\big)=P_i.\]
Now suppose that we proved the decomposition for vertices of degree smaller or equal to $d$, and let $i$ have degree $d+1$. The section $S\colon I^{op}\to D$ induces a section of the canonical map
\[
\xymatrix{P_i\ar[r]^-{\simeq}&\holim_{j>i}P_i\ar[r]& \holim_{j>i}P_j
\ar@/_1.2pc/[l]_S
},\]
which gives a decomposition $P_i\simeq M_i\vee \holim_{j>i}P_j$.
 We need to identify $\holim_{j>i}P_j$ with $\bigvee_{j>i}M_j$. By the inductive hypothesis the homotopy limit decomposes as
\[\holim_{j>i}P_j\simeq \holim_{j>i}\bigvee_{k\geq j}M_k.\]
For vertices $j>i$ with $j$ of degree $d$, the $k=j$ summand of the wedge decomposition $\bigvee_{k\geq j}M_k$ does not receive any map, and it is collapsed by the maps out of $j$. Therefore these summands commute with the homotopy limits
\[\holim_{j>i}P_j\simeq \holim_{j>i}\bigvee_{k\geq j}M_k\simeq \big(\bigvee_{\substack{
j>i\\
\deg(j)=d
}}M_j\big)\vee
\holim_{j>i}\left\{\begin{array}{ll}\bigvee_{k>j}M_k& \ , \ \deg(j)=d\\
\bigvee_{k\geq j}M_k& \ , \ \deg(j)\leq d-1
\end{array}\right.\]
We are left with determining the right-hand homotopy limit.
Each summand of $\bigvee_{k\geq j'}M_k$ for $j'$ of degree smaller than $d-1$ is hit by a summand of $\bigvee_{k\geq j}M_k$ with $j$ of degree $d$. Moreover any common summand of $\bigvee_{k\geq j}M_k$ and $\bigvee_{k\geq l}M_k$ for $j$ and $l$ of degree $d$ maps to $\bigvee_{k\geq j'}M_k$ for some $j'$ of degree smaller than $d-1$. It follows that the limit is determined by the decompositions of vertices of degree $d$, without repetitions:
\[\holim_{j>i}\left\{\begin{array}{ll}\bigvee_{k>j}M_k& \ , \ \deg(j)=d\\
\bigvee_{k\geq j}M_k& \ , \ \deg(j)\leq d-1
\end{array}\right.\simeq \bigvee_{\substack{
j>i\\
\deg(j)=d
}}\bigvee_{k>j}M_k\big).\]
Combining these splittings together we obtain the decomposition
\[P_i\simeq M_i\vee \big(\bigvee_{\substack{
j>i\\
\deg(j)=d
}}M_j\big)\vee 
\big(\bigvee_{\substack{
j>i\\
\deg(j)=d
}}\bigvee_{k>j}M_k\big)=
\bigvee_{j\geq i}M_j.\]
\qedhere
\end{proof}

\begin{ex}\label{classicalTD}
Given two conjugacy classes $(K)$ and $(H)$ of subgroups of $G$, we write $(K)\leq (H)$ if $K$ is a subconjugate of $H$. Let $\mathcal{C}_{(H)}$ be the family of subgroups of $G$ subconjugate to $H$, and let $E\mathcal{C}_{(H)}$ be a functorial classifying space for this family. Given a pointed $G$-space $X$, the assignment $(H)\mapsto (\Sigma^{\infty}X\wedge( E\mathcal{C}_{(H)})_+)^G$ defines a functor from the poset of conjugacy classes of subgroups of $G$ to spectra. For $(H)\leq (K)$, the map $(\Sigma^{\infty}X\wedge( E\mathcal{C}_{(H)})_+)^G\to (\Sigma^{\infty}X\wedge( E\mathcal{C}_{(K)})_+)^G$ is a split monomorphism (\cite[2.1]{Lewis}). It follows by \ref{decomp} that the value at $(G)$ decomposes as
\[(\Sigma^{\infty}X\wedge( E\mathcal{C}_{(G)})_+)^G=(\Sigma^{\infty}X)^G\simeq 
\bigvee_{(H)}\hofib\big((\Sigma^{\infty}X\wedge( E\mathcal{C}_{(H)})_+)^G\!\to\! \holim_{(K)< (H)}(\Sigma^{\infty}X\wedge( E\mathcal{C}_{(K)})_+)^G \big).\]
The homotopy fiber in the $H$-summand is equivalent to $(\Sigma^{\infty}X\wedge(E_{(H)}))^G$, where the $G$-space 
\[E_{(H)}=\holim_{(K)<(H)}\hofib((E\mathcal{F}_{(H)})_+\to (E\mathcal{F}_{(K)})_+)\] has $L$-fixed points equivalent to $S^0$ if $(L)=(H)$, and contractible otherwise.
%The $G$-space $E_{(H)}$ is therefore equivalent to $\widetilde{E\mathcal{P}}_H\wedge EG/H_+$, and we get equivalences
%\[(\Sigma^{\infty}X\wedge E_{(H)})^G\simeq (\Sigma^{\infty}X\wedge \widetilde{E\mathcal{P}}_H\wedge EG/H_+)^G\simeq F_{G/H}^{G/H}F_{G}^H(\Sigma^{\infty}X\wedge \widetilde{E\mathcal{P}}_H\wedge EG/H_+)\simeq F_{G/H}^{G/H}(\Sigma^{\infty}X^H\wedge EG/H_+)\simeq X^{H}_{hG/H}\]
%identifying the decomposition above with the Tom Dieck-splitting.
The fixed-points spectrum $(\Sigma^{\infty}X\wedge E_{(H)})^G$ is then equivalent to $\Sigma^{\infty}X^{H}_{hWH}$, and the decomposition of Lemma \ref{decomp} is the Tom Dieck-splitting.
\end{ex}

The challenge in achieving the splitting of Theorem \ref{higertomdieck} is to lift the maps $p_{J,H}^\ast\colon D_{J/H}\Phi\to D_{J}\Phi$ and their retractions to a suitable category of spectra, in order to reduce the problem to the decomposition criterion of \ref{decomp}.
Let $\Sp_{J}^G$ be the category of $J$-spectra introduced in \S\ref{delooping}. By the delooping theorem \ref{Jdelooping} the respectively $J$-homogeneous and $J/H$-homogeneous functors $D_J\Phi$ and $D_{J/H}\Phi$ deloop equivariantly as
\[D_J\Phi\simeq\Omega^{\infty J}\widehat{D}_J\Phi \ \ \ \ \ \ \ \ \ \ \ \ \ \  D_{J/H}\Phi\simeq \Omega^{\infty J/H}\widehat{D}_{J/H}\Phi \]
for functors $\widehat{D}_J\Phi\colon \Top_\ast\to \Sp_{J}^G$ and $\widehat{D}_{J/H}\Phi\colon \Top_\ast\to \Sp_{J/H}^G$. We need to understand the functoriality of these extensions with respect to the maps $p_{J,H}^\ast$ and $r_{J,H}$. The category $\Sp_{J}^G$ has fixed points functors $(-)^H\colon \Sp_{J}^G\to\Sp$ for every subgroup $H$ of $G$, defined by
\[(E^H)_k=(\Omega^{k\overline{J}}E_{kJ})^H.\]
The category $\Sp_{J/H}^G$ has similar fixed point functors.

\begin{prop}\label{liftonfixedpts}
Let $\Phi\colon \Top_\ast\to \Top_{\ast}^G$ be a homotopy functor which commutes with fixed-points, and let $H$ be a normal subgroup of $G$. The map $p\colon J\to J/H$ induces a split monomorphism
\[\xymatrix@C=40pt{\big(\widehat{D}_{J/H}\Phi\big)^G\ar[r]_{p^\ast}&
\big(\widehat{D}_J\Phi\big)^G\ar@/_1.5pc/[l]_{r}
}\]
of functors $\Top_{\ast}^G\to \Sp$.
\end{prop}

%\begin{ex}
%Consider the group $G=\mathbb{Z}/2$ and the identity functor $I$ on $\Top_{\ast}^{\mathbb{Z}/2}$. The functor $D_1I\simeq Q(-)$ lifts to the suspension spectrum functor $\mathbb{S}\wedge (-)\colon \Top_{\ast}^{\mathbb{Z}/2}\to \Sp^{\mathbb{Z}/2}_1$ to na\"{i}ve $\mathbb{Z}/2$-spectra, and the functor $D_{\mathbb{Z}/2}I\simeq Q^{\mathbb{Z}/2}(-)$ lifts to the suspension spectrum functor $\mathbb{S}\wedge (-)\colon \Top_{\ast}^{\mathbb{Z}/2}\to \Sp^{\mathbb{Z}/2}_{\mathbb{Z}/2}$ to genuine $\mathbb{Z}/2$-spectra. The fixed points functor in the category $\Sp^{\mathbb{Z}/2}_1$ is just the categorical fixed points, which is taken levelwise.
%The splitting of Theorem \ref{liftonfixedpts} induced by the projection $\mathbb{Z}/2\to 1$ is
%\[\xymatrix@C=40pt{\mathbb{S}\wedge (X^{\mathbb{Z}/2})\ar[r]_{\alpha^\ast}&
%(\mathbb{S}\wedge X)^{\mathbb{Z}/2}\ar@/_1.5pc/[l]_{\alpha_!}
%}\]
%An identification of the homotopy fiber of $\alpha_!$ as $\mathbb{S}\wedge (X_{h\mathbb{Z}/2})$ gives the Tom-Dieck splitting. Notice however that the delooping of \ref{delooping} for $D_1$ and $D_{\mathbb{Z}/2}$ is not the suspension spectrum functor on the nose, but it is a fibrant replacement.
%\end{ex}

\begin{proof}
For every $G$-orbit $o$ in $J/G$ and every homotopy functor $\Psi\colon \Top_\ast\to \Top_{\ast}^G$, let 
\[R_o\Psi=\hocolim_k\holim_{\mathcal{A}_{o_+}^k}\Psi(\Lambda^{J,k})\]
be the construction of \S\ref{delooping}. We recall that when $\Psi$ is $J$-homogeneous this defines an $o$-delooping
\[\Omega^oR_o\Psi\stackrel{\simeq}{\longrightarrow}\hocolim_k\holim_{W\in\mathcal{P}_0(o_+)}\holim_{\mathcal{A}_{o_+}^k}
\Psi(\Lambda^{J,k})\stackrel{\simeq}{\longleftarrow}P_J\Psi(\Lambda^{J,k})\stackrel{\simeq}{\longleftarrow}\Psi\]
of $\Psi$. Let us also recall that the value of the $J$-spectrum $\widehat{\Psi}$ at the $G$-set $J$ was defined formally as a certain homotopy limit with the purpose of fixing the direction of the wrong-way pointing equivalence in the delooping above. To this end, let us denote
\[\widetilde{\Omega}^JR_J\Psi:=\hocolim_{k}\holim_{W_1\in\mathcal{P}_0((o_1)_+)}
\holim_{\mathcal{A}_{W_1}^{k}}\dots \holim_{W_n\in\mathcal{P}_0((o_n)_+)}
\holim_{\mathcal{A}_{W_n}^{k}}\Psi(\Lambda^{J,nk})\]
for a chosen order on the orbits set. Next, we need to describe how to extract the reduced loop space $\Omega^{\overline{J}}R_J\Psi$ from the new loop space model $\widetilde{\Omega}^JR_J\Psi$. Given any finite $G$-set $I$ there is a left-cofinal functor
\[F\colon \mathcal{P}_0(1_+)\times \mathcal{P}_0(I)\longrightarrow \mathcal{P}_0(I_+)\]
that sends $(1,U)$ to $U$, $(+,U)$ to $+$ and $(1_+,U)$ to $U_+$. If $X$ is a pointed $G$-space, the homotopy limit of the $I_+$-cube $\omega^IX$ with terminal vertex $X$ and points everywhere else is homeomorphic to $\Omega^IX$, and there is a commutative diagram
\[\xymatrix{\holim_{\mathcal{P}_0(I_+)}\omega^IX\ar[d]_{F^\ast}^{\simeq}\ar[r]^{\cong}
&\Omega^I X\ar[d]^{\cong}\\
\holim_{\mathcal{P}_0(1_+)}\holim_{\mathcal{P}_0(I)}F^\ast\omega^JX\ar[r]^-{\cong}&
\Omega\Omega^{\overline{I}}X
}.\]
This shows that the reduced loop space $\Omega^{\overline{J}}R_J\Psi$ can be obtained by restricting
$\widetilde{\Omega}^JR_J\Psi$ along the inclusions $\mathcal{P}(o)\to \mathcal{P}(o_+)$ which send $U$ to $U_+$. That is, we define
\[\widetilde{\Omega}^{\overline{J}}R_J\Psi:=\hocolim_{k}\holim_{W_1\in\mathcal{P}_0(o_1)}
\holim_{\mathcal{A}_{(W_1)_+}^{k}}\dots \holim_{W_n\in\mathcal{P}_0(o_n)}
\holim_{\mathcal{A}_{(W_n)_+}^{k}}\Psi(\Lambda^{J,nk}).\]
By naturality of the lifts $\widehat{D}_{J}\Phi$ and $\widehat{D}_{J/H}\Phi$ it is going to be sufficient to define a splitting 
\[\xymatrix@C=40pt{(\widetilde{\Omega}^{\overline{J/H}}R_{J/H}D_{J/H}\Phi)^G\ar[r]_-{p^\ast}&
(\widetilde{\Omega}^{\overline{J}}R_{J}D_{J}\Phi)^G\ar@/_1.5pc/[l]_{r}
}.\]
The map $p^\ast$ is the restriction on $G$-fixed-points of a $G$-equivariant map
which is defined in the following manner, similar to the definition of the map $p^\ast\colon P_{J/H}\Phi\to P_J\Phi$. Observe that the image functor $p\colon \mathcal{P}_0(J_+)^{k}\to \mathcal{P}_0(J/H_+)^{k}$ restricts to a functor $p\colon \mathcal{A}_{W}^{k}\to \mathcal{A}_{p(W)}^{k}$ for every $W\in \mathcal{P}_0(o_+)$ and every orbit $o$ of $J$. The monomorphism of the splitting is the map
\[\xymatrix{
\widetilde{\Omega}^{\overline{J/H}}R_{J/H}D_{J/H}\Phi\ar[r]\ar@{-->}[d]_{p^\ast}
&
\displaystyle\hocolim_{k}\holim_{W_1\in\mathcal{P}_0(o_1)}
\holim_{\mathcal{A}_{(W_1)_+}^{k}}\dots \holim_{W_n\in\mathcal{P}_0(o_n)}
\holim_{\mathcal{A}_{(W_n)_+}^{k}}D_{J/H}\Phi(p^{\ast}\Lambda^{J/H,nk})\ar[d]^{\eta}
\\
\widetilde{\Omega}^{\overline{J}}R_{J}D_{J}\Phi
&
\widetilde{\Omega}^{\overline{J}}R_{J}D_{J/H}\Phi
\ar[l]
}.\]
The top horizontal map is the restriction map on homotopy limits induced by the image functors $p\colon \mathcal{A}_{W}^{k}\to \mathcal{A}_{p(W)}^{k}$ and $p\colon \mathcal{P}_0(o)\to \mathcal{P}_0(p(o))$. The map $\eta$ is induced by the natural transformation of diagrams $\eta_p\colon p^{\ast}\Lambda^{J/H}\to \Lambda^J$ from \S\ref{sec:tree}. The bottom horizontal map is induced by the map $D_{J/H}\Phi\to D_J\Phi$ from \S\ref{layers}.

%Let $o_1,\dots, o_n$ be an ordering of $J/G$.
%Notice that the projection $p\colon J\to J/H$ is a bijection on $G$-orbits, and therefore the $G$-orbits of $J/H$ inherit an order $p(o_1),\dots,p(o_n)$.
%The reduced loop spaces $\Omega^{\overline{J}}$ and $\Omega^{\overline{J/H}}$ decompose respectively as
%\[\Omega^{\overline{J}}=\Omega^{n-1}\Omega^{\overline{o_1}}\dots
%\Omega^{\overline{o_n}}
%\ \ \ \ \ \ \ \ \ \mbox{and}\ \ \ \ \ \ \ \ \ \Omega^{\overline{J/H}}=\Omega^{n-1}\Omega^{\overline{p(o_1)}}\dots
%\Omega^{\overline{p(o_n)}}\]
%The iteration of the maps $p_{o}^\ast$ above followed by the map $p^\ast\colon D_{J/H}\Phi\to D_{J}\Phi$ of \ref{layers} define a $G$-map
%\[p^{\ast}\colon \Omega^{\overline{J/H}}R_{p(o_1)}\dots R_{p(o_n)}D_{J/H}\Phi\longrightarrow \Omega^{\overline{J}}R_{o_1}\dots R_{o_n}D_{J/H}\Phi\longrightarrow \Omega^{\overline{J}}R_{o_1}\dots R_{o_n}D_{J}\Phi\]
%

Let us define the retraction $r$. 
Observe that the isomorphism $\mathcal{P}(J_+)^H\cong \mathcal{P}(J/H_+)$ induces an isomorphism $(\mathcal{A}_{W}^{k})^H\cong \mathcal{A}_{p(W)_+}^{k}$ for every $W\in\mathcal{P}_0(o_+)$ and every orbit $o\in J/G$. Hence the fixed-points restriction map (\ref{fixedptsdiag}) from \S\ref{sec:fixedpts} defines a $G$-map
\[(\widetilde{\Omega}^{\overline{J}}R_{J}D_{J}\Phi)^H\longrightarrow
\displaystyle\hocolim_{k}\holim_{V_1\in\mathcal{P}_0(p(o_1))}
\holim_{\mathcal{A}_{(V_1)_+}^{k}}\dots \holim_{V_n\in\mathcal{P}_0(p(o_n))}
\holim_{\mathcal{A}_{(V_n)_+}^{k}}\big(D_{J}\Phi(\Lambda^{J,nk})\big)^H.\]
The retraction $r$ is defined by composing this map with the map induced by
\[\big(D_{J}\Phi(\Lambda_{U}^{J,nk})\big)^H\longrightarrow \big(D_{J/H}\Phi(\Lambda_{U}^{J,nk})\big)^H\cong\big(D_{J/H}\Phi(\Lambda_{U/H}^{J/H,nk})\big)^H\]
for $U\in\mathcal{P}_0(J)^H$, and restricting on $G$-fixed-points. In this last composition the first map is $(D_{J}\Phi)^H\to (D_{J/H}\Phi)^H$ defined at the beginning of the section, and the second map is an isomorphism because $D_{J/H}\Phi$ commutes with $H$-fixed-points.
\end{proof}

\begin{proof}[Proof of \ref{higertomdieck}]
The splittings of Theorem \ref{liftonfixedpts} satisfy the conditions of Lemma \ref{decomp}, by the commutativity of the diagrams before \ref{higertomdieck}.
Thus Lemma \ref{decomp} gives an equivalence of spectra
\[(\widehat{D}_{J}\Phi)^G\simeq
\bigvee_{H\lhd G}\hofib
\big((\widehat{D}_{J/H}\Phi)^G\to \holim_{H<L\lhd G}(\widehat{D}_{J/L}\Phi)^G\big).\]
Its $J$-infinite loop space gives the desired decomposition of $(D_J\Phi)^G$ by the delooping Theorem \ref{delooping}.
\end{proof}

\section{The equivariant derivatives of equivariant homotopy theory}\label{sec:Id}

Let $G$ be a finite group and let $I\colon \Top_\ast\to \Top_{\ast}^G$ be the inclusion of pointed spaces with the trivial $G$-action, which extends to the identity functor on pointed $G$-spaces. We proved in \ref{convergence} that for every pointed $G$-space $X$ the canonical map
\[\holim_{n}P_nI(X)\stackrel{\simeq}{\longrightarrow} \holim_nP_{nG}I(X)\]
is a $G$-equivalence. In particular if all the fixed-points of $X$ are nilpotent, the left-hand limit is $G$-equivalent to $X$. Thus the collection of $nG$-excisive approximations provides a tower of weaker and weaker ``genuine homology theories'' that converges to $X$. The aim of this section is to understand its layers $D_{nG}I=\hofib\big(P_{nG}I\to P_{(n-1)G}I\big)$.

Let $T_{k}$ be the $k$-partition complex of \cite{AM}, with its usual $\Sigma_k$-action. We recall that the $p$-simplices of $T_{k}$ are defined to be a point if $p=0$, two points $\ast_+$ if $p=1$, and for larger $p$ they are the $(p-2)$-simplices of the nerve of the poset of unordered partitions of $\underline{k}=\{1,\dots,k\}$, ordered by refinement, with a disjoint base-point. We consider $T_{k}$ as a $G\times\Sigma_k$-space on which $G$ acts trivially.

We recall that $\mathcal{F}_k$ is the family of graphs of group homomorphisms $\rho\colon H\to \Sigma_k$, for subgroups $H\leq G$. Following the notation of \ref{DnGsym} we let $\mathcal{F}_k(n)$ be the subset of $\mathcal{F}_k$ of the graphs such that the corresponding $H$-set has $n-1$ orbits, for $k>n$. For $k=n$ we let $\mathcal{F}_n(n)$ be the set of graphs for which the corresponding $H$-set is either trivial or it has $n-1$ orbits. For any orthogonal $G\times\Sigma_k$-spectrum $E$ we let $E_{h\mathcal{F}_k(n)}$ denote the corresponding homotopy orbits, the orthogonal $G$-spectrum
\[E_{h\mathcal{F}_k(n)}=E\wedge_{\Sigma_k}\overline{E}\mathcal{F}_k(n),\]
where $\overline{E}\mathcal{F}_k(n)$ is a pointed universal space for the set of subgroups $\mathcal{F}_k(n)$, as defined in \S\ref{indexedpowers}.

\begin{theorem}\label{derI}
For every pointed $G$-space $X$ and positive integer $n$, there is a natural $G$-equivalence
\[D_{nG}I(X)\simeq \Omega^{\infty G}\big(\bigvee_{k=n}^{n|G|}Map_\ast (T_k,\mathbb{S}_G\wedge X^{\wedge k})_{h\mathcal{F}_k(n)}\big)\]
where $\mathbb{S}_G$ denotes $G$-equivariant sphere spectrum.
\end{theorem}

Explicitly, the infinite loops space above is the homotopy colimit
\[\hocolim_{m\in\mathbb{N}}\Omega^{mG}\big(\bigvee_{k=n}^{n|G|}Map_\ast (T_k,S^{mG}\wedge X^{\wedge k})_{h\mathcal{F}_k(n)}\big)\]
where $\Sigma_k$ acts on $X^{\wedge k}$ by permuting the smash components, and it acts trivially on the permutation representation $S^{mG}$ of the $G$-set $mG$.

\begin{rem}\label{equivariantpart}
By the decomposition formula \ref{handy} the derivative $D_{nG}I(X)$ is the infinite loop space of a genuine $G$-spectrum whose geometric $H$-fixed-points are equivalent to
\[
\bigvee_{\substack{[H\rcirclearrowright K]\\
K/H=n-1\ \mbox{or}\ K=n
}}
\Phi^H Map_\ast (T_K,\mathbb{S}_G\wedge X^{\wedge K})_{hAut_K}.
\]
The wedge runs over the isomorphism classes of finite $H$-sets $K$ whose corresponding group homomorphism $\rho_K$ belongs to $\mathcal{F}_{k}(n)$, where $k=|K|$. The subscript $(-)_{hAut_K}$ denotes the homotopy orbits of the action of the group $Aut_K$ of equivariant automorphisms of $K$, which acts as a subgroup of $\Sigma_k$. Finally $T_K$ is the partition complex of $k$ on which $H$ acts via $\rho_K$. This $H$-space is built from the partitions of $K$ where $H$ acts by taking the image of the partitioning sets by the $H$-action.
\end{rem}

\begin{rem}
Upon the existence of a classification theorem for $nG$-homogeneous functor suggested in \ref{rem:classification}(iv), this theorem would say that the $nG$-derivative of the identity functor on pointed $G$-spaces is a Spanier-Whitehead dual of the partition complex of the $H$-sets with $n-1$ orbits, parametrized by the families $\mathcal{F}_k(n)$.  
\end{rem}

Before proving Theorem \ref{derI} we identify the terms of the splitting of $(D_{nG}I(X))^G$ given by Theorem \ref{higertomdieck}. For every normal subgroup $H\lhd G$ we let $\mathcal{Q}_{k,H}$ be the family of subgroups of $G\times\Sigma_k$ which are the graphs of $H$-invariant group homomorphisms $\rho\colon L\to \Sigma_k$ for subgroups $H\leq L\leq G$, such that $H$ is maximal among the normal subgroups of $G$ contained in $L$ on which $\rho$ is invariant. That is to say $H$ is the largest subgroup of $L$ normal in $G$ which acts trivially on the corresponding $L$-set. We write $\mathcal{Q}_{k,H}(n)$ for the subset of graphs with the orbit condition of \ref{DnGsym}, and $(-)_{h\mathcal{Q}_{k,H}(n)}=(-)\wedge_{\Sigma_k}\overline{E}\mathcal{Q}_{k,H}(n)$ for the corresponding homotopy orbits functor.

\begin{cor}[Higher Tom Dieck-splitting]\label{fixderI} For every pointed $G$-space $X$ there is a weak equivalence
\[(D_{nG}I(X))^G\simeq \prod_{H\lhd G}\prod_{k=n}^{n|G|}\Big(\Omega^{\infty G/H}Map_\ast\big(T_k,\mathbb{S}_{G/H}\wedge (X^H)^{\wedge k}\big)_{h\mathcal{Q}_{k,H}(n)}\Big)^{G/H}\]
where $\mathbb{S}_{G/H}$ is the $G/H$-sphere spectrum.
\end{cor}

\begin{proof}
%For every normal subgroup $H$ of $G$ the family $\mathcal{Q}_{n,H}$ includes in $\mathcal{F}_n$. An analysis on fixed-points shows that the induced map
%\[\bigvee_{H\lhd G}(E\mathcal{Q}_{n,H})_+\longrightarrow (E\mathcal{F}_n)_+\]
%is a $G\times\Sigma_n$-equivalence. Since $\Sigma_n$ acts freely on both sides this induces a decomposition of the fixed-points spectrum
%\[F_{G}^G\Big(Map_\ast\big(T_n,\mathbb{S}\wedge X^{\wedge n}\big)\wedge_{\Sigma_n}(E\mathcal{F}_n)_+\Big)\simeq \bigvee_{H\lhd G}F_{G}^G\Big(Map_\ast\big(T_n,\mathbb{S}\wedge X^{\wedge n}\big)\wedge_{\Sigma_n}(E\mathcal{Q}_{n,H})_+\Big)\]
%Let $\mathcal{F}_H$ be the family of subgroups of $G$ which do not contain $H$, and let $\widetilde{E\mathcal{F}}_H$ be the cofiber of $E\mathcal{F}_H\to S^0$.
%Observe that the map $(E\mathcal{Q}_{n,H})_+\to (E\mathcal{Q}_{n,H})_+\wedge\widetilde{E\mathcal{F}}_H$ induced by the inclusion $S^0\to\widetilde{E\mathcal{F}}_H$ is a $G\times \Sigma_n$-equivalence. Since the functor $F^{G}_G$ is equivalent to $F^{G/H}_{G/H}F^{H}_G$ it follows that
%\[F_{G}^G\Big(Map_\ast\big(T_n,\mathbb{S}\wedge X^{\wedge n}\big)\wedge_{\Sigma_n}(E\mathcal{Q}_{n,H})_+\Big)\simeq F^{G/H}_{G/H}\Phi^H\Big(Map_\ast\big(T_n,\mathbb{S}\wedge X^{\wedge n}\big)\wedge_{\Sigma_n}(E\mathcal{Q}_{n,H})_+\Big)\]
%Moreover $H$ acts trivially on $T_n$ and on $E\mathcal{Q}_{n,H}$, and this spectrum is equivalent to
%\[F^{G/H}_{G/H}\Big(Map_\ast\big(T_n,\mathbb{S}\wedge (X^H)^{\wedge n}\big)\wedge_{\Sigma_n}(E\mathcal{Q}_{n,H})_+\Big)\]
%concluding the proof.
We identify the terms of the splitting Theorem \ref{higertomdieck}. By Theorem \ref{derI} the $G$-fixed-points space of $D_{nG/H}I(X)$, for a normal subgroup $H$ of $G$, is equivalent to the infinite loop space
\[\Omega^{\infty}F_{G/H}^G\Big(\bigvee_{k=n}^{n|G|}Map_\ast\big(T_k,\mathbb{S}\wedge X^{\wedge k}\big)_{h\mathcal{F}_{k,H}(n)}\Big)\]
where $\mathbb{S}$ is the sphere spectrum, $F_{G/H}^G$ is the $G$-fixed-points spectrum with respect to the universe $\bigoplus_{k}\mathbb{R}[G/H]$, and $\mathcal{F}_{k,H}$ is the family of graphs of group homomorphisms $S\to \Sigma_k$ , for subgroups $S\leq G/H$. We let $G\times \Sigma_k$ act on $\overline{E}\mathcal{F}_{k,H}$ via the projection map to $G/H\times\Sigma_k$.
As a $G\times \Sigma_k$ space, $\overline{E}\mathcal{F}_{k,H}$ is a classifying space for the family of subgroups of $G\times\Sigma_k$ which are graphs of group homomorphisms $\rho\colon L\to \Sigma_k$ where $L$ is a subgroup of $G$ and $\rho(H\cap L)=1$. Let us simplify the notation and write $M_k(X)$ for the $G$-spectrum
\[M_k(X):=Map_\ast\big(T_k,\mathbb{S}\wedge X^{\wedge k}\big).\]
Since $H$ acts trivially both on $\overline{E}\mathcal{F}_{k,H}(n)$ and $T_k$, there are $G/H$-equivariant isomorphisms
\[\Phi^H (M_k(X)_{h\mathcal{F}_{k,H}(n)})\cong F_{1}^H (M_k(X)_{h \mathcal{F}_{k,H}(n)})\cong M_k(X^H)_{h \mathcal{F}_{k,H}(n)}.\]
The argument used in the proof of \ref{transTD}(\ref{one}) shows that there is an equivalence
\[F_{G/H}^G(M_k(X)_{h \mathcal{F}_{k,H}(n)})\simeq F_{G}^G( (M_k(X)\wedge \widetilde{E\mathcal{F}}_H)_{h\mathcal{F}_{k,H}(n)})\]
where $\mathcal{F}_H$ is the family of subgroups of $G$ that do not contain $H$ and $\widetilde{E\mathcal{F}}_H$ is the cofiber of the map $(E\mathcal{F}_H)_+\to S^0$, with the trivial $\Sigma_k$-action. The $G\times\Sigma_k$-space $\overline{E}\mathcal{F}_{k,H}(n)\wedge\widetilde{E\mathcal{F}}_H$ is characterized by the following property
\[\big(\overline{E}\mathcal{F}_{k,H}(n)\wedge\widetilde{E\mathcal{F}}_H\big)^\Gamma=\left\{
\begin{array}{ll}
S^0
&
\begin{tabular}{l}
\mbox{if $\Gamma$ is the graph of an $H$-invariant group homomorphism}\\
\mbox{$\rho\colon L\to\Sigma_k$, for a subgroup $H\leq L\leq G$, and $\Gamma\in \mathcal{F}_k(n)$,}
\end{tabular}
\\
\ast&\ \ \mbox{otherwise}.
\end{array}
\right.
\]
Let us denote this $G\times\Sigma_k$-space by $\mathcal{I}_{k,H}(n)$, and we write the equivalence above as
\[F_{G/H}^GM_k(X)_{h \mathcal{F}_{k,H}(n)}\simeq F_{G}^G(M_k(X)\wedge_{\Sigma_k}\mathcal{I}_{k,H}(n)).\]
The $H$-factor of the splitting \ref{higertomdieck} is thus equivalent to the infinite loop space of
\[\bigvee_{k=n}^{n|G|}\hofib\big(
F_{G}^G(M_k(X)\wedge_{\Sigma_k} \mathcal{I}_{k,H}(n))\longrightarrow\holim_{H\leq K\lhd G}F_{G}^G(M_k(X)\wedge_{\Sigma_k} \mathcal{I}_{k,K}(n))
\big).\]
By commuting the constructions appropriately this is equivalent to
\[\bigvee_{k=n}^{n|G|}F^{G}_{G}\big( M_k(X)\wedge_{\Sigma_k} \holim_{H\leq K\lhd G} \hofib(\mathcal{I}_{k,H}(n)\to \mathcal{I}_{k,K}(n))\big).\]
Let us show that the inner homotopy limit satisfies the fixed-points condition of the space $\overline{E}\mathcal{Q}_{k,H}(n)$ of \ref{fixderI}. If $\Gamma\leq G\times\Sigma_k$ is not in $\mathcal{F}_k(n)$ the $\Gamma$-fixed-points of the homotopy limit are clearly contractible. 
Let $\Gamma\in \mathcal{F}_k(n)$ be the graph of $\rho\colon L\to \Sigma_k$, for some subgroup $L$ of $G$. The $\Gamma$-fixed-points of the homotopy fiber are
\[
\hofib\big(\mathcal{I}_{k,H}(n)\to \mathcal{I}_{k,K}(n)\big)^\Gamma\simeq
\left\{\begin{array}{ll}
S^0& \begin{tabular}{ll}\mbox{if $H\leq L$, $\rho$ is $H$-invariant and}\\
\mbox{either $K\not\leq L$ or $\rho$ is not $K$-invariant}
\end{tabular}
\\
\ast & \ \! \ \mbox{otherwise}
\end{array}\right.
\]
and the maps induced by an inclusion $K\leq K'$ are either an identity or the  base-point inclusion $\ast\to S^0$. Thus the only possibility for the $\Gamma$-fixed-points of the homotopy limit over $K$ to be equivalent to $S^0$ is if all the spaces in the diagram are $S^0$. This is the case precisely when $H\leq L$ is maximal among the normal subgroups of $G$ over which $\rho$ is $H$-invariant. This is precisely the condition defining $\mathcal{Q}_{k,H}(n)$. Thus the $H$-term of the splitting Theorem \ref{higertomdieck} is the infinite loop space of the wedge of the spectra
\[F^{G}_{G}\big(M_k(X)_{h\mathcal{Q}_{k,H}(n)}\big).\]
Observe that the map $\overline{E}\mathcal{Q}_{k,H}(n)\to \overline{E}\mathcal{Q}_{k,H}(n)\wedge\widetilde{E\mathcal{F}}_H$ induced by the $G$-fixed point of $\widetilde{E\mathcal{F}}_H$ different than the base-point is a $G\times \Sigma_k$-equivalence.
Thus the term above is equivalent to 
\[F^{G}_{G}\big(M_k(X)_{h\mathcal{Q}_{k,H}(n)}\big)\simeq F_{G/H}^{G/H} \Phi^H\big(M_k(X)_{h\mathcal{Q}_{k,H}(n)}\big).\]
Since $H$ acts trivially on $\overline{E}\mathcal{Q}_{k,H}(n)$ and on $T_k$ the inner geometric fixed-points spectrum is equivalent to $M_k(X^H)_{h\mathcal{Q}_{k,H}(n)}$, and this concludes the proof.
\end{proof}

Our proof of Theorem \ref{derI} is an analysis of the iterated equivariant Snaith-splitting similar to \cite{Arone}. We let
\[Q_G=\Omega^{\infty G}\Sigma^{\infty G}\colon \Top_{\ast}^G\longrightarrow \Top_{\ast}^G\]
be stable equivariant homotopy. We recall from \cite{LMS} that the stable Snaith-splitting on the complete $G$-universe is the $G$-equivariant decomposition
\[Q_GQ_G(X)\simeq Q_G\Big(\bigvee_{k=1}^{\infty}X^{\wedge k}_{h\mathcal{F}_k}\Big)\]
where we denoted $(-)_{h\mathcal{F}_k}=(-)\wedge_{\Sigma_k}(E\mathcal{F}_k)_+$, and $\Sigma_k$ acts on $X^{\wedge k}$ by permuting the smash factors. Iterations of the Snaith-splitting give the following formula. It already appears in \cite{Carlsson}, and in \cite{Arone} for the trivial group. 

\begin{lemma}\label{itersnaith}
For every pointed $G$-space $X$ and every $p\geq 1$, there is a $G$-equivalence
\[Q_{G}^{p+1}X\simeq Q_G\Big(\bigvee_{k=1}^{\infty}(T_{k}^p\wedge X^{\wedge k})_{h\mathcal{F}_k}\Big)\]
where $T_{k}^p$ is the set of $p$-simplices of $T_k$. In particular for every $n\geq 0$ there is an equivalence
\[D_{nG}Q^{p+1}_G(X)\simeq
\Omega^{\infty G}\big(\bigvee_{k=n}^{n|G|}Map_\ast (T^{p}_k,\mathbb{S}_G\wedge X^{\wedge k})_{h\mathcal{F}_k(n)}\big)\]
where $\mathcal{F}_k(n)$ is the family of subgroups of $G\times\Sigma_k$ determined by the $H$-sets with $n-1$-orbits (see \ref{DnGsym}).
\end{lemma}

\begin{proof}
The case $p=1$ is the Snaith splitting. We show the case $p=2$. For larger $p$ the proof is by induction, using an analogous argument.

Applying the Snaith splitting twice we get a decomposition
\[Q_{G}^{3}X\simeq Q_G\Big(\bigvee_{m=1}^{\infty}\big(
\bigvee_{l=1}^{\infty}X^{\wedge l}_{h\mathcal{F}_l}
\big)_{h\mathcal{F}_m}^{\wedge m}\Big)\cong
Q_G\Big(\bigvee_{m=1}^{\infty}\big(
\bigvee_{l_1,\dots,l_m\geq 1}X^{\wedge l_1}_{h\mathcal{F}_{l_1}}\wedge\dots\wedge X^{\wedge l_m}_{h\mathcal{F}_{l_m}}
\big)_{h\mathcal{F}_m}\Big).\]
The quotient of a wedge of spaces indexed on a $G$-set $J$ is described by the formula
\[\big(\bigvee_{i\in I} X_i\big)/_G\cong \bigvee_{[i]\in I/G}X_{i}/_{G_i}\]
where $G_i$ is the stabilizer of $i$ in $G$. The action on the wedge is defined from compatible maps $g\colon X_i\to X_{gi}$. This formula for the quotient ${h\mathcal{F}_{l_m}}$ gives a $G$-equivalence
\[Q_{G}^{3}X\simeq Q_G\Big(\bigvee_{m=1}^{\infty}
\bigvee_{[\underline{l}]\in \{l_1,\dots,l_m\}/_{\Sigma_m}}\big(X^{\wedge l_1}_{h\mathcal{F}_{l_1}}\wedge\dots\wedge X^{\wedge l_m}_{h\mathcal{F}_{l_m}}
\wedge (E\mathcal{F}_m)_+ \big)/_{(\Sigma_{m})_{\underline{l}}}\Big).\]
By reordering the terms of the wedge by the sum of the $l_1,\dots,l_m$'s we can rewrite this expression as
\[Q_{G}^{3}X\simeq Q_G\Big(\bigvee_{k=1}^{\infty}
\bigvee_{m=1}^n\bigvee_{[\underline{l}]\in \left\{\substack{l_1,\dots,l_m\\
\sum l_i=k
}\right\}/_{\Sigma_m}}\big(X^{\wedge k}_{h\mathcal{F}_{l_1}\dots h\mathcal{F}_{l_m}}
\wedge (E\mathcal{F}_m)_+ \big)/_{(\Sigma_{m})_{\underline{l}}}\Big)\]
where both $\Sigma_{l_1}\times \dots \times\Sigma_{l_m}$ and $(\Sigma_{m})_{\underline{l}}$ act on $X^{\wedge k}$ via their inclusions in $\Sigma_k$ determined by $\underline{l}$. On the other hand also the $(T_{k}^2\wedge X^{\wedge k})_{h\mathcal{F}_k}$ decompose as
\[Q_G\Big(\bigvee_{k=1}^{\infty}(T_{k}^2\wedge X^{\wedge k})_{h\mathcal{F}_k}\Big)\simeq
Q_G\Big(\bigvee_{k=1}^{\infty}\bigvee_{[\lambda]\in P_{k}/_{\Sigma_k}}( X^{\wedge k}\wedge (E\mathcal{F}_k)_+)/_{(\Sigma_k)_{\lambda}}\Big)
\]
where $P_k$ is the set of unordered partitions of $k$, so that by definition $T^{2}_k=(P_k)_+$. It is therefore sufficient to show that for every $k\geq 1$ there is a $G$-equivalence
\[\bigvee_{m=1}^k\bigvee_{[\underline{l}]\in \left\{\substack{l_1,\dots,l_m\\
\sum l_i=k
}\right\}/_{\Sigma_m}}\big(X^{\wedge k}_{h\mathcal{F}_{l_1}\dots h\mathcal{F}_{l_m}}
\wedge (E\mathcal{F}_m)_+ \big)/_{(\Sigma_{m})_{\underline{l}}}\simeq \bigvee_{[\lambda]\in P_{k}/_{\Sigma_k}}( X^{\wedge k}\wedge (E\mathcal{F}_k)_+)/_{(\Sigma_k)_{\lambda}}.\]
The sets indexing the wedges are canonically isomorphic, by taking the equivalence class of a partition to the sizes of its sets. Thus we are left with finding a natural $G$-equivalence
\[\big(X^{\wedge k}\wedge (E\mathcal{F}_{l_1}\times\dots \times E\mathcal{F}_{l_m}\times E\mathcal{F}_m)_+ \big)/_{(\Sigma_{l_1}\times \dots \times \Sigma_{l_m})\oplus(\Sigma_{m})_{\underline{l}}}\simeq (X^{\wedge k}\wedge (E\mathcal{F}_k)_+)/_{(\Sigma_k)_{\underline{l}}}\]
for every $m$-tuple of positive integers $\underline{l}=(l_1,\dots,l_m)$ whose components have sum $k$. First observe that $(\Sigma_{l_1}\times \dots \times \Sigma_{l_m})\oplus(\Sigma_{m})_{\underline{l}}$ and $(\Sigma_k)_{\underline{l}}$ are equal as subgroups of $\Sigma_k$. For every subgroup $H$ of $G$ the $H$-fixed-points of the respective spaces are
\[\big(\bigvee_{\rho\colon H\to (\Sigma_k)_{\underline{l}}} (X^{\wedge k})^H\wedge (E\mathcal{F}_{l_1}\times\dots \times E\mathcal{F}_{l_m}\times E\mathcal{F}_m)^{H}_+ \big)/_{(\Sigma_k)_{\underline{l}}}\]
and
\[\big(\bigvee_{\rho\colon H\to (\Sigma_k)_{\underline{l}}} (X^{\wedge k})^H\wedge (E\mathcal{F}_k)^{H}_+ \big)/_{(\Sigma_k)_{\underline{l}}}\]
by the decomposition formula of \ref{handy}.
Both $(E\mathcal{F}_{l_1}\times\dots \times E\mathcal{F}_{l_m}\times E\mathcal{F}_m)^{H}$ and $(E\mathcal{F}_k)^{H}$ are contractible for every $\rho\colon H\to (\Sigma_k)_{\underline{l}}$, and $(\Sigma_k)_{\underline{l}}$ acts freely on the wedges. Thus both $G$-spaces are equivalent to the homotopy orbits space
\[\big(\bigvee_{\rho\colon H\to (\Sigma_k)_{\underline{l}}} (X^{\wedge k})^H\big)_{h(\Sigma_k)_{\underline{l}}}.\]

The formula for $D_{nG}Q^{k+1}_G(X)$ now follows easily from Example \ref{DnGsym}. Since $D_{nG}$ commutes with the infinite loops space functor, there are $G$-equivalences
\[D_{nG}Q^{p+1}_G(X)\simeq Q_G\big(\bigvee_{k=n}^{n|G|}(T^{p}_k\wedge X^{\wedge k})_{h\mathcal{F}_k(n)}\big)\simeq
\Omega^{\infty G}\big(\bigvee_{k=n}^{n|G|}Map_\ast (T^{p}_k,\mathbb{S}_G\wedge X^{\wedge k})_{h\mathcal{F}_k(n)}\big)\]
for every $n\geq 0$ and every $p\geq 1$. The second equivalence is the Wirthm\"{u}ller isomorphism Theorem.
\end{proof}

\begin{rem}\label{Gregmodel}
The simplicial structure of the partition complex $T_k$ induces by functoriality a cosimplicial structure on the $G$-spectra $Map_\ast (T^{\bullet}_k,\mathbb{S}_G\wedge X^{\wedge k})_{h\mathcal{F}_k(n)}$. It is unclear from the iterations of the Snaith-splitting if this cosimplicial structure corresponds to the cosimplicial structure of $D_{nG}Q^{\bullet}_G(X)$ induced by the monadic cosimplicial structure of $Q^{\bullet}_G(X)$, under the equivalence of \ref{itersnaith}. It is a highly non-trivial result of Arone and Kankaanrinta \cite{Arone} that for the trivial group at least the co-faces do agree. The constructions of \cite{Arone} are completely functorial, and they can be generalized equivariantly, showing that the thick totalizations of $D_{nG}Q^{\bullet}_G(X)$ and $\Omega^{\infty G}\big(\bigvee_{k=n}^{n|G|}Map_\ast (T^{\bullet}_k,\mathbb{S}_G\wedge X^{\wedge k})_{h\mathcal{F}_k(n)}\big)$ are equivalent.
%The proof of \cite{Arone} needs care, and the main connectivity argument about little cubes operads needs to be replaced with its equivariant counterpart from \cite{?????????}
%\textcolor{red}{The delicate spot is the connectivity range of 2.15. Analogoue for $G$-operads?}
%For the rest of the paper we will consider this issue solved.
\end{rem}

The next step of our calculation is to commute the thick totalization of the cosimplicial $G$-space $Q_{G}^{\bullet}$ with $D_{nG}$. This is not automatic since $D_{nG}$ commutes in general only with finite limits, and the thick totalization of a cosimplicial space is an infinite limit. Our argument is almost identical to the non-equivariant proof of \cite[App. A]{Arone}.
The thick totalization of the cosimplicial $G$-space $Q^{\bullet}_G(X)$ is the homotopy limit of a tower
\[Tot Q^{\bullet}_G(X)\simeq \holim_m\big(\dots\to Tot_m Q^{\bullet}_G(X)\stackrel{f_m}{\longrightarrow} Tot_{m-1} Q^{\bullet}_G(X)\stackrel{}{\to}\dots\to Q^{\bullet}_G(X)\big).\]
Let $c_m\colon \mathcal{P}(\underline{m})\to\Delta$ be the functor of \cite[7.2]{Sinha} that sends a subset $U$ of $\underline{m}$ to $m-|U|$. The homotopy fiber $L_m$ of $f_m$ is $G$-equivalent to
\[L_m\simeq \Omega^m T\hofib( U\longmapsto Q^{c_m(U)}_G(X))\]
where $T\hofib$ denotes the total homotopy fiber of the $m$-cube. This is proved in \cite[7.3]{Sinha} for pointed spaces, and it generalizes immediately to pointed $G$-spaces because fixed-points commute with homotopy limits.

\begin{lemma}\label{layerstot}
For every $m\geq 2$, there is a natural $G$-equivalence
\[T\hofib\big( U\longmapsto Q^{c_m(U)}_G(X)\big)\simeq Q_G\big(\bigvee_{k=m-1}^{\infty}(\widetilde{T}_{k}^m\wedge X^{\wedge k})_{h\mathcal{F}_k}\big)\]
where $\widetilde{T}_{k}^m$ is the set of non-degenerate $m$-simplices of the partition complex $T_{k}^m$, together with a disjoint base-point. 
\end{lemma}

\begin{proof}
By Lemma \ref{itersnaith} the homotopy fiber of a co-degeneracy map $\sigma^j\colon Q^{p+1}_G(X)\to Q^{p}_G(X)$ is equivalent to
%\[\Omega^{\infty G}\Big(\bigvee_{n=1}^{\infty}Map_\ast(T_{n}^{k+1},\mathbb{S}_{G}\wedge X^{\wedge n})_{h\mathcal{F}_n}\Big)\longrightarrow \Omega^{\infty G}\Big(\bigvee_{n=1}^{\infty}Map_\ast(T_{n}^{k},\mathbb{S}_{G}\wedge X^{\wedge n})_{h\mathcal{F}_n}\Big)\]
\[\Omega^{\infty G}\Big(\bigvee_{k=1}^{\infty}Map_\ast(cof (T_{k}^{p}\stackrel{s_j}{\longrightarrow} T_{k}^{p+1}),\mathbb{S}_{G}\wedge X^{\wedge k})_{h\mathcal{F}_k}\Big).\]
The cofiber of $s_j$ is the set of sequences of partitions $\underline{\lambda}\in P^{p-1}_k$ such that $\lambda_j\neq \lambda_{j+1}$, with a disjoint base-point.
By calculating the total homotopy fiber of the $m$-cube $Q^{c_m(U)}_G(X)$ inductively, it is easy to see that it is equivalent to
\[Q_G\big(\bigvee_{k=1}^{\infty}(\widetilde{T}_{k}^m\wedge X^{\wedge k})_{h\mathcal{F}_k}\big).\]
Finally, observe that if $k<m-1$ the set $\widetilde{T}_{k}^m$ must be a point (there are no chains of strictly increasing partitions of $k-2$ of length longer than $k-2$).
\end{proof}

\begin{lemma}\label{commuteDntot}
The canonical map $D_{nG}Tot Q^{\bullet}_G(X)\to Tot (D_{nG}Q^{\bullet}_G)(X)$ is a $G$-equivalence.
\end{lemma}

\begin{proof}
Recall that $L_m$ is the $m$-th layer of the tower computing $Tot Q^{\bullet}_G$.
By Lemma \ref{layerstot} there is an equivalence
\[D_{nG}L_m(X)\simeq\Omega^m\Omega^{\infty G}\bigvee_{k=m-1}^\infty D_{nG}\big((\mathbb{S}_G\wedge\widetilde{T}_{k}^m\wedge X^{\wedge k})_{h\mathcal{F}_k}\big)
.\]
This is contractible for $m>n|G|$, by the calculation of the derivatives of the indexed symmetric powers of \ref{DnGsym}. It follows that for every fixed $n$, the tower
\[\dots\longrightarrow D_{nG}Tot_m Q^{\bullet}_G(X)\stackrel{}{\longrightarrow} D_{nG}Tot_{m-1} Q^{\bullet}_G(X)\stackrel{}{\longrightarrow}\dots\longrightarrow D_{nG}Q^{\bullet}_G(X)\]
stabilizes, and $\holim_m D_{nG}Tot_m Q^{\bullet}_G(X)$ is equivalent to $D_{nG}Tot_q Q^{\bullet}_G(X)$ for every sufficiently large integer $q$.

We claim that also $D_{nG}\holim_m Tot_m Q^{\bullet}_G(X)$ is equivalent to $D_{nG}Tot_q Q^{\bullet}_G(X)$ for a sufficiently large $q$, which will conclude the proof. We show this by a connectivity argument. By \ref{layerstot} the connectivity of $L_m$ increases with $m$. It follows that the homotopy fiber of the restriction map 
\[R_q(X)\colon\holim_m Tot_m Q^{\bullet}_G(X)\to Tot_q Q^{\bullet}_G(X)\]
is as connected as $L_{q+1}$. We show that $D_{nG}R_q(X)$ is an equivalence for sufficiently large $q$ by estimating the connectivity of $T_{nG}R_q(X)$.
By \ref{layerstot} the $H$-fixed-points of $R_{q}(X)$ are
\[\conn R_{q}(X)^H=\min_{m\geq q}\conn (\widetilde{T}_{m}^{q+1}\wedge X^{\wedge m})_{h\mathcal{F}_m}^H-(q+1)\]
connected.
By the fixed-points formula of \cite[I.4]{Carlsson} for $(-)_{h\mathcal{F}_m}$, this connectivity range is
\begin{equation}\label{connL}
\conn R_{q}(X)^H=\min_{m\geq q}\min_{\rho: H\to\Sigma_{m}}\sum_{[i]\in m/H}X^{H_i}-(q+1)
\end{equation}
where $H_i$ and $m/H$ are respectively the stabilizer group and the quotient of the $H$-action $\rho\colon H\to \Sigma_m$. We use this connectivity range to show that the connectivity of $T_{nG}^{(l)}R_{q}$ diverges with $l$, provided $q$ is sufficiently large. This shows that $P_{nG}R_{q}$ and $P_{(n-1)G}R_{q}$ are equivalences for large $q$, and therefore so is $D_{nG}R_{q}$. This is similar to the statement that functors which agree to the order $n$ have equivalent $n$-excisive approximation, from \cite[1.6]{calcIII}. The connectivity of $T_{nG}^{(l)}R_{q}(X)$ on $H$-fixed-points is
\[\min_{K\leq H}\min_{\underline{U}\in(\mathcal{P}_0(nG_+)^{\times l})^K}\conn R_q(\Lambda^{nG,l}_{\underline{U}}(X))^K-\sum_{j=1}^l(|U_j/K|-1).\]
This expression is obtained from the formula for the connectivity of homotopy limits of \cite[A.1.2]{GdiagTop}. By substituting (\ref{connL}) we write this expression as
\[\min_{K\leq H}\min_{\underline{U}\in(\mathcal{P}_0(nG_+)^{\times l})^K}\min_{m\geq q}\min_{\rho: K\to\Sigma_{m}}\sum_{[i]\in m/K}(\Lambda^{nG,l}_{\underline{U}}(X))^{K_i}-(q+1)-\sum_{j=1}^l(|U_j/K|-1).\]
In \ref{spelloutlambda} we showed that the homotopy type of $\Lambda^{nG,l}_{\underline{U}}(X)$ is either that of a point, of an $l$-fold join by finite sets, or of an iteration of gluing of suspensions $\Sigma^G X$. It follows that the minimum over $\underline{U}$ is reached either for $\underline{U}=nG^{\times l}$ or for $\underline{U}=(nG_+)^{\times l}$. Thus the connectivity is
\[\min_{K\leq H}\min_{m\geq q}\min_{\rho}\min\big\{\sum_{[i]\in m/K}(\conn X^{K_i}+l)+1, \sum_{[i]\in m/K}(\conn X^{K_i}+l|G/K_i|)\big\}-(q+1)-ln|G/K|.\]
This is larger than
\begin{equation}\label{range}
c+\min_{K\leq H}\min_{m\geq q}\min_{\rho:K\to\Sigma_m}\ l(m/K-n|G/K|)
\end{equation}
for a certain constant $c$. One can choose $q$ sufficiently large so that $m/K>n|G/K|$ for every $m\geq q$ and every $K$-action on a set with $m$-elements. For such a $q$, the quantity (\ref{range}) diverges with $l$.
\end{proof}

\begin{proof}[Proof of \ref{derI}]
Since mapping spaces turn realizations into totalizations, there is a natural equivalence 
\[\Omega^{\infty G}\big(\bigvee_{k=n}^{n|G|}Map_\ast (T_k,\mathbb{S}_G\wedge X^{\wedge k})_{h\mathcal{F}_k(n)}\big)\simeq Tot\  \Omega^{\infty G}\big(\bigvee_{k=n}^{n|G|}Map_\ast (T^{\bullet}_k,\mathbb{S}_G\wedge X^{\wedge k})_{h\mathcal{F}_k(n)}\big).\]
By \ref{itersnaith} the right-hand $G$-space is equivalent to $Tot D_{nG} Q^{\bullet}_G(X)$, and by \ref{commuteDntot} this is equivalent to $D_{nG} Tot Q^{\bullet}_G(X)$. Since the totalization of $Q^{\bullet}_G(X)$ is equivalent to the limit of the tower of the identity functor (see \ref{Carlssonthing}), this is equivalent to $D_{nG}\holim_kP_kI$. The map $X\to \holim_kP_kI(X)$ is an equivalence on $G$-spaces which have simply-connected fixed-points, and in defining $T_{nG}$ we evaluate the functors at $G$-spaces that have this property. It follows that $D_{nG}\holim_kP_kI$ is equivalent to $D_{nG}I$ for every $n$.
\end{proof}

\appendix
\section{Appendix}

\subsection{Homotopy (co)limits and equivariant covers}
We recall that a cover of a category $I$ is a collection of subcategories $\{I_j\}_{j\in J}$ of $I$ such that for every subset $U\subset I$ the set of morphisms $\cup_{j\in U}\hom_{I_j}$ is closed under composition, and such that $I=\bigcup_{j\in J}I_j$.
 
\begin{defn}\label{def:Gcover}
Let $G$ be a finite group and let $I$ be a category with $G$-action. An equivariant cover of $I$ is a cover $\{I_j\}_{j\in J}$ of $I$ indexed on a $G$-set $J$, with the property that for every $g$ in $G$ the automorphism $g\colon I\to I$ restricts to a functor
\[g\colon I_j\to I_{gj}\]
for every $j\in J$.
\end{defn}

Let $F\in \mathscr{C}^{I}_a$ be a $G$-diagram in a $G$-model category $\mathscr{C}$, and let $\{I_j\}_{j\in J}$ be an equivariant cover of $I$. There is a cube $X$ with vertices
\[X_U=\holim_{\bigcap\limits_{u\in U}I_u}F.\]
The maps of the cube are induced by the restriction maps along the inclusions of categories $\bigcap_{j\in U}I_j\to \bigcap_{j\in V}I_j$, for $V\subset U$. The $G$-structure of $F$ combined with the fact that $\{I_j\}_{j\in J}$ is an equivariant cover gives $X$ the structure of a $J$-cube. The $G$-structure maps are the natural transformations
\[X_U=\holim_{\bigcap\limits_{u\in U}I_u}F\stackrel{(g^{-1})^{\ast}}{\longrightarrow}\holim_{\bigcap\limits_{u\in g(U)}I_u}F\circ g^{-1}\stackrel{g}{\longrightarrow}\holim_{\bigcap\limits_{u\in g(U)}I_u} F=X_{g(U)}\]
where the first map is the restriction along the functor $g^{-1}\colon \bigcap_{j\in g(U)}I_{j}=\bigcap_{j\in U}I_{gj}\to \bigcap_{j\in U}I_{j}$ and the second map is from the $G$-structure of $F$. The following is an equivariant version of \cite[1.9,1.10]{calcII}.

\begin{prop}[Covering Lemma]\label{coverings}
Let $F\in \mathscr{C}^{I}_a$ be a $G$-diagram, and let $\{I_j\}_{j\in J}$ be an equivariant cover of $I$. The $J$-cube $X\in \mathscr{C}^{\mathcal{P}(J)}_a$ defined above
is homotopy cartesian, that is the map 
\[\displaystyle\holim_{I}F\stackrel{\simeq}{\longrightarrow}\holim_{U\in \mathcal{P}_0(J)}\holim_{\bigcap\limits_{j\in U}I_j}F\] is an equivalence in $\mathscr{C}^G$. Dually, there is a cocartesian $J$-cube $Y\in \mathscr{C}^{\mathcal{P}(J)}_a$ with vertices
\[Y_U=\hocolim_{\bigcap\limits_{j\in J\backslash U}I_j}F.\] 
\end{prop}

\begin{proof}
We prove the statement about homotopy limits.
Consider the $G$-diagram $\Gamma\colon \mathcal{P}_0(J)^{op}\rightarrow Cat$ defined by the categories $\Gamma(U)=\bigcap_{u\in U}I_u$.
The $G$-structure of $\Gamma$ is induced by the $G$-action on $I$, and it is well-defined because the cover is equivariant. The op-Grothendieck construction of $\Gamma$ is the category $\mathcal{P}_0(J)\wr\Gamma$ with objects the pairs $(U\in \mathcal{P}_0(J)^{op},i\in \Gamma(U))$. The set of morphisms from $(U,i)$ to $(U',i')$ is the set of maps $i\rightarrow i'$ in $\Gamma(U)$ if $U\subset U'$, and the empty-set otherwise. There is a commutative diagram
\[\xymatrix{\displaystyle\holim_{\mathcal{P}_0(J)\wr\Gamma}\overline{F}
\ar[r]^-{\simeq}&\displaystyle\holim_{U\in \mathcal{P}_0(J)}\holim_{\bigcap\limits_{u\in U}I_u}F\\
\displaystyle\holim_{I}F\ar[ur]\ar[u]}\]
where $\overline{F}$ is the functor that sends $(U,i)$ to $F(i)$. The top horizontal map is a $G$-equivalence by the Fubini theorem for homotopy limits. A proof of this statement for the trivial group can be found in \cite[31.5]{CS}, and the equivariant version for the dual statement in terms of homotopy colimits is proved in \cite[2.26]{Gdiags}. The vertical map is restriction along the projection $p\colon\mathcal{P}_0(J)\wr\Gamma\rightarrow I$ sending $(U,i)$ to $i$. Thus $X$ is cartesian if we can show that $p$ is left $G$-cofinal, that is if the categories $p/_{i}$ are $G_i$-contractible for every object $i$ of $I$. For every object $k\in I$ and every morphism $\alpha\in I$ define
\[U_k=\{j\in J\ |\ k\in I_j\} \ \ \ \ \ \  \ \ \  \  \ \ \  U_\alpha=\{j\in J\ |\ \alpha\in I_j\}.\]
These are non-empty subsets of $J$, as the categories $I_j$ cover $I$. We define a zigzag of functors between the identity on $p/_{i}$ and the constant functor that sends every object to $(U_i,i,\id_i)\in p/_{i}$. Consider the functors $\Psi,\chi\colon p/_{i}\rightarrow p/_{i}$ that send a triple $(U,k\in\Gamma(U),(\alpha\colon k\rightarrow i)\in I)$ respectively to
\[\Psi(U,k,\alpha)=(U_k,k,\alpha) \ \ \ \ \ \ \  \ \ \  \ \ \ \chi(U,k,\alpha)=(U_\alpha,k,\alpha).\]
These are well defined because $k$ belongs to $\Gamma(U_k)$ by definition, and because $\Gamma(U_k)$ is a subcategory of $\Gamma(U_\alpha)$ ( since $U_\alpha\subset U_k$). There are inclusions $U\subset U_k\supset U_\alpha\subset U_i$ inducing a zig-zag of morphisms
\[(U,k,\alpha)\stackrel{\id_k}{\longrightarrow} (U_k,k,\alpha)\stackrel{\id_k}{\longleftarrow} (U_\alpha,k,\alpha)\stackrel{\alpha}{\longrightarrow} (U_i,i,\id_i).\]
This defines a zigzag of $G_i$-equivariant natural transformations $\id\rightarrow \Psi\leftarrow\chi\rightarrow \ast$, and a $G_i$-contraction of $p/_i$.
\end{proof}

\bibliographystyle{amsalpha}
\bibliography{bib}

\providecommand{\bysame}{\leavevmode\hbox to3em{\hrulefill}\thinspace}
\providecommand{\MR}{\relax\ifhmode\unskip\space\fi MR }
\providecommand{\MRhref}[2]{%
  \href{http://www.ams.org/mathscinet-getitem?mr=#1}{#2}
}
\providecommand{\href}[2]{#2}
\begin{thebibliography}{LMSM86}

\bibitem[AK98]{Arone}
Greg Arone and Marja Kankaanrinta, \emph{A functorial model for iterated
  {S}naith splitting with applications to calculus of functors}, Stable and
  unstable homotopy ({T}oronto, {ON}, 1996), Fields Inst. Commun., vol.~19,
  Amer. Math. Soc., Providence, RI, 1998, pp.~1--30. \MR{1622334 (99b:55009)}

\bibitem[AM99]{AM}
Greg Arone and Mark Mahowald, \emph{The {G}oodwillie tower of the identity
  functor and the unstable periodic homotopy of spheres}, Invent. Math.
  \textbf{135} (1999), no.~3, 743--788. \MR{1669268 (2000e:55012)}

\bibitem[Bar14]{Mackey}
Clark Barwick, \emph{Spectral mackey functors and equivariant algebraic
  {$K$}-theory}, arXiv: 1404.0108, 2014.

\bibitem[BCR07]{BCR}
Georg Biedermann, Boris Chorny, and Oliver R{\"o}ndigs, \emph{Calculus of
  functors and model categories}, Adv. Math. \textbf{214} (2007), no.~1,
  92--115. \MR{2348024 (2008k:55036)}

\bibitem[BH14]{BH}
Andrew Blumberg and Michael~A. Hill, \emph{Operadic multiplications in
  equivariant spectra, norms, and transfers}, arXiv:1309.1750, 2014.

\bibitem[Blu06]{Blumberg}
Andrew~J. Blumberg, \emph{Continuous functors as a model for the equivariant
  stable homotopy category}, Algebr. Geom. Topol. \textbf{6} (2006),
  2257--2295. \MR{2286026 (2008a:55006)}

\bibitem[BR14]{BR}
Georg Biedermann and Oliver R{\"o}ndigs, \emph{Calculus of functors and model
  categories, {II}}, Algebr. Geom. Topol. \textbf{14} (2014), no.~5,
  2853--2913. \MR{3276850}

\bibitem[Car91]{Carlsson}
Gunnar Carlsson, \emph{Equivariant stable homotopy and {S}ullivan's
  conjecture}, Invent. Math. \textbf{103} (1991), no.~3, 497--525. \MR{1091616
  (92g:55007)}

\bibitem[CS02]{CS}
Wojciech Chach{\'o}lski and J{\'e}r{\^o}me Scherer, \emph{Homotopy theory of
  diagrams}, Mem. Amer. Math. Soc. \textbf{155} (2002), no.~736, x+90.
  \MR{1879153 (2002k:55026)}

\bibitem[DM94]{DM}
Bj{\o}rn~Ian Dundas and Randy McCarthy, \emph{Stable {$K$}-theory and
  topological {H}ochschild homology}, Ann. of Math. (2) \textbf{140} (1994),
  no.~3, 685--701. \MR{1307900 (96e:19005a)}

\bibitem[DM15]{Gdiags}
Emanuele Dotto and Kristian Moi, \emph{Homotopy theory of {$G$}-diagrams and
  equivariant excision}, To appear in Algebraic \& Geometric Topology, 2015.

\bibitem[Dot12]{thesis}
Emanuele Dotto, \emph{Stable {R}eal {$K$}-theory and {R}eal topological
  {H}ochschild homology}, Ph.D. thesis, University of Copenhagen, 2012,
  arXiv:1212.4310.

\bibitem[Dot15a]{Gcalc}
Emanuele Dotto, \emph{Equivariant calculus of functors and
  $\mathbb{Z}/2$-analyticity of {R}eal algebraic {$K$}-theory}, Journal of the
  Institute of Mathematics of Jussieu \textbf{FirstView} (2015), 1--55.

\bibitem[Dot15b]{GdiagTop}
Emanuele Dotto, \emph{Equivariant diagrams of spaces}, arXiv:1502.05725, 2015.

\bibitem[Goo90]{calcI}
Thomas~G. Goodwillie, \emph{Calculus. {I}. {T}he first derivative of
  pseudoisotopy theory}, $K$-Theory \textbf{4} (1990), no.~1, 1--27.
  \MR{1076523 (92m:57027)}

\bibitem[Goo03]{calcIII}
\bysame, \emph{Calculus. {III}. {T}aylor series}, Geom. Topol. \textbf{7}
  (2003), 645--711 (electronic). \MR{2026544 (2005e:55015)}

\bibitem[Goo92]{calcII}
\bysame, \emph{Calculus. {II}. {A}nalytic functors}, $K$-Theory \textbf{5}
  (1991/92), no.~4, 295--332. \MR{1162445 (93i:55015)}

\bibitem[Hau14]{Markus}
Markus Hausmann, \emph{{$G$}-symmetric spectra, semistability and the
  multiplicative norm}, arXiv: 1411.2290, 2014.

\bibitem[HHR15]{HHR}
Michael~A. Hill, Michael~J. Hopkins, and Douglas~C. Ravenel, \emph{On the
  non-existence of elements of kervaire invariant one}, arXiv:0908.3724v4,
  2015.

\bibitem[Joh95]{BJ}
Brenda Johnson, \emph{The derivatives of homotopy theory}, Trans. Amer. Math.
  Soc. \textbf{347} (1995), no.~4, 1295--1321. \MR{1297532 (96b:55012)}

\bibitem[JS01]{js}
Stefan Jackowski and Jolanta S{\l}omi{\'n}ska, \emph{{$G$}-functors,
  {$G$}-posets and homotopy decompositions of {$G$}-spaces}, Fund. Math.
  \textbf{169} (2001), no.~3, 249--287. \MR{1852128 (2002h:55017)}

\bibitem[Lew00]{Lewis}
L.~Gaunce Lewis, Jr., \emph{Splitting theorems for certain equivariant
  spectra}, Mem. Amer. Math. Soc. \textbf{144} (2000), no.~686, x+89.
  \MR{1679450 (2000i:55030)}

\bibitem[LMSM86]{LMS}
L.~G. Lewis, Jr., J.~P. May, M.~Steinberger, and J.~E. McClure,
  \emph{Equivariant stable homotopy theory}, Lecture Notes in Mathematics, vol.
  1213, Springer-Verlag, Berlin, 1986, With contributions by J. E. McClure.
  \MR{866482 (88e:55002)}

\bibitem[Man04]{Mandell}
Michael~A. Mandell, \emph{Equivariant symmetric spectra}, Homotopy theory:
  relations with algebraic geometry, group cohomology, and algebraic
  {$K$}-theory, Contemp. Math., vol. 346, Amer. Math. Soc., Providence, RI,
  2004, pp.~399--452. \MR{2066508 (2005d:55019)}

\bibitem[McC97]{McCarthy}
Randy McCarthy, \emph{Relative algebraic {$K$}-theory and topological cyclic
  homology}, Acta Math. \textbf{179} (1997), no.~2, 197--222. \MR{1607555
  (99e:19006)}

\bibitem[Mer15]{Mona}
Mona Merling, \emph{Equivariant algebraic {$K$}-theory}, arXiv: 1505.07562v1,
  2015.

\bibitem[Rez13]{rezk}
Charles Rezk, \emph{A streamlined proof of {G}oodwillie's n-excisive
  approximation}, arXiv: 0812.1324v2, 2013.

\bibitem[RSS01]{RSS}
Charles Rezk, Stefan Schwede, and Brooke Shipley, \emph{Simplicial structures
  on model categories and functors}, Amer. J. Math. \textbf{123} (2001), no.~3,
  551--575. \MR{1833153 (2002d:55025)}

\bibitem[Sch13]{Schwede}
Stefan Schwede, \emph{Lectures on equivariant stable homotopy theory}, 2013.

\bibitem[Sin09]{Sinha}
Dev~P. Sinha, \emph{The topology of spaces of knots: cosimplicial models},
  Amer. J. Math. \textbf{131} (2009), no.~4, 945--980. \MR{2543919
  (2010m:57033)}

\bibitem[tD87]{dieck}
Tammo tom Dieck, \emph{Transformation groups}, de Gruyter Studies in
  Mathematics, vol.~8, Walter de Gruyter \& Co., Berlin, 1987. \MR{889050
  (89c:57048)}

\bibitem[VF04]{vilf}
Rafael Villarroel-Flores, \emph{The action by natural transformations of a
  group on a diagram of spaces}, arXiv: 0411502v1, 2004.

\bibitem[WJR13]{WJR}
Friedhelm Waldhausen, Bj{\o}rn Jahren, and John Rognes, \emph{Spaces of {PL}
  manifolds and categories of simple maps}, Annals of Mathematics Studies, vol.
  186, Princeton University Press, Princeton, NJ, 2013. \MR{3202834}

\end{thebibliography}

\end{document}